\documentclass[preprint]{imsart}
\pdfoutput=1

\RequirePackage[utf8]{inputenc}
\RequirePackage[OT1]{fontenc}
\RequirePackage{amsmath, amssymb, amsthm, mathtools}
\RequirePackage[mathscr]{euscript}
\RequirePackage{xfrac}
\RequirePackage{bbm}
\RequirePackage{amsfonts}
\RequirePackage[numbers]{natbib}
\RequirePackage[colorlinks,citecolor=blue,urlcolor=blue]{hyperref}
\RequirePackage{graphicx}
\RequirePackage{enumitem}
\RequirePackage{multirow}
\RequirePackage{hhline}
\RequirePackage{xcolor}
\RequirePackage{xr}
\RequirePackage{algorithm}
\RequirePackage{algorithmicx}
\RequirePackage{algpseudocode}
\RequirePackage{float}
\RequirePackage[margin=1.5in]{geometry}

\floatname{algorithm}{Generator}

\startlocaldefs
\numberwithin{equation}{section}
\theoremstyle{plain}
\newtheorem{theorem}{Theorem}[section]
\newtheorem{corollary}[theorem]{Corollary}
\newtheorem{proposition}[theorem]{Proposition}
\newtheorem{lemma}[theorem]{Lemma}
\newtheorem{statement}{Statement}[section]
\theoremstyle{definition}
\newtheorem{definition}{Definition}[section]
\theoremstyle{remark}

\let\mybrk\\
\newcommand{\p}[1]{\left({#1}\right)}
\newcommand{\brk}[1]{\left[{#1}\right]}
\newcommand{\set}[1]{\left\lbrace #1 \right\rbrace}

\newcommand{\prob}[2][]{ \mathbb P_{#1}\left[{#2}\right]}
\newcommand{\probout}[2][]{ \mathbb P_{#1}^*\left[{#2}\right]}
\newcommand{\ind}[1]{\mathbbm 1\left\lbrace{#1}\right\rbrace}
\newcommand{\abs}[1]{\left|{#1}\right|}
\newcommand{\de}[1]{\ \text{d}{#1}}

\newcommand{\E}[2][]{ \mathbb E_{#1}\left[{#2}\right]}

\DeclareMathOperator{\.Var}{Var}
\newcommand{\Var}[2][]{\.Var_{#1}\left[{#2}\right]}

\newcommand{\im}[1]{\text{im}\left(#1\right)}

\newcommand{\diam}[1]{\text{diam}\left({#1}\right)}
\newcommand{\st}{\text{ s.t. }}

\newcommand{\runif}[1]{\text{Unif}\left(#1\right)}
\newcommand{\rexp}[1]{\text{Exp}\left(#1\right)}
\newcommand{\iid}{\overset{\text{iid}}{\sim}}
\newcommand{\rnorm}[1]{\text{N}\left(#1\right)}

\newcommand{\rpois}[1]{\text{Pois}\left(#1\right)}

\newcommand{\reals}{\mathbb R}
\newcommand{\nats}{\mathbb N}
\newcommand{\ints}{\mathbb Z}

\newcommand{\AND}{ \ \text{and} \ }

\newcommand{\norm}[2][]{{\left|\left|{#2}\right|\right|}_{#1}}

\endlocaldefs

\begin{document}

\begin{frontmatter}
\title{Bootstrapping Persistent Betti Numbers \\ and Other Stabilizing Statistics}
\runtitle{Bootstrapping Stabilizing Statistics}

\begin{aug}
\author[A]{\fnms{Benjamin} \snm{Roycraft}
\ead[label=e1,mark]{btroycraft@ucdavis.edu}},
\author[B]{\fnms{Johannes} \snm{Krebs} 
\ead[label=e2,mark]{krebs@uni-heidelberg.de}}
\and
\author[A]{\fnms{Wolfgang} \snm{Polonik} 
\ead[label=e3,mark]{wpolonik@ucdavis.edu}}
\address[A]{Department of Statistics, University of California, Davis, One Shields Avenue, 95616, USA\\ \printead{e1,e3}}
\address[B]{Institute for Applied Mathematics, Heidelberg University, Im Neuenheimer Feld 205, 69120 Heidelberg, Germany\\\printead{e2}}
\runauthor{Roycraft et al.}
\affiliation{UC Davis and Heidelberg University}
\end{aug}

\begin{abstract}
The present contribution investigates multivariate bootstrap procedures for general stabilizing statistics, with specific application to topological data analysis. Existing limit theorems for topological statistics prove difficult to use in practice for the construction of confidence intervals, motivating the use of the bootstrap in this capacity. However, the standard nonparametric bootstrap does not directly provide for asymptotically valid confidence intervals in some situations. A smoothed bootstrap procedure, instead, is shown to give consistent estimation in these settings. The present work relates to other general results in the area of stabilizing statistics, including central limit theorems for functionals of Poisson and Binomial processes in the critical regime. Specific statistics considered include the persistent Betti numbers of \v{C}ech and Vietoris-Rips complexes over point sets in $\reals^d$, along with Euler characteristics, and the total edge length of the $k$-nearest neighbor graph. Special emphasis is made throughout to weakening the necessary conditions needed to establish bootstrap consistency. In particular, the assumption of a continuous underlying density is not required. A simulation study is provided to assess the performance of the smoothed bootstrap for finite sample sizes, and the method is further applied to the cosmic web dataset from the Sloan Digital Sky Survey (SDSS). Source code is available at \href{https://github.com/btroycraft/stabilizing_statistics_bootstrap}{github.com/btroycraft/stabilizing\_statistics\_bootstrap}.

\end{abstract}



\end{frontmatter}


\section{Introduction}

In recent years, a multitude of topological statistics have been developed to describe and analyze the structure of data, achieving notable success. These methods have seen application in astrophysics \cite{Adler2017, Pranav2019, Pranav2016, Pranav2019a}, cancer genomics \cite{Arsuaga2015, DeWoskin2010, Camara2016}, medical imaging \cite{Crawford2019}, materials science \cite{Kramar2013}, fluid dynamics \cite{Kramar2016} and chemistry \cite{Xia2014}, and other wide ranging fields.

The use of simplicial complexes to summarize the geometric and topological properties of data culminates in the techniques of persistent homology. Summary statistics based on persistent homology, persistent Betti numbers, persistence diagrams, and derivatives thereof effectively extract essential topological properties from point cloud data. A broad introduction to the methods of topological data analysis can be found in \cite{Wasserman2018, Chazal2017}.

While the use of such statistics has seen wide success, very little is currently known about the statistical properties of these topological summaries. An initial attempt at statistical analysis using persistent homology can be seen in \cite{Bubenik2007}, with the later introduction of persistence landscapes in \cite{Bubenik2015}. Likewise, central limit theorems have been developed for persistence landscapes \cite{Chazal2015a}, Betti numbers \cite{Yogeshwaran2017} and persistent Betti numbers \cite{Hiraoka2018, Krebs2019} under a variety of asymptotic settings. However, the form of these results is insufficient to provide for valid confidence intervals.

In the construction of asymptotically valid confidence intervals, subsampling and bootstrap estimation have proven successful. In \cite{Fasy2014}, various techniques are given for constructing confidence sets for persistence diagrams and derived statistics, including persistence diagrams generated from sublevel sets of the density function, as well as for the \v{C}ech and Vietoris-Rips complexes of data constrained to a manifold embedded in $\reals^d$. In \cite{Chazal2015a, Chazal2015}, bootstrap consistency is established very generally for persistence landscapes drawn from independently generated point clouds in $\reals^d$, assuming that the number of independent samples is allowed to grow.

However, even with these recent developments, the available techniques for constructing confidence sets using topological statistics remain severely limited. The bootstrap has proven one of the only effective tools, however the theoretical properties of bootstrap estimation applied to topological statistics are not well understood. For the large-sample asymptotic regime in particular, results are largely nonexistent.

The goal of this work is to provide the foundational theory for the bootstrap in this area. Here the validity of the bootstrap in the multivariate setting is established, a key step towards an eventual process-level result. However, the latter remains a significant technical hurdle. While motivated primarily by application to topological data analysis, the results presented here apply much more generally over a class of \textit{stabilizing} statistics. For an additional application, we show convergence for the bootstrap applied to the total edge length of the $k$-nearest neighbor graph.

We also analyze the large-sample asymptotic properties of the bootstrap applied to the \v{C}ech and Vietoris-Rips complexes directly, where the underlying point cloud is a sample drawn from a common distribution on $\reals^d$. In particular, we will show that the standard nonparametric bootstrap can fail to provide asymptotically valid confidence intervals directly in some cases. Via a smoothed bootstrap, however, we will construct multivariate confidence intervals for the mean persistent Betti number, which lie in bijection with the corresponding persistence diagram.

As defined in \cite{Penrose2001}, a statistic \textit{stabilizes} if the change in the function value induced by addition of new points to the underlying sample is at most locally determined. Applications of stabilization have allowed for the development of central limit theorems for several topological statistics. \cite{Yogeshwaran2017} show that Betti numbers exhibit the stabilization property, and provide a central limit theorem for Betti numbers derived from a homogenous Poisson process with unit intensity. \cite{Hiraoka2018} considers persistent Betti numbers in the homogenous Poisson process case with arbitrary intensity. Most recently \cite{Krebs2019} established multivariate central limit theorems for persistent Betti numbers with an underlying point cloud coming from either a nonhomogenous Poisson or binomial process. For the results in the present contribution, we draw significant inspiration from this most recent work.

An application of our general consistency result is made to the persistent Betti numbers of a class of distance-based simplicial complexes, including the \v{C}ech and Vietoris-Rips complexes. Throughout this work, a special focus is given towards weakening the necessary assumptions compared to previous results. Specifically, the theorems presented here apply for distributions with unbounded support, unbounded density, and possible discontinuities. We assume only a bound for the $L_p$-norm of the underlying sampling density.

For the first half of this paper, we focus on the theory of bootstrap estimation applied to stabilizing statistics. In Section~\ref{section::stabilizing_statistics} we will introduce the concept of stabilization and establish intermediate technical results in this context. We then present our general bootstrap consistency theorem.

In the second half, we introduce the main topological and geometric statistics of interest, applying the theory presented in the previous sections. In Section~\ref{section::simpicial_complex} we connect the general theory to the specific case of persistent homology and related statistics. Towards this end, we give a short introduction to simplicial complexes and persistent homology. In Section~\ref{section::topology_bootstrap}, the stabilization properties of persistent Betti numbers are analyzed, along with the Euler characteristic, for general classes of distance-based simplicial complexes. We establish bootstrap consistency in the large-sample limit for each of these statistics, as well as for the total edge length of the $k$-nearest neighbor graph. In Section~\ref{section::simulation_study} we provide several simulations demonstrating the finite-sample properties of the smoothed bootstrap applied to persistent Betti numbers. Finally, Section~\ref{section::data_analysis} illustrates the utility of the smoothed bootstrap with an application to a cosmic web dataset from the Sloan Digital Sky Survey (SDSS) \cite{Blanton2017}. Source code for the computational sections is available at \href{https://github.com/btroycraft/stabilizing_statistics_bootstrap}{github.com/btroycraft/stabilizing\_statistics\_bootstrap} \cite{Roycraft2021}.

Appendix~\ref{section::b_bounded} gives an investigation of several altered problem settings in which precise stabilization properties may be derived. The proofs for all results can be found in Appendix~\ref{appendix::proofs}. Functionals considered include the ``$B$-bounded'' persistent Betti numbers and the ``$q$-truncated'' Euler characteristic. 

\section{Stabilizing Statistics} \label{section::stabilizing_statistics}

\subsection{Central Limit Theorems for Stabilizing Statistics}

Before proving bootstrap convergence, we give a brief overview of the existing work regarding stabilizing statistics. For the precise definitions used throughout this paper, see Section~\ref{section::stabilization}.

In the seminal work of \cite{Penrose2001}, the chief objects of study are real valued functionals applied over point sets in $\reals^d$. It is here that a stabilization property was first defined, and used to show central limit theorems for certain types of geometric functionals, including the length of the $k$-nearest neighbor graph and the number of edges in the sphere of influence graph. This initial work distilled two properties key to showing central limit theorems for geometric functionals. First is the \textit{stabilization} property, and second is a moment bound. In short, we say that a functional $\psi$ \textit{stabilizes} if the cost of adding an additional point, or a set of points, to the point cloud varies only on a bounded region. Specific definitions differ by context.

In \cite{Penrose2001}, the authors distinguish between two data generating regimes. First, results are shown for a homogenous Poisson process over $\reals^d$. Alternatively, a binomial process is considered, being equivalent to a sample of fixed size from an appropriate probability distribution. Here, the functional under consideration is restricted to a bounded domain $B_n$ of volume $n$, where $n$ is allowed to increase.  In this initial work, only homogenous Poisson processes and uniform binomial sampling are considered. In \cite{Penrose2003}, a similar framework is used to establish laws of large numbers for graph-based functionals, including the number of connected components in the minimum spanning tree. Further quantitative refinements on the general central limit theorems for stabilizing statistics are shown in \cite{LachiezeRey2019}, \cite{LachiezeRey2020}, and \cite{Last2015}.

As pertains to topological statistics, an initial central limit theorem for Betti numbers (see Section~\ref{section::homology} for definitions) was shown in \cite{Yogeshwaran2017}, establishing so-called \textit{weak stabilization} for Betti numbers in the homogenous Poisson and uniform Binomial sampling settings. There an alternative set-up is being used where the domain is kept fixed, while the filtration parameter is decreasing to zero. A similar result for persistent Betti numbers is given in \cite{Hiraoka2018}.

Finally, \cite{Krebs2019} establishes multivariate central limit theorems for persistent Betti numbers under a flexible sampling setting. Here, a nonhomogeneous Poisson or binomial process is generated again over a growing domain with fixed filtration radii. 

With these central limit theorem results, the stabilization property plays a central role in understanding the asymptotic behavior for wide classes of geometric and topological functionals. Unfortunately, as a reoccurring trend, explicit forms for the asymptotic normal distributions are unavailable or computationally intractable. In this work it is shown how a smoothed bootstrap procedure allows for consistent estimation of these inaccessible limiting distributions, and thus for any subsequent inference derived therefrom.

Further, the bootstrap convergence results shown in this paper apply even more broadly, given that the necessary assumptions are much weaker than normally used to establish central limit theorems. To the best of our knowledge, it is not known whether there exist stabilizing statistics which exhibit a non-normal limit, but our convergence results apply equally for any distributional limit.

\subsection{Stabilization} \label{section::stabilization}
Here, we extend and rephrase existing definitions found in \cite{Penrose2001}, \cite{Penrose2003}, \cite{Yogeshwaran2017}, and \cite{Krebs2019} to provide a more general and consistent statistical framework. Let $\mathcal X\p{\reals^d}$ denote the space consisting of multisets drawn from $\reals^d$ with no accumulation points, with the further restriction that no point in a given multiset may be counted more than finitely often. Any locally-finite point process on $\reals^d$ can be represented as a random element of $\mathcal X\p{\reals^d}$. Let $\tilde{\mathcal X}\p{\reals^d}\subset \mathcal X\p{\reals^d}$ contain the finite multisets drawn from $\reals^d$ and $\psi\colon\tilde{\mathcal{X}}\p{\reals^d}\to\reals$ be a measurable function. Furthermore, for $S, T\in \tilde{\mathcal X}\p{\reals^d}$ define the \textit{addition cost} of $T$ to $S$ as $D\p{S; \psi, T}:=\psi\p{S\cup T}-\psi\p{S}$. When $T=\set{z}$ consists of a single point, we call $D_z\p{S; \psi} := D\p{S; \psi, \set{z}}$ an \textit{add-one cost} or the \textit{add-$z$ cost}.

Broadly, we say that $\psi$ \textit{stabilizes} if the addition cost of a given $T$ varies only on a bounded region. In the preceding literature, the terms ``strong" and ``weak" stabilization are very often used, with precise definitions changing based on circumstance. In the interest of providing more explanatory and specific terminology, we propose the following definitions.

Seen below, almost-sure and locally-determined almost-sure stabilization  (see Definitions~\ref{definition::stabilization_almost_sure} and \ref{definition::local_determined}) correspond, respectively, to Definitions 3.1 and 2.1 in \cite{Penrose2001}. Here we have generalized by accounting for possible measurability issues, however the definitions are essentially equivalent. Let $B_z\p{r}$ denote the closed Euclidean ball centered at $z\in\reals^d$ with radius $r$. For convenience, the dependence on $\psi$ and $T$ is implicit in each of the following.

\begin{definition}[Terminal Addition Cost] \label{definition::terminal_addition}
	$D^\infty\colon\mathcal X\p{\reals^d}\rightarrow \reals$ is a \textit{terminal addition cost} centered at $z\in\reals^d$ if $D^\infty\p{S}=\lim_{l\rightarrow\infty}D\p{S\cap B_z\p{l}}$ for any $S\in\mathcal X\p{\reals^d}$ such that the limit exists.
\end{definition}

For a finite multiset $S\in\tilde{\mathcal X}\p{\reals^d}$, the terminal addition cost centered at $z\in\reals^d$ is $D^\infty\p{S}=D\p{S}$, because no further changes to the addition cost may occur once $S\cap B_z\p{a}$ contains all of $S$. This does not hold for infinite multisets, motivating a separate definition. In the special case where $T=\set{z}$ is a singleton at the centerpoint, the notation $D^\infty = D_z^\infty$ may be used, and will be seen throughout the remaining sections of the paper.

\begin{definition}[Stabilization in Probability] \label{definition::stabilization_probability}
	For $\mathbf S$ a point process taking value in $\mathcal X\p{\reals^d}$, $\psi$ \textit{stabilizes} on $\mathbf S$ \textit{in probability} if there exists a center point $z\in\reals^d$ and a terminal addition cost $D^\infty$ for $\psi$ such that
	\begin{equation}
		\lim_{l\rightarrow\infty}\probout{D\p{\mathbf S\cap B_z\p{l}}\neq D^\infty\p{\mathbf S}}=0.
	\end{equation}
\end{definition}

Here $\mathbb P^*$ denotes the outer probability of a set. Stabilization is said to occur \textit{in probability} because, for any sequence of non-negative radii $\p{l_i}_{i\in\nats}$ such that $l_i\rightarrow\infty$, $D\p{\mathbf S\cap B_z\p{l_i}}\overset{p}{\rightarrow}D^\infty\p{\mathbf S}$ whenever both quantities are measurable. $D^\infty$ is unique up to a null set in this case. Stabilization in probability is difficult to show directly for many functions of interest. As such, we have the following:

\begin{definition}[Radius of Stabilization] \label{definition::radius_stability}
	$\rho\colon \mathcal X\p{\reals^d}\to \brk{0, \infty}$ is a \textit{radius of stabilization} for $\psi$ centered at $z\in\reals^d$ if, for any $S\in\mathcal X\p{\reals^d}$ and $l\in\reals$ such that $\rho\p{S}\leq l <\infty$,
	\begin{equation}
		D\p{S\cap B_z\p{l}}=D\p{S\cap B_z\p{\rho\p{S}}}.
	\end{equation}
\end{definition}

$D^\infty\p{S} := D\p{S\cap B_z\p{\rho\p{S}}}$ is a valid terminal addition cost. In the case where $\mybrk \lim_{l\rightarrow\infty}D\p{S\cap B_z\p{l}}$ does not exist, $\rho\p{S}=\infty$ necessarily, with the stabilization criterion satisfied vacuously. As with the terminal addition cost, when $T=\set{z}$ we denote $\rho=\rho_z$.

In general, for any $\psi$ there exists a unique minimal radius of stabilization, defined as the pointwise minimum over all such radii sharing the same centerpoint. This minimum exists because $\psi\p{S\cap B_{z}\p{l}}$ is piecewise constant in $0\leq l <\infty$, changing value only when a new point of $S$ is added, and because $S$ has no accumulation points.

\begin{definition}[Stabilization Almost Surely] \label{definition::stabilization_almost_sure}
	For $\mathbf S$ a point process taking value in $\mathcal X\p{\reals^d}$, $\psi$ stabilizes on $\mathbf S$ \textit{almost surely} if there exists a radius of stabilization $\rho\colon\mathcal X\p{\reals^d}\rightarrow\brk{0, \infty}$ for $\psi$ centered at $z\in\reals^d$ such that
	\begin{equation}
		\lim_{L\rightarrow\infty}\probout{\rho\p{\mathbf S}>L}=0.
	\end{equation}
\end{definition}

Mirroring our previous terminology, we say stabilization occurs \textit{almost surely} because, for any sequence of nonnegative radii $\p{l_i}_{i\in\nats}$ such that $l_i\rightarrow\infty$, $D\p{\mathbf S\cap B_z\p{l_i}}\overset{a.s.}{\rightarrow}D^\infty\p{\mathbf S} = D\p{\mathbf S\cap B_z\p{\rho\p{\mathbf S}}}$ whenever both quantities are measurable. Here we use outer probability, because a radius of stabilization may not be a measurable function, specifically considering the unique minimal radius. Almost sure stabilization implies stabilization in probability, as shown in the following.

\begin{proposition} \label{proposition::almost_sure_in_probability}
	For $\mathbf S$ a simple point process taking values in $\mathcal X\p{\reals^d}$, let $\psi$ stabilize on $\mathbf S$ almost surely. Then $\psi$ stabilizes on $\mathbf S$ in probability. 
\end{proposition}

For our proof techniques, it is often necessary to compare the stabilization properties of a function over a range of related point processes. For example, corresponding binomial, Poisson, and Cox processes can be shown to have essentially equivalent local properties, while differing globally. As defined in Definition~\ref{definition::radius_stability}, a given radius of stabilization could feasibly show completely different behavior on each process type. This motivates the following:

\begin{definition}[Locally Determined Radius of Stabilization] \label{definition::local_determined}
	The radius of stabilization $\rho$ centered at $z\in\reals^d$ is \textit{locally determined} if for any $S, T\in\mathcal X\p{\reals^d}$
	\begin{equation*}
		T\cap B_z\p{\rho\p{S}}=S\cap B_z\p{\rho\p{S}} \implies \rho\p{T}=\rho\p{S}.
	\end{equation*}
\end{definition}

With the local-determination criterion from Definition~\ref{definition::local_determined}, we can assure that stabilization must occur simultaneously on any two point processes which are locally equivalent. As in the non-locally-determined case, there exists a unique minimal locally-determined radius of stabilization:

\begin{proposition} \label{proposition::locally_determinined_minimum}
	For $\mathcal R$ the space of locally-determined radii of stabilization for $\psi$ centered at $z\in\reals^d$, let $\rho^*\colon\mathcal X\p{\reals^d}\rightarrow \brk{0, \infty}$ such that $\rho^*\p{S}=\inf_{\rho\in \mathcal R}\rho\p{S}$. Then $\rho^*$ is a locally determined radius of stabilization for $\psi$ centered at $z$.
\end{proposition}

\subsection{Technical Results} \label{section::technical_results}

In all of the following, $\mathcal P\p{\reals^d}$ denotes the set of probability distributions over $\reals^d$. $Y_1, ..., Y_n\iid G$ is a sample from $G\in\mathcal P\p{\reals^d}$ and $Y'\sim G$ an independent copy. Let $\mathbf Y_n = \set{Y_i}_{i=1}^n$ be the induced multiset. This definition may be simply denoted by $\mathbf Y_n:=\set{Y_i}_{i=1}^n\iid G$. For a measurable function $\psi\colon\tilde{\mathcal X}\p{\reals^d}\rightarrow \reals$, define the following conditions:
\begin{enumerate}[label=(E\arabic*), ref=E\arabic*]
	\item \label{condition::expectation} For a given $\mathcal C\subseteq \mathcal P\p{\reals^d}$ and some $a>2$, there exists $E_a<\infty$ such that
	\begin{equation} 
		\sup_{G\in \mathcal C}\sup_{n\in\nats}\E{\abs{\psi\p{\sqrt[d]{n}\p{\mathbf Y_n\cup \set{Y'}}}-\psi\p{\sqrt[d]{n}\mathbf Y_n}}^{a}}\leq E_a.
	\end{equation}
	\item \label{condition::local_count} For some $a>2$ and $R>0$, there exist $U_a>0$\ and $u_a>1$ satisfying the following property: For any $S\in \tilde {\mathcal X}\p{\reals^d}$ and $y\in \reals^d$,
	\begin{equation}
		\abs{\psi\p{S\cup \set{y}}-\psi\p{S}}^{a}\leq U_a\p{1+\#\set{S\cap B_{y}\p{R}}^{u_a}}.
	\end{equation}
\end{enumerate}

\eqref{condition::expectation} requires a moment bound that holds uniformly in the sample size and distribution $G\in \mathcal C$. Clearly, if \eqref{condition::expectation} is satisfied for $\mathcal C$, it is also satisfied for any subset of $\mathcal C$. In the context of the topological statistics considered in this work, \eqref{condition::expectation} is primarily useful for proof purposes, and is mainly established via \eqref{condition::local_count} (See Lemma~\ref{lemma::condition_expectation}). However, as will be seen with the case of the $k$-nearest neighbor graph, Corollary~\ref{theorem::bootstrap_knn}, there exist useful statistics which do not conform to \eqref{condition::local_count}, and the more general condition must be used. \eqref{condition::expectation} is related to the ``uniform bounded moments" condition, Definition 2.2 in \cite{Penrose2001}. Our version has been suitably generalized, the original definition considering only $a=4$. Let $\mathcal C_{p,M}\p{\reals^d}$ denote the class of probability distributions $G\in \mathcal P\p{\reals^d}$ admitting a density $g$ such that $\|g\|_p\leq M$. We have the following:

\begin{lemma} \label{lemma::condition_expectation}
	For $p>2$, let $\psi$ satisfy \eqref{condition::local_count} with $u_a\leq p-1$ for some $a>2$. Then for any $M<\infty$, $\psi$ satisfies \eqref{condition::expectation} for $\mathcal C_{p, M}\p{\reals^d}$.
\end{lemma}

For $d_{\text{TV}}$ the total variation distance between probability distributions and $B_F(\epsilon,d_{\text{TV}})$ the closed $\epsilon$-neighborhood of $F$ under $d_{\text{TV}}$, we have the following stabilization conditions:

\begin{enumerate}[label=(S\arabic*), ref=S\arabic*]
	\item \label{condition::stabilization} For a given $\mathcal C\subseteq \mathcal P\p{\reals^d}$, $F\in\mathcal C$, $b>0$, and some $\p{l_\epsilon}_{\epsilon>0}$ such that $\lim_{\epsilon\rightarrow 0}l_\epsilon\epsilon^{b}=0$, as $\epsilon \to 0,$
	\begin{equation*}
		\sup_{ G\in \mathcal C\cap B_F\p{\epsilon; d_{\text{TV}}}}\sup_{n\in\nats}\prob{D_{\sqrt[d]{n}Y'}\p{\p{\sqrt[d]{n}\mathbf Y_n}\cap B_{\sqrt[d]{n}Y'}\p{l_\epsilon}}\neq D_{\sqrt[d]{n}Y'}\p{\sqrt[d]{n}\mathbf Y_n}} \to 0 .
	\end{equation*}
	\item \label{condition::stabilization_radius} For $G\in\mathcal P\p{\reals^d}$, there exist locally-determined radii of stabilization $\p{\rho_z}_{z\in\reals^d}$ for $\psi$ satisfying
	\begin{equation}
		\lim_{L\to \infty}\sup_{n\in\nats}\probout{\rho_{\sqrt[d]{n}Y'}\p{\sqrt[d]{n}\mathbf Y_n}>L}=0.
	\end{equation}
\end{enumerate}

\eqref{condition::stabilization} and \eqref{condition::stabilization_radius} can be summarized as uniform stabilization conditions, either in probability or almost surely. \eqref{condition::stabilization} as stated is a technical condition mainly serving to weaken the necessary conditions providing for bootstrap consistency. As such, we have the following lemma linking \eqref{condition::stabilization} and \eqref{condition::stabilization_radius}.

\begin{lemma} \label{lemma::condition_stabilization}
	Let $\psi$ satisfy \eqref{condition::stabilization_radius} for $F\in C_{p, M}\p{\reals^d}$. Then $\psi$ satisfies \eqref{condition::stabilization} for $C_{p, M}\p{\reals^d}$, $F$, $b=\p{p-2}/\p{d\p{p-1}}$, and any $\p{l_\epsilon}_{\epsilon>0}$ such that $\lim_{\epsilon\rightarrow 0}l_\epsilon\epsilon^{\p{p-2}/\p{d\p{p-1}}}=0$ and $\lim_{\epsilon\rightarrow0}l_\epsilon=\infty$.
\end{lemma}

We can often greatly simplify the addition costs and radii of stabilization required in \eqref{condition::stabilization} and \eqref{condition::stabilization_radius}. For example, given a translation-invariant function $\psi$ and any $D_0$, $\rho_0$ for $\psi$ centered at $0$, corresponding quantities can be constructed for any other center point. For $z\in\reals^d$, $D_z\colon\mathcal X\p{\reals^d}\rightarrow \reals$ where $D_z\p{S}=D_0\p{S-z}$ is an add-$z$ cost for $\psi$ centered at $z$. Likewise $\rho_z\colon\mathcal X\p{\reals^d}\rightarrow \brk{0,\infty}$ where $\rho_z\p{S}=\rho_0\p{S-z}$ is a radius of stabilization for $\psi$ centered at $z$. In the following, $\mathbf P_\lambda$ denotes a homogeneous Poisson process on $\reals^d$ with intensity $\lambda$.

\begin{lemma} \label{lemma::stabilization_poisson}
	Let $F\in C_{p, M}$ with $p>2$ and $M<\infty$. Let $\rho_0$ be a locally-determined radius of stabilization for $\psi$ centered at $0$. Suppose that for any given $a,b \in \p{0, \infty}$, and $\delta>0$, there exists an $L_{a, b, \delta}<\infty$ and a measurable set $A_{a, b, \delta}$ with $\rho_0^{-1}\p{\left(L_{a, b, \delta}, \infty\right]}\subseteq A_{a, b, \delta}$ such that
	\begin{equation}
		\sup_{\lambda\in\brk{a, b}}\probout{\rho_0\p{\mathbf P_{\lambda}}>L_{a,b,\delta}} \leq \sup_{\lambda\in\brk{a, b}}\prob{\mathbf P_{\lambda}\in A_{a, b, \delta}}\leq \delta.
	\end{equation}
	Then for any $\delta>0$ there exists an $n_\delta<\infty$ and $L_\delta<\infty$ such that
	\begin{equation}
		\sup\limits_{n\geq n_\delta}\probout{\rho_0\p{\mathbf X_n-X'}>L_\delta}\leq \delta.
	\end{equation}
\end{lemma}

Lemma~\ref{lemma::stabilization_poisson} provides a convenient tool for ``de-Poissonizing" a locally-determined radius of stabilization. Often it is easier to show stabilization properties for a homogeneous Poisson process than for a binomial process directly. Lemma~\ref{lemma::stabilization_poisson} allows for the extension of homogeneous Poisson results to the binomial setting, as is required for Lemma~\ref{lemma::stabilization_pbn} and Corollary~\ref{theorem::bootstrap_knn}. Note that the conclusion is not the same as the statement of \eqref{condition::stabilization}, only applying for $n\geq n_\delta$. Some extra effort is required for the conclusion to hold for all $n\in\nats$, depending on the specifics of the function $\psi$ considered. We come to the following important proposition, the main supporting result for our general bootstrap consistency theorem, Theorem~\ref{theorem::bootstrap_general}.

\begin{proposition} \label{proposition::w2_general}
	For $p>2$ and $M<\infty$, let $\psi$ satisfy \eqref{condition::expectation} and \eqref{condition::stabilization} for $\mathcal C_{p, M}\p{\reals^d}$, $F\in \mathcal C_{p, M}\p{\reals^d}$, and some $a>2$. Then for any $G\in\mathcal C_{p, M}\p{\reals^d}\cap B_F\p{\epsilon, d_{\text{TV}}}$, there exist iid coupled random variables $\p{\p{X_i, Y_i}}_{i\in\nats}$ such that $\mathbf X_n=\set{X_i}_{i=1}^n\iid F$, $\mathbf Y_n=\set{Y_i}_{i=1}^n\iid G$, and
	\begin{equation}
		\sup_{n\in \nats} \Var{\frac{1}{\sqrt{n}}\p{\psi\p{\sqrt[d]{n}\mathbf X_n}-\psi\p{\sqrt[d]{n}\mathbf Y_n}}}\leq \gamma_{\epsilon}.
	\end{equation}
	The value $\gamma_\epsilon$ does not depend on $G$ and satisfies $\lim_{\epsilon\rightarrow 0}\gamma_\epsilon=0$.
\end{proposition}

For any two distributions $\mathcal L_1$ and $\mathcal L_2$ over $\reals$, we may define the 2-Wasserstein distance between $\mathcal L_1$ and $\mathcal L_2$ as
\begin{equation}
	W_2\p{\mathcal L_1, \mathcal L_2}:=\sqrt{\inf_{U\sim \mathcal L_1, V\sim \mathcal L_2}\E{\p{U-V}^2}}
\end{equation}
where it is assumed that $U$ and $V$ follow a joint distribution with marginals $\mathcal L_1$ and $\mathcal L_2$. For $\mathcal L$ denoting the law or distribution of a random variable, the variance given in the conclusion of Proposition~\ref{proposition::w2_general} bounds above
\begin{align}
	&W_2^2\Biggl( \mathcal L\set{\frac{1}{\sqrt{n}}\p{\psi\p{\sqrt[d]{n}\mathbf X_n}-\E{\psi\p{\sqrt[d]{n}\mathbf X_n}}}},\\
	&\qquad\qquad\qquad\qquad \mathcal L\set{\frac{1}{\sqrt{n}}\p{\psi\p{\sqrt[d]{n}\mathbf Y_n}-\E{\psi\p{\sqrt[d]{n}\mathbf Y_n}}}} \Biggl) \nonumber.
\end{align}

Consequently, Proposition~\ref{proposition::w2_general} shows that this $W_2$ distance can be made arbitrarily small uniformly over a neighborhood of distributions around $F$. An appropriately smoothed empirical distribution falls within such a small neighborhood with high probability, given sufficiently large sample sizes.

Furthermore, it can be seen that Proposition~\ref{proposition::w2_general} extends directly to finite sums. Given any $\p{A_i}_{i=1}^k$ and $\p{B_i}_{i=1}^k$, we have that $\Var{\sum_{i=1}^kA_i-\sum_{i=1}^kB_i}\leq k\sum_{i=1}^k\Var{A_i-B_i}$. Thus, if the conclusion of Proposition~\ref{proposition::w2_general} holds for any finite set of functions, $\p{\psi_i}_{i=1}^k$, it also holds for $\sum_{i=1}^k\psi_i$, with rate depending on the worst case $\psi_i$.

It should be noted that \eqref{condition::stabilization} is slightly stronger than necessary to establish Proposition~\ref{proposition::w2_general}. As stated, $D_{\sqrt[d]{n}Y'}\p{\p{\sqrt[d]{n}\mathbf Y_n}\cap B_{\sqrt[d]{n}Y'}\p{l_\epsilon}}$ itself is compared to the terminal add-one cost $D_{\sqrt[d]{n}Y'}\p{\sqrt[d]{n}\mathbf Y_n}$. As could be useful for some statistics, it is only required that an appropriate bound displays the desired stabilization property, see the provided proof for details.

\subsection{Smoothed Bootstrap} \label{section::bootstrap}

The bootstrap is an estimation technique used to construct approximate confidence intervals for a given population parameter. In cases where asymptotic approximations for the sampling distribution of a statistic are inconvenient or unavailable, bootstrap estimation provides a general tool for constructing approximate confidence intervals. Bootstrap estimation is well-studied in the statistical literature, an introduction being provided in \cite{Politis1999}. In this section, we will show consistency for a smoothed bootstrap in estimating the limiting distribution of a standardized stabilizing statistic, $\psi$, in the multivariate setting. We describe the general procedure below:

Let $\mathbf X_n=\set{X_i}_{i=1}^n\iid F$. We estimate the sampling distribution of
\begin{equation}
		\frac{1}{\sqrt{n}} \p{\psi\p{\sqrt[d]{n}\mathbf X_n}-\E{\psi\p{\sqrt[d]{n}\mathbf X_n}}}
\end{equation}
using a plug-in estimator $\hat F_n$ for the underlying data distribution $F$. In the standard nonparametric bootstrap, we estimate $F$ by the empirical distribution, giving probability to each unique value of $\p{X_i}_{i=1}^n$, proportional to the number of repetitions within $\mathbf X_n$. We have the bootstrap statistic
\begin{equation}
	\frac{1}{\sqrt{m}} \p{\psi\p{\sqrt[d]{m}\mathbf X_{m}^*}-\E{\psi\p{\sqrt[d]{{m}}\mathbf X_{m}^*}\big| \mathbf X_n}},
\end{equation}
where $\mathbf X_{m}^*=\set{X_i^*}_{i=1}^{m}\iid \hat F_n|\mathbf X_n$, conditional on $\mathbf X_n$. The sampling distribution of the bootstrap version provides an estimate for the distribution of the original statistic, which in the ideal case converges to the truth in the large-sample limit. Confidence intervals for $\E{\psi\p{\sqrt[d]{n}\mathbf X_n}}$ are then constructed from the bootstrap distribution and $\psi\p{\sqrt[d]{n}\mathbf X_n}$.

However, as will be seen in Section~\ref{section::nonparametric_bootstrap}, for some classes of topological statistics the standard bootstrap may not directly replicate the correct sampling distribution asymptotically. Consequently, we instead estimate $F$ by a smoothed distribution approximation. Such a smoothed bootstrap procedure can be shown to provide consistent estimation, even when the standard nonparametric bootstrap may fail.

For the smoothed bootstrap sampling procedure outlined here, we require that $F$ has a density $f$. Let $\hat f_n$ be an estimator for the true density with corresponding distribution $\hat F_n$, each a function of the sample $\mathbf X_n$. Conditional on $\mathbf X_n$, we draw bootstrap samples $\mathbf X_m^*$ independently from $\hat F_n|\mathbf X_n$. A particular choice of $\hat f_n$ is given via kernel density estimation. For a kernel function $Q$ and bandwidth $h>0$, the kernel density estimator of $f(x)$ based on the sample $\p{X_i}_{i=1}^n$ is $\hat f_{n, h}\p{x} := 1/\p{nh^d}\sum_{i=1}^nQ\p{\p{x-X_i}/h}$.

In practice, when $Q$ corresponds to a probability density, the kernel density estimator allows for convenient sampling, as is required for implementation. Generating a sample from $\hat f_{n, h}$ is equivalent to first drawing from the empirical distribution on $\mathbf X_n$, then adding independent noise following the distribution defined by $Q$, scaled by the bandwidth $h$. Other density estimators, including those using higher-order kernels, may not facilitate efficient sampling. However, the theory established here supports the use of any density estimator which meets the required convergence criteria, computational factors aside. More complicated data-dependent estimators are also possible, falling under a similar sampling framework. See Sections~\ref{section::simulation_study} and \ref{section::data_analysis} for specifics on density estimation as pertains to this work from a practical perspective.

We now present our main result. The following theorem establishes consistency for the smoothed bootstrap in the multivariate setting. We give the result for a vector of stabilizing statistics. In the context of the topological statistics introduced in Section~\ref{section::simpicial_complex}, this can be the persistent Betti numbers or Euler characteristic evaluated at different filtration parameters.

\begin{theorem} \label{theorem::bootstrap_general}
	Let $F\in \mathcal P\p{\reals^d}$ with density $f$ such that $\|f\|_p<\infty$ for some $p>2$. Furthermore, let $F$ and $\hat f_n$ be such that $\| \hat f_n-f \|_1 \to 0$ and $\|\hat f_n-f \|_p \to 0$ in probability (resp. a.s.). Suppose $\vec{\psi}\colon \tilde{\mathcal X}\p{\reals^d}\rightarrow\reals^k$ has component functions $\psi_j\colon \tilde{\mathcal X}\p{\reals^d}\rightarrow\reals$, $1\leq j\leq k$ satisfying \eqref{condition::expectation} and \eqref{condition::stabilization} for $\mathcal C_{p, M}\p{\reals^d}$, $M>\|f\|_p$, $F$, and $b=\p{p-2}/\p{d\p{p-1}}$. Then for a sample $\mathbf X_n=\set{X_i}_{i=1}^n\iid F$, $\p{m_n}_{n\in\nats}$ such that $\lim_{n\rightarrow \infty }m_n=\infty$, a bootstrap sample $\mathbf X_{m_n}^*=\set{X_i^*}_{i=1}^{m_n}\iid \hat F_n|\mathbf X_n$, and a multivariate distribution $G$,
	\begin{equation*}
		\frac{1}{\sqrt{n}}\p{\vec{\psi}\p{\sqrt[d]{n}\mathbf X_n}-\E{\vec{\psi}\p{\sqrt[d]{n}\mathbf X_n}}}\overset{d}{\rightarrow}G
	\end{equation*}
	\centering{if and only if}
	\begin{equation*}
		\frac{1}{\sqrt{{m_n}}}\p{\vec{\psi}\p{\sqrt[d]{{m_n}}\mathbf X_{m_n}^*}-\E{\vec{\psi}\p{\sqrt[d]{{m_n}}\mathbf X_{m_n}^*}\big| \mathbf X_n}}\overset{d}{\rightarrow}G \text{ in probability (resp. a.s.)}.
	\end{equation*}
\end{theorem}

Theorem~\ref{theorem::bootstrap_general} establishes the asymptotic validity of bootstrap estimation for a range of stabilizing statistics under fairly mild conditions on the underlying density. However, it should be noted that further restrictions on the density and density estimate may be required to satisfy \eqref{condition::expectation} and \eqref{condition::stabilization}, see Corollary~\ref{theorem::bootstrap_knn} for example. The conditions under which $\| \hat f_{n, h_n}-f \|_1 \to 0$ in probability or a.s. can be found in \cite{Devroye1979}. Proposition~\ref{theorem::lp_norm} considers the convergence of $\| \hat f_{n, h_n}-f \|_p$, either in probability or almost surely. This result is outside the main contribution of this paper, but is interesting in its own right. Notably, no conditions are placed on the density $f$ except $\|f\|_p < \infty$.

As a point of caution, it is known that kernel density estimators suffer from a curse of dimensionality. The convergence properties of the density estimator $\hat f_n$ appear implicitly within the necessary assumptions for Theorem~\ref{theorem::bootstrap_general}. In particular, diminishing performance can be expected in higher dimensions, as shown by the provided simulations of Section~\ref{section::simulation_study}.

The above result holds for any choice of $m_n$ such that $\lim_{n\rightarrow \infty }m_n=\infty$, and is stated as such for the sake of generality. In practical application, $m_n=n$ is standard, and will be used throughout the simulation and data analysis sections of this paper. However, given that the computational complexity of $\psi$ often grows quickly with $n$, using a smaller $m_n$ could prove more feasible from a computational perspective.

Strictly speaking, convergence to a limiting distribution is not required for the bootstrap to provide asymptotically valid confidence intervals. Proposition~\ref{proposition::w2_general} gives that, with high probability, the smoothed bootstrap and true sampling distributions become close in $2$-Wasserstein distance. Provided that the cumulative distribution function $F_{\vec\psi_n}$ of $(\vec{\psi}\p{\sqrt[d]{n}\mathbf X_n}-\mathbb{E}[\vec{\psi}\p{\sqrt[d]{n}\mathbf X_n}])/\sqrt{n}$ has the property
\begin{equation}
	\lim_{\delta\rightarrow 0}\limsup_{n\rightarrow\infty }\sup_{x\in\reals^d}\abs{F_{\vec\psi_n}\p{x+\delta}-F_{\vec\psi_n}\p{x}}\to 0,
\end{equation}
it can be shown that confidence intervals constructed from the bootstrap statistic still achieve the stated confidence level with high probability, given a sufficiently large sample. Convergence to a continuous limiting CDF is just one way of satisfying this condition. However, this extension is unavailable for the topological statistics considered here, as the behavior of the finite sample statistics is currently very poorly understood.

In the later sections, we will show that the necessary moment and stabilization conditions for Theorem~\ref{theorem::bootstrap_general} are satisfied for several specific statistics of interest, chiefly the Euler characteristic and persistent Betti numbers for a class of simplicial complexes.

\section{Simplicial Complexes and Persistence Homology} \label{section::simpicial_complex}
\subsection{Simplicial Complexes} \label{section::simplicial_complexes}

Let $\mathcal K = \{K^r\}_{r\in\reals}$ be a filtration of simplicial complexes, with $ K^r\subseteq K^t$ for $r<t$. Each complex is a collection of \textit{simplices}, subsets of the vertex multiset, $V$. Here any repeated vertices are considered distinct. For a collection of simplices $K$ to be a simplicial complex, for any two simplices $S\subset V$ and $T\subset S$, $S\in K$ only if $T\in K$. Here a simplex is only included along with all of its subsets. For a given simplicial complex $K$, $K_q$ denotes the subset of $K$ consisting of all $q$-simplices. $q$-simplices are those simplices consisting of $q+1$ vertices. Each $q$-simplex is said to have dimension $q$. A graph or network refers to a simplicial complex consisting of only $1$-simplices (edges) and $0$-simplices (vertices).

We will be looking at simplicial complexes constructed over point clouds in $\reals^d$. The two major examples are the \v{C}ech and Vietoris-Rips complexes:
\begin{align}
	K_\text{C}^r\p{S} & =\set{\sigma\subseteq S\colon \exists z\in\reals^d \st \|z-x\|\leq r \ \forall x\in\sigma} \\
	K_\text{VR}^r\p{S} & =\set{\sigma\subseteq S\colon \|x-y\|\leq 2r \ \forall x, y\in\sigma}.
\end{align}

Each of these complexes summarizes the geometric and topological properties within a given point cloud. The Vietoris-Rips complex can be considered a ``completion" of the \v{C}ech complex, in so much that the Vietoris-Rips complex is the largest simplicial complex with the same edge set as the \v{C}ech complex. While the primary motivation for the results given here is application to the \v{C}ech and Vietoris-Rips complexes, our main results apply for a range of possible complexes. For example, for computational reasons it is often convenient to limit the number of simplices present within the final complex. As such, we have two approximations, the alpha complex and its completion
\begin{align*}
	K_\alpha^r\p{S} & =\set{\sigma\subseteq S\colon \exists z\in\reals^d \st \|z-x\|\leq r \AND \|z-x\|\leq \|z-y\| \ \forall x\in\sigma \ \forall y\in S} \\
	K_{\alpha^*}^r\p{S} & =\set{\sigma\subseteq S\colon \set{x, y}\in K_{\alpha}^r\p{S} \ \forall x, y\in\sigma}.
\end{align*}

These complexes avoid adding simplices between disparate points, controlling the total size of the complex. It has been shown that the alpha and \v{C}ech complexes are both homotopy equivalent to a union of closed balls around the underlying point set, thus sharing equivalent homology groups. However, for the completion, denoted here as the alpha* complex, there is no such relationship. The alpha complex is a subcomplex of the \v{C}ech complex as well as the Delaunay complex
\begin{equation}
	K_\text{D}\p{S} = \set{\sigma\subseteq S\colon \exists z\in\reals^d \st \|z-x\|\leq \|z-y\| \ \forall x\in\sigma \ \forall y\in S}.
\end{equation}

\subsection{Persistent Homology} \label{section::homology}

Now, of chief interest are the topological properties for a given simplicial complex. Both the \v{C}ech and Vietoris-Rips complexes reflect the structure present within an underlying point cloud. As such the topology of each provides an effective summary statistic for describing the structural properties of a dataset in $\reals^d$. We provide below a short introduction to homology and persistence homology as used in topological data analysis.

Define $C\p{K}$ to be the free abelian group generated by the simplices in $K$. Elements of $C\p{K}$ are sums of the form $\sum_{i\in I}a_i\sigma_i$, where $\sigma_i\in K$ for $a_i$ an appropriate group element. If we further allow the coefficients to come from a field, then $C\p{K}$ is a vector space. For the purposes of this paper, coefficients are drawn from the two-element field $\mathbb F_2=\set{0, 1}$. $C\p{K}$ is equipped with a linear boundary operator $\partial\colon C\p{K} \rightarrow C\p{K}$ where $\partial\p {\{x_1, ..., x_{q+1}\}}=\sum_{i=1}^q\p{-1}^{i}\{x_1, ..., x_{i-1}, x_{i+1}, ..., x_{q+1}\}$. As a fundamental property, $\partial\circ\partial=0$. With coefficients in $\mathbb F_2$, the boundary of a simplex reduces to the sum of all its faces.  $C_q\p{K}=C\p{K_q}$ is the subspace spanned by the $q$-simplices of $K$, with the image of $C_q\p{K}$ under $\partial$ lying in $C_{q-1}\p{K}$. $\partial_q\colon C_q\p{K}\rightarrow C_{q-1}\p{K}$ denotes the restriction of $\partial$ to $C_q\p{K}$.

We now construct the homology groups of $K$. Let $Z\p{K}=\ker\p{\partial}$ be the subspace of $C\p{K}$ containing the cycles, those elements whose boundary under $\partial$ is $0$. $Z_q\p{K}=Z\p{K_q}=\ker\p{\partial_q}$ is the restriction of $Z\p{K}$ to dimension $q$. Let $B\p{K}=\im{\partial}$ denote the subspace of boundaries in $C\p{K}$. $B_q\p{K}=B\p{K_q}=\im{\partial_{q+1}}$ is the subspace consisting of the boundaries of elements in $C_{q+1}\p{K}$, lying in $C_q\p{K}$.

The \textit{homology groups} are given by $H_q\p K := Z_{q}\p{K}/B_{q}\p{K}$, the cycles $Z_q$ in dimension $q$ modulo the boundaries $B_q$. In words, the elements of the homology groups represent ``holes" within the simplicial complex, shown by closed loops whose interior is not filled by other elements in the complex. These homology groups provide a topological summary of the structure in the simplicial complex $K$. As stated previously, because we assume field coefficients for $C\p{K}$, each homology group is also a vector space. The \textit{Betti numbers} of the complex represent the degree or dimension of each homology space. We denote the $q$-th Betti number of $K$ by $\beta_{q}\p{K}=\dim\p{Z_q\p{K}/B_q\p{K}}=\dim\p{Z_q\p{K}}-\dim\p{B_q\p{K}}$. Moving forward, Betti numbers and their like will be of primary interest.

Homology provides a topological invariant constructed from a single simplicial complex. For a filtration of nested simplicial complexes, \textit{persistent homology} provides more detail. Given a filtration $\mathcal K=\{K^{r}\}_{r\in\reals}$, the homology groups for each complex, $H_q\p{K^r}$, are defined. However, due to the nested structure of the filtration, simplices are shared across complexes, and thus there exists a natural inclusion map between homology spaces. Cycles in $Z_q\p{K^r}$ are also cycles in $Z_q\p{K^t}$ if $r<t$. The boundary spaces behave similarly. For a given equivalence class $x+B_q\p{K^r}\in H_q\p {K^r}$, $x+B_q\p{K^r}\rightarrow x+B_q\p{K^t}$ specifies the inclusion map from $H_q\p{K^r}$ to $H_q\p{K^t}$.
	
If a given element $\tilde{x}\in H_q\p{K^r}$ maps to $\tilde{y}\in H_q\p{K^t}$ upon inclusion, with $\tilde{y}\neq B_q\p{K^t}$, we say that $\tilde{x}$ represents a persistent cycle across the filtration. Essentially the same underlying element is reflected in the homology groups over a range of simplicial complexes. The collection of homology groups and inclusion maps form a \textit{persistence module}. A wide body of work exists on the properties of these persistence modules, see \cite{Zomorodian2005} for an introduction. For any cycle feature in the filtration, there is a well defined death time, being the smallest parameter level for which the given element lies in the kernel. The Betti numbers of a filtration form a function in the filtration parameter, $r$. We use the notation $\beta_q^r\p{\mathcal K}:=\beta_q\p{K^r}$. The Betti numbers in this context count the number of persistent features extant at $r$.

It is a fundamental theorem of persistent homology that a sufficiently well-behaved persistence module can be represented by a \textit{persistence diagram}. A diagram $\mathcal D\p{\mathcal K}$ is a multiset in $\reals^2\times \ints$ of points $(b, d, q)$. Each point represents a single persistent feature in the module. $b$ denotes the birth time of the feature, being the smallest parameter level for which that feature is represented in the homology groups. Likewise $d$ gives the death time, and $q$ the dimension of the feature. The collection of persistent features represented by the diagram are a basis for the corresponding persistence module.

The persistence diagram is a simple summary statistic which condenses the complex topological information present within a filtration. An example of a persistence diagram is shown in Figure~\ref{figure::bias_plots}.

\subsection{Persistent Betti Numbers} 

We arrive at the main focus of this section. For $r \le s,$ define the \textit{persistent homology groups} of a filtration $\mathcal K=\set{K^r}_{r\in\reals}$ as 
\begin{equation}
		H_q^{r,s}\p{\mathcal K}\coloneqq Z_q\p{K^r}/\p{B_q\p{K^s}\cap Z_q\p{K^r}}.
\end{equation}
Nonzero elements in this group represent features born at or before time $r$ which persist until at least time $s$. The dimension of these spaces gives the \textit{persistent Betti numbers}
\begin{align}
	\beta_q^{r,s}\p{K} &\coloneqq \dim\p{Z_q\p{K^r}/B_q\p{K^s}\cap Z_q\p{K^r}} \\
	& =\dim\p{Z_q\p{K^r}}-\dim\p{B_q\p{K^s}\cap Z_q\p{K^r}}.
\end{align}
Persistent Betti numbers are in one-to-one correspondence with the respective persistence diagram. Here $\beta_q^{r,s}\p{\mathcal K}$ counts the number of points in $\mathcal D\p{\mathcal K}$ of feature dimension $q$ falling within $\left(-\infty, r\right]\times\left(s, \infty\right]$. When $s=r$, we recover the regular Betti numbers, $\beta_q^{r,r}\p{\mathcal K}=\beta_q\p{K^r}$. An important result for persistent Betti numbers is given in the following lemma.

\begin{lemma}[Geometric Lemma][Lemma~2.11 in \cite{Hiraoka2018}] \label{lemma::geometric}
	Let $\mathcal J=\set{J^r}_{r\in\reals}$ and $\mathcal K=\set{K^r}_{r\in\reals}$ be filtrations of simplicial complexes with with $J^r\subseteq K^r$ for all $r\in\reals$. Then
	\begin{align}
		\abs{\beta_q^{r,s}\p{\mathcal K} - \beta_q^{r,s}\p{\mathcal J}} \leq \ & \max\set{\#\set{K_q^r\setminus J_q^r}, \#\set{K_{q+1}^s\setminus J_{q+1}^s}}\\
		\leq \ & \#\set{K_q^r\setminus J_q^r}+\#\set{K_{q+1}^s\setminus J_{q+1}^s}.
	\end{align}
\end{lemma}

The Geometric Lemma~\ref{lemma::geometric} relates the change in persistent Betti numbers between two filtrations to the additional simplices gained moving between them. As a brief explanation of the lemma, simplices can be divided into two classes, positive and negative. For two simplicial complexes $J\subset K$, if we imagine adding the additional $q$-simplices in $K$ to $J$ one by one, a positive $q$-simplex will increase the dimension of $Z_q$ by one, and a negative $q$-simplex will increase the dimension of $B_{q-1}$ by one. Either change can affect the persistent Betti numbers. This dichotomy is a basic result from persistent homology, see \cite{Boissonnat2018}. The bound given in the Geometric Lemma describes a worst case, when all $q$-simplices at time $r$ are positive or all $(q+1)$-simplices at time $s$ are negative. The Geometric Lemma will be critical moving forward, as it allows us to control the change in persistent Betti numbers by counting appropriate simplices.

\subsection{Euler Characteristic}

For a given simplicial complex $K$, the \textit{Euler characteristic} is defined as
\begin{equation}
	\chi\p{K}:=\sum\limits_{k=0}^\infty\p{-1}^k\#\set{K_k}.
\end{equation}

Provided there is an $m\in\nats$ such that the Betti numbers $\beta_q\p{K}$ are $0$ for all $q> m$ (as in \eqref{condition::complex_higher} holds), it can be shown that the Euler characteristic has the following identity with the Betti numbers:
\begin{equation}
	\chi\p{K}=\sum\limits_{k=0}^\infty\p{-1}^k \beta_k\p{K}.
\end{equation}

This relationship with the Betti numbers makes the Euler characteristic an important topological invariant in its own right. Applications of the Euler characteristic and derivatives may be found in \cite{Pranav2019, Pranav2019a, Turner2014}.

\subsection{$k$-Nearest Neighbor Graph}

The $k$-nearest neighbor graph $\mathcal K_{\text{NN}, k}$ of a vertex set $S$ connects each point $x \in S$ with the $k$ closest vertices to $x$ within $\mathbf S\setminus x$. This graph may either be directed or undirected. $\mathcal K_{\text{NN}, k}$ is commonly used to analyze the clustering structure of a point cloud. Let the total length of the edges in this graph be denoted by $l_{\text{NN}, k}$. The total length of the $k$-nearest neighbor graph, when suitably scaled, provides a measure of the average local ``density'', or concentration of the points in $S$. In Section~\ref{section::bootstrap_knn}, we will show bootstrap consistency for $l_{\text{NN}, k}$ within the stabilization framework.

\section{Bootstrapping Topological Statistics} \label{section::topology_bootstrap}

\subsection{Nonparametric Bootstrap} \label{section::nonparametric_bootstrap}

In this section, we will argue that the standard nonparametric bootstrap may fail to reproduce the correct sampling distribution asymptotically when applied to common topological statistics.

For a wide class of simplicial complexes built over point sets in $\reals^d$, the corresponding persistence diagram is unaffected by the inclusion of repeated points within the vertex set. This behavior holds for both the Vietoris-Rips and \v{C}ech complexes, defined in Section~\ref{section::simplicial_complexes}. In the case of the \v{C}ech complex, this phenomenon is seen most directly. The \v{C}ech complex under the Euclidean metric is homologically equivalent to a union of closed balls centered on the vertex points in $\reals^d$. Additional repetitions of vertex points leave both this union and the derived persistence diagram unchanged.

In these cases where repetitions may be ignored in the calculation of statistics, the standard bootstrap behaves effectively like a subsampling technique. The size of a given subsample is random, equal to the number of unique points present in the corresponding bootstrap sample.

Given a random sample $\mathbf X_n = \set{X_1, ..., X_n}$, it can be shown using elementary arguments that a given bootstrap sample $\mathbf X_n^*$ of size $n$ from the empirical distribution over $\mathbf X_n$ is expected to contain $n(1-\p{1-1/n}^n)\approx \p{1-e^{-1}}n\approx 0.632n$ unique points. As such, $\mathbf X_n^*$ behaves similarly to a sample of size $0.632n$, but is not scaled accordingly within the statistic $\p{\beta_q^{r, s}\p{\sqrt[d]{n}\mathbf X_{n}^*}-\E{\beta_q^{r, s}\p{\sqrt[d]{{n}}\mathbf X_{n}^*}\big| \mathbf X_n}}/\sqrt[d]{n}$. This discrepancy in scaling introduces a non-negligible  asymptotic bias. The effect is illustrated in Figure~\ref{figure::bias_plots} for the Vietoris-Rips complex.

\begin{figure}
	\begin{center}
		\includegraphics[width = \textwidth]{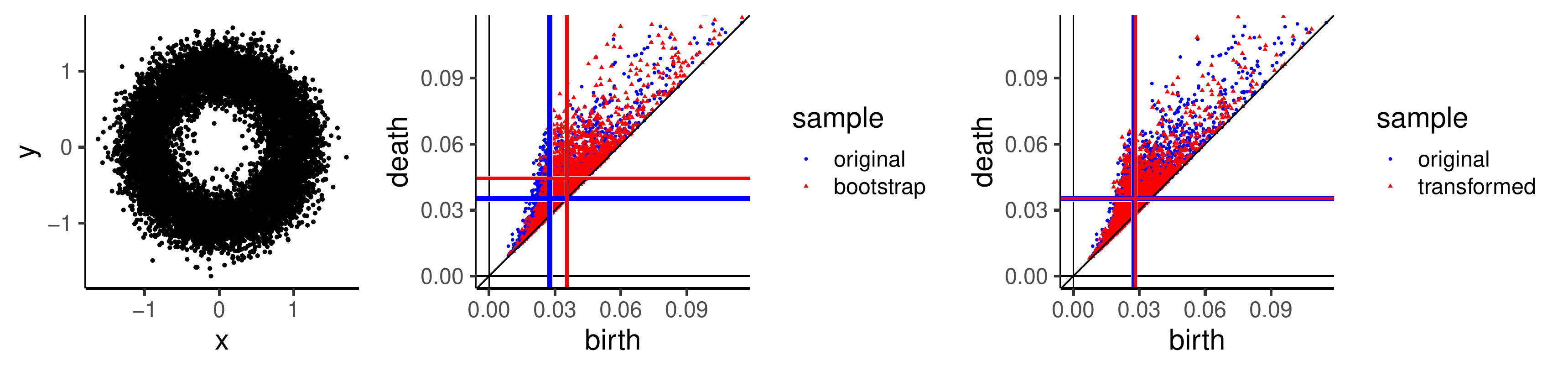}
	\end{center}
	\caption{{\em Left:} The original data set of size $n=10,000$, from which a single standard bootstrap sample is drawn. {\em Middle:} Persistence diagrams for both the original and bootstrap samples, along with lines denoting the median birth and death in each diagram. The asymptotic bias discussed in Section~\ref{section::nonparametric_bootstrap} can be clearly seen. {\em Right:} Persistence diagrams after application of a multiplicative correction factor of $\sqrt{1-e^{-1}}\approx 0.795$ to the bootstrap sample. Note that the median birth/death times correspond after transformation.\label{figure::bias_plots}}
\end{figure}

Furthermore, the standard nonparametric bootstrap results in a fundamentally different point process limit at small scales when compared to the original sample. For the original sample, when $\mathbf X_n$ is drawn from a distribution with density $f$, the shifted and rescaled sample $\sqrt[d]{n}\p{\mathbf X_n-z}$ approaches a homogeneous Poisson process $\mathbf P_z$ with intensity $f\p{z}$. From the preceeding stabilization literature (\cite{Penrose2001}, \cite{Krebs2019}), this limiting local point process drives the asymptotic sampling distribution of $\p{\beta_q^{r, s}\p{\sqrt[d]{n}\mathbf X_n}-\E{\beta_q^{r, s}\p{\sqrt[d]{n}\mathbf X_n}}}/\sqrt[d]{n}$. Considering the large-sample behavior of $\sqrt[d]{n}\p{\mathbf X_n^*-z}|\mathbf X_n$, the smoothed bootstrap sampling procedure described in Section~\ref{section::bootstrap} can be shown to reproduce the same local Poisson process $\mathbf P_z$ asymptotically.

However, the same is not true for the standard bootstrap when repeated points are ignored. In this case, $\sqrt[d]{n}\p{\mathbf X_n^*-z}|\mathbf X_n$ is restricted to the discrete set $\sqrt[d]{n}\p{\mathbf X_n-z}$, and thus cannot reproduce $\mathbf P_z$, whose domain is $\reals^d$. For this case, we describe the resulting point process limit $\mathbf Q_z$ in two steps. First, a homogenous Poisson process $\mathbf P_z$ is generated, representing $\sqrt[d]{n}\p{\mathbf X_n-z}$. Defined conditionally, $\mathbf Q_z|\mathbf P_z$ is a random subset of $\mathbf P_z$ such that $\prob{x\in\mathbf Q_z|\mathbf P_z}=1-e^{-1}\approx .632$, considering each point $x\in\mathbf P_z$ independently. We have $\sqrt[d]{n}\p{\mathbf X_n^*-z}\rightarrow \mathbf Q_z$.

This difference in local behavior, combined with the asymptotic bias effect illustrated earlier, are strong indicators that $\p{\beta_q^{r, s}\p{\sqrt[d]{n}\mathbf X^*_n}-\E{\beta_q^{r, s}\p{\sqrt[d]{n}\mathbf X^*_n}}\big| \mathbf X_n} / \sqrt[d]{n}$ and $\mybrk\p{\beta_q^{r, s}\p{\sqrt[d]{n}\mathbf X_n}-\E{\beta_q^{r, s}\p{\sqrt[d]{n}\mathbf X_n}}}/\sqrt[d]{n}$ likely do not share a weak limit. A technical treatment is omitted here, and is outlined merely to justify the use of our smoothed bootstrap procedure in place of the standard nonparametric bootstrap. The smoothed bootstrap procedure provides for bootstrap consistency (Corollaries~\ref{theorem::bootstrap_pbn} and \ref{theorem::bootstrap_pbn_ball}), and in the following sections we consider only this approach.

\subsection{General Conditions for Simplicial Complexes} \label{section::general_conditions}

The results presented in the following sections apply for a range of simplicial complexes constructed over point clouds in $\reals^d$. Here we will explain the specific conditions used, and for which common simplicial complexes they apply. Let $K$ be a function taking as input $S\in\tilde{\mathcal X}\p{\reals^d}$, giving as output a simplicial complex with vertices in $S$. For a given simplex $\sigma$, let the set diameter be $\diam{\sigma}$. We have the following conditions:

\begin{enumerate}[label=(K\arabic*), ref=K\arabic*]
	\item \label{condition::complex_increasing} For any $S\in \tilde{\mathcal X}\p{\reals^d}$ and $z\notin S$, $K\p{S}\subseteq K\p{S\cup\set{z}}$. Furthermore, $\sigma\in K\p{S\cup \set{z}}\setminus K\p{S}$ only if $z\in\sigma$.
	\item \label{condition::complex_translate} For any $S\in\tilde{\mathcal X}\p{\reals^d}$ and $z\in\reals^d$, $\sigma\in K\p{S}$ only if $\sigma-z\in K\p{S-z}$.
\end{enumerate}

\begin{enumerate}[label=(D\arabic*), ref=D\arabic*]
	\item \label{condition::complex_local_adj} There exists $\phi<\infty$ such that for any $S\in\tilde{\mathcal X}\p{\reals^d}$, $\sigma\in K^r\p{S}$ only if $\diam{\sigma}\leq \phi$.
	\item \label{condition::complex_local_ball} There exists $\phi<\infty$ such that for any $S\in\tilde{\mathcal X}\p{\reals^d}$ and $z\in\reals^d$, $\sigma\in K\p{S\cup\set{z}}\triangle K\p{S}$ only if $\sigma \subset B_z\p{\phi}$.
	\item \label{condition::complex_minimum} There exists an $\eta>0$ such that for any $S\in\tilde{\mathcal X}\p{\reals^d}$ and $x\in Z\p{K\p{S}}$, $\diam{x}\leq\eta$ only if $x\in B\p{K\p{S}}$.
	\item \label{condition::complex_higher} There exists an $m\in\nats$ such that for any $k>m$ and $S\in\tilde{\mathcal X}\p{\reals^d}$, $Z_k\p{K\p{S}}=B_k\p{K\p{S}}$.
\end{enumerate}

\eqref{condition::complex_increasing} means that the addition of a new point will not change the existing complex, only add new simplices. Furthermore, any new simplices gained must contain the added point as a vertex. \eqref{condition::complex_translate} gives that the complex is essentially translation invariant. \eqref{condition::complex_local_adj} sets a maximum diameter for any simplex in the complex. \eqref{condition::complex_local_ball} gives that the influence of a new point on the complex is confined to a local region around that point, within a fixed diameter. This condition allows for both the addition and removal of simplices from the complex, but only within the prescribed radius. It can be easily shown that if \eqref{condition::complex_local_ball} holds for $\phi$, \eqref{condition::complex_local_adj} holds for $2\phi$. Conversely if both \eqref{condition::complex_increasing} and \eqref{condition::complex_local_adj} hold for $\phi$, \eqref{condition::complex_local_ball} also holds for $\phi$. Finally, \eqref{condition::complex_minimum} gives that no small loops can exist with unfilled interiors, and \eqref{condition::complex_higher} gives that all Betti numbers are $0$ in sufficiently high feature dimensions.

Now, let $\mathcal K=\p{K^r}_{r\in\reals}$ be a function taking as input $S\in\tilde{\mathcal X}\p{\reals^d}$, giving as output a filtration of simplicial complexes with vertices in $S$. As a slight abuse, we will often refer to the function $\mathcal K$ as a filtration of simplicial complexes, even though it is a function defining more than a single filtration, depending on the underlying point cloud. We say that a given condition is satisfied for $\mathcal K$ if it is satisfied by $K^r$ for any $r\in\reals$. In the cases of \eqref{condition::complex_local_adj}, \eqref{condition::complex_local_ball}, and \eqref{condition::complex_minimum}, $\phi$ and $\eta$ may depend on $r$ as increasing functions $\phi\colon\reals\rightarrow\left[0, \infty\right)$ and $\eta\colon\reals\rightarrow\left[0, \infty\right)$.

It can be shown that all of \eqref{condition::complex_increasing}-\eqref{condition::complex_minimum} are satisfied for both the Vietoris-Rips and \v{C}ech complexes in $\reals^d$ using $\phi\p{r}=\eta\p{r}=2r$. The same functions apply for the alpha complex in $\reals^d$ and its completion $\mathcal K_{\alpha^*}$, with the notable exception that \eqref{condition::complex_increasing} is violated. Finally, it is known that \eqref{condition::complex_higher} is satisfied by the alpha, \v{C}ech, and Delauney complexes in $\reals^d$ for $m=d-1$. 

While covering a wide class of distance-based simplicial complexes, there are several complexes used in practice that may fail to satisfy any or all of these. For example, the addition of a new point to the Delaunay complex, Gabriel graph, witness complex, or $k$-nearest neighbor graph can both add and remove simplices, violating \eqref{condition::complex_increasing}. Furthermore, there is not any limit on the simplex diameter within any of these complexes, violating \eqref{condition::complex_local_adj}. Likewise, the addition of a single point can alter simplices at arbitrarily large distances, violating \eqref{condition::complex_local_ball}. As a special note, it is common in practice to consider the intersection of the Vietoris-Rips and Delaunay complexes, which unfortunately may violate all the assumptions here. It is unclear if an extension or special consideration could be made to incorporate these complexes.

\subsection{Stabilization of Persistent Betti Numbers} \label{section::stabilization_results}

To apply the general bootstrap theorem, we first require a technical lemma establishing a locally-determined radius of stabilization for persistent Betti numbers. The result given applies for general classes of simplicial complexes constructed over subsets of $\reals^d$, using the conditions listed previously. Reiterating, $\mathcal C_{p, M}\p{\reals^d}$ is the class of distributions $G$ on $\reals^d$ with densities $g$ such that $\|g\|_p\leq M$. We have the following:

\begin{lemma} \label{lemma::stabilization_pbn}
	Let $F\in\mathcal C_{p, M}\p{\reals^d}$ for some $p>2$ and $M<\infty$, and let $\mathcal K=\set{K^r}_{r\in\reals}$ be a filtration of simplicial complexes satisfying \eqref{condition::complex_translate}, \eqref{condition::complex_local_ball}, and \eqref{condition::complex_minimum}. Then for any $r\in \reals$, $s\in\reals$, and $q\geq0$, $\beta_q^{r, s}\p{\mathcal K}$ satisfies \eqref{condition::stabilization_radius} for $F$.
\end{lemma}
	
\subsection{Bootstrap Results for Persistence Homology} \label{section::bootstrap_results}

Here we present the main applied results of this paper. Each is derived from Theorem~\ref{theorem::bootstrap_general} and the stabilization lemma for persistent Betti numbers (Lemma~\ref{lemma::stabilization_pbn}). For given vectors of birth and death times, $\vec r=\p{r_i}_{i=1}^k$ and $\vec s=\p{s_i}_{i=1}^k$, let $\beta_q^{\vec r, \vec s}=\p{\beta_q^{r_i, s_i}}_{i=1}^k$ denote the multivariate function whose components are the persistent Betti numbers evaluated at each pair of birth and death times. For a vector of filtration times $\vec{r}=\p{r_i}_{i=1}^k$, let $\chi^{\vec r}$ denote the function giving the Euler characteristic at each time $r_i$, with $\chi^{\vec r}:=\p{\chi\p{K^{r_i}}}_{i=1}^k$. 

The following apply for $F\in \mathcal P\p{\reals^d}$ with density $f$ such that $\|f\|_p<\infty$ for some $p>2$, as specified. $F$ and $\hat F_n$ are such that $\hat F_n$ has density $\hat f_n$, $\| \hat f_n-f\|_1 \to 0$, and $\| \hat f_n-f \|_p \rightarrow 0$ in probability (resp.\ $a.s.$). Let $\mathbf X_n=\set{X_i}_{i=1}^n\iid F$ and $\p{m_n}_{n\in\nats}$ such that $\lim_{n\rightarrow\infty}m_n=\infty$. $\mathbf X_{m_n}^*=\set{X_i^*}_{i=1}^{m_n}\iid \hat F_n\big|\mathbf X_n$ is a bootstrap sample and $G$ a multivariate distribution. Recalling the conclusion of Theorem~\ref{theorem::bootstrap_general}, for a multivariate statistic $\vec\psi$:
\begin{statement} \label{conclusion::bootstrap}
	\begin{equation*}
		\frac{1}{\sqrt{n}}\p{\vec{\psi}\p{\sqrt[d]{n}\mathbf X_n}-\E{\vec{\psi}\p{\sqrt[d]{n}\mathbf X_n}}}\overset{d}{\rightarrow}G 
	\end{equation*}
	\centering{if and only if}
	\begin{equation*}
		\frac{1}{\sqrt{{m_n}}}\p{\vec{\psi}\p{\sqrt[d]{{m_n}}\mathbf X_{m_n}^*}-\E{\vec{\psi}\p{\sqrt[d]{{m_n}}\mathbf X_{m_n}^*}\big| \mathbf X_n}  }\overset{d}{\rightarrow}G \text{ in probability (resp. a.s.)}.
	\end{equation*}
\end{statement}

For cases with a corresponding central limit theorem, $G$ is the limiting normal distribution of the original standardized statistic.

\begin{corollary}[\bf Persistent Betti Numbers] \label{theorem::bootstrap_pbn}
	Let $q\geq 0$ and $p>2q+3$. Let $\mathcal K$ be a filtration of simplicial complexes satisfying \eqref{condition::complex_increasing}, \eqref{condition::complex_translate}, \eqref{condition::complex_local_adj}, and \eqref{condition::complex_minimum}. Then for any given $\vec r$, $\vec s$, Statement~\ref{conclusion::bootstrap} holds for $\beta_q^{\vec r, \vec s}$.
\end{corollary}

\begin{corollary}[\bf Persistent Betti Numbers - Alt.] \label{theorem::bootstrap_pbn_ball}
	Let $q\geq 0$ and $p>2q+5$. Let $\mathcal K$ be a filtration of simplicial complexes satisfying \eqref{condition::complex_translate}, \eqref{condition::complex_local_ball}, and \eqref{condition::complex_minimum}. Then for any given $\vec r$, $\vec s$, Statement~\ref{conclusion::bootstrap} holds for $\beta_q^{\vec r, \vec s}$.
\end{corollary}

The only differences between the above corollaries are the conditions satisfied by the underlying simplicial complex and the necessary norm bound on the density. The corresponding results for the Betti numbers follow as special cases of Corollaries~\ref{theorem::bootstrap_pbn} and \ref{theorem::bootstrap_pbn_ball}, when the given birth and death parameters are equal ($\beta^{\vec r}_q=\beta^{\vec r,\vec r}_q$). Also, although the statements of Corollaries~\ref{theorem::bootstrap_pbn} and \ref{theorem::bootstrap_pbn_ball} are given in terms of a fixed feature dimension $q$, a direct extension exists if $q=q_i$ is allowed to differ for each $\p{r_i, s_i}$. The form as given shows the dependence of the density norm assumption on the chosen feature dimension.

The higher value of $p$ required in Corollary~\ref{theorem::bootstrap_pbn_ball} compared to Corollary~\ref{theorem::bootstrap_pbn} can be explained intuitively based on the assumptions used. For the persistent Betti numbers, the main quantity controlling convergence is the expected number of simplices altered or introduced when a new datapoint is added to the sample. \eqref{condition::complex_local_ball} ensures that these simplices fall within a small ball around the new data point. The stated density norm conditions control the expected number of points, and by extension possible simplices, that can lie within that small ball. Introducing \eqref{condition::complex_increasing} further controls the number of possible simplices, and allows for a weakening of the necessary norm condition. \eqref{condition::complex_increasing} requires that, as the sample grows by a single point, any additional simplices must contain the new point as a vertex, and no deletion of simplices is possible. This means that every added simplex has one less ``free" vertex, and a weaker norm condition is required for control. The same intuition applies whenever \eqref{condition::complex_increasing} is assumed.

In the specific case of the alpha complex, both of the above Corollaries~\ref{theorem::bootstrap_pbn} and \ref{theorem::bootstrap_pbn_ball} apply. While the alpha complex does not satisy \eqref{condition::complex_increasing}, it has equal persistent Betti numbers to the \v{C}ech complex, which does. Thus, the weaker conditions of Corollary~\ref{theorem::bootstrap_pbn} are sufficient in this unique case.

\begin{corollary}[\bf Euler Characteristic] \label{theorem::bootstrap_euler_nontrunc}
	Let $m<\infty$ and $p>2m+3$. Let $\mathcal K$ be a filtration of simplicial complexes satisfying \eqref{condition::complex_increasing}, \eqref{condition::complex_translate}, \eqref{condition::complex_local_adj}, \eqref{condition::complex_minimum}, and \eqref{condition::complex_higher}. Then for any given $\vec r$, Statement~\ref{conclusion::bootstrap} holds for $\chi^{\vec r}$.
\end{corollary}

\begin{corollary}[\bf Euler Characteristic - Alt.] \label{theorem::bootstrap_euler_nontrunc_ball}
	Let $m<\infty$ and $p>2m+5$. Let $\mathcal K$ be a filtration of simplicial complexes satisfying \eqref{condition::complex_translate}, \eqref{condition::complex_local_ball}, \eqref{condition::complex_minimum}, and \eqref{condition::complex_higher}. Then for any given $\vec r$, Statement~\ref{conclusion::bootstrap} holds for $\chi^{\vec r}$.
\end{corollary}

It is suspected that some of the simplicial complex assumptions can be relaxed in the persistent Betti number and Euler characteristic cases, but the extent to which this is possible is still unknown. Specifically, Corollary~\ref{theorem::bootstrap_pbn} requires a translation-invariant simplicial complex \eqref{condition::complex_translate}, along with the elimination of small loops via \eqref{condition::complex_minimum}. See Appendix~\ref{section::b_bounded} for altered ``$B$-bounded persistent Betti number'' and ``$q$-truncated Euler characteristic'' problem settings where these issues may be resolved.

To strengthen Corollaries~\ref{theorem::bootstrap_pbn}-\ref{theorem::bootstrap_euler_nontrunc_ball} with rates, we require more specific knowledge about the convergence to $G$ of the original statistic. For persistent Betti numbers in the multivariate setting, general central limit theorems have been shown in \cite{Krebs2019}, but little is known at this time with regards to rates of convergence. Proposition~\ref{proposition::w2_general} does allow for rates of convergence in 2-Wasserstein distance between the bootstrap and true sampling distributions for finite sample sizes, but is phrased in terms of a tail probability for the radius of stabilization. See the proofs of Corollaries~\ref{theorem::bootstrap_pbn}-\ref{theorem::bootstrap_euler_nontrunc_ball} for details. For persistent Betti numbers the tail behavior of the radius of stabilization is poorly understood. Owing to these difficulties, we may only conclude consistency of the smoothed bootstrap for the functions considered.

\subsection{Bootstrap Results for $k$-Nearest Neighbor Graphs} \label{section::bootstrap_knn}

In the following, let $\mathcal D_{\gamma, r_0}\p{C}$ be the class of distributions $G$ with support on a bounded $C\subset \reals^d$ such that $\int_{B_x\p{r}}\de G\geq \gamma r^d$ for all $r\leq r_0$ and $x\in C$.

\begin{corollary}[\bf Total Edge Length of the $k$-Nearest Neighbor Graph] \label{theorem::bootstrap_knn}
	Let $p>2$. Furthermore, let $F\in \mathcal D_{\gamma, r_0}\p{C}$ and $\ind{\hat F_n\in \mathcal D_{\gamma, r_0}\p{C}}\rightarrow 1$ in probability (resp. a.s.). Then Statement~\ref{conclusion::bootstrap} holds for $l_{\text{NN}, k}$.
\end{corollary}

The conditions of Corollary~\ref{theorem::bootstrap_knn} are in particular satisfied when $C$ is known and convex, with $f$ bounded below on $C$ by a constant, provided further that $\|\hat f_n-f\|_{\infty}\rightarrow 0$ in probability (resp. a.s.). We include this final result to demonstrate the utility of stabilization as a general tool for proving bootstrap convergence theorems outside of topological data analysis. The $k$-nearest neighbor graph does not fall under the general simplicial complex conditions provided in Section~\ref{section::general_conditions}, thus special treatment is needed to show the required stabilization and moment conditions. Here we rely on previous results from the literature, see \cite{Penrose2001} for stabilization results and the corresponding central limit theorem.

\section{Simulation Study} \label{section::simulation_study}

In this section we present the results of a series of simulations illustrating the finite-sample properties of the smoothed bootstrap applied to persistent Betti numbers $\beta_q^{r,s}$ of the Vietoris-Rips complex constructed over point sets in $\reals^d$. Precise definitions and an introduction to the properties of these statistics may be found in Section~\ref{section::simpicial_complex}. Source code for this section, as well as for the data analysis of Section~\ref{section::data_analysis} is available at \mybrk \href{https://github.com/btroycraft/stabilizing_statistics_bootstrap}{github.com/btroycraft/stabilizing\_statistics\_bootstrap} \cite{Roycraft2021}.

We investigate the coverage probability of bootstrap confidence intervals on the expected persistent Betti numbers $\E{\beta_q^{r,s}\p{\sqrt[d]{n}\mathbf X_n}}$ for a variety of feature dimensions, sample sizes, data generating mechanisms, and bandwidth selectors. Table~\ref{table::distributions} lists brief descriptions of the data distributions considered. For more detailed explanations, see Appendix~\ref{appendix::simulation_details}. The results of the simulations are given in Table~\ref{table::sim_results}. For the persistent Betti numbers, a single choice of $\p{r, s}$ was made for each combination of distribution and feature dimension, chosen to lie within the main body of features in the corresponding persistence diagram. For computational reasons, only feature dimensions $q=1$ and $q=2$ are considered.

We consider five data-driven bandwidth selectors. First are the ``Hpi.diag'' (plug-in), ``Hlscv.diag'' (least-squares cross-validation), and ``Hscv.diag'' (smoothed cross-validation) selectors from the \textit{ks} package in R. Second, we include the adaptive bandwidth selector described in Section~\ref{section::data_analysis}. While this selector is tailored for the specifics of astronomical data, we include it here for completeness. Each of these four selectors are available for data dimension up to $d=6$. Last, we consider Silverman's rule of thumb (see \cite{Silverman1986}) via ``bw.silv'' from the \textit{kernelboot} package in R, which accepts data in any dimension.

For the two cross-validation selectors, note that a bandwidth is not always selected, throwing errors on some datasets. To accommodate the automatic setting of this simulation study, any error-producing data sets were simply rejected for each of these cases.
 
There is a noticeable drop-off in coverage as the data dimension increases. This is expected, as the kernel density estimator is known to suffer from a ``curse of dimensionality". For distribution $F_6$, which exhibits heavy tails, only the adaptive bandwidth selector performed well, because outliers are weighted much less heavily in this case. It is likely that performance will suffer generally in the presence of heavy tailed data when using one of the selectors with common bandwidth.

The coverage proportion is generally smaller than the nominal level of $95\%$. Therefore, it is recommended to use a larger than desired level, especially for limited sample sizes. In terms of general performance, we recommend any of ``Hpi.diag'', ``Hlscv.diag'', or ``Hscv.diag''. These selectors provide the most consistent coverage, and effectively replicate the nominal 95\% level in many cases, especially for the largest sample size $n=400$. Silverman's rule performs badly in several cases, and should only be used in the absence of better alternatives.

\begin{table}
	\centering
	\begin{tabular}{||c|c||}
		\hline
		Label & Description \\
		\hhline{#=|=#}
		$F_1$ & Rotationally symmetric in $\reals^2$, finite $L_8$ norm\\
		$F_2$ & Rotationally symmetric in $\reals^2$, finite $L_2$ norm, infinite $L_8$ norm\\
		$F_3$ & $\mathbb S^1$ embedded in $\reals^2$, additive Gaussian noise \\
		$F_4$ & Uniformly distributed over $B_0\p{1}$ in $\reals^3$, additive Gaussian noise\\
		$F_5$ & 5 clusters in $\reals^3$, additive exponential noise \\
		$F_6$ & $\mathbb S^2$ embedded in $\reals^5$, additive Cauchy noise \\
		$F_7$ & Flat figure-8 embedded in $\reals^{10}$, additive Gaussian noise \\
		\hline
	\end{tabular}
	
	\caption{Description of densities or distributions considered for the simulation study of Section~\ref{section::simulation_study}. For the distributions based on manifolds, we first draw uniformly from the manifold, then apply the prescribed additive noise. Detailed explanations of the distributions considered, along with precise definitions are available in Appendix~\ref{appendix::simulation_details}. \label{table::distributions}}
\end{table}

\begin{table}[]
	\begin{tabular}{||c||l|l|l|l|l|l|l||l|l|l|l||}
		\hline
		\multicolumn{1}{||c||}{Distr.} &
		$F_1$ & $F_2$ & $F_3$ & $F_4$ & $F_5$ & $F_6$ & $F_7$ & $F_4$ & $F_5$ & $F_6$ & $F_7$ \\ \hline
		& \multicolumn{7}{c||}{$q=1$} & \multicolumn{4}{c||}{$q=2$}    \\ \hline
		$r$ & 4.94 & 5.20 & 3.03 & 1.92 & 0.30 & 1.78 & 1.28 & 2.96 & 0.39 & 2.71 & 1.46 \\
		$s$ & 5.36 & 5.60 & 3.28 & 2.12 & 0.31 & 1.91 & 1.32 & 3.04 & 0.40 & 2.80 & 1.47 \\ \hhline{#=#=======#====#}
		$n=100$ & 0.896 & 0.965 & 0.921 & 0.859 & 0.954 & 0.19  &       & 0.908 & 0.705 & 0.038 &       \\
		        & 0.931 & 0.959 & 0.914 & 0.809 & 0.941 & 0.133 &       & 0.903 & 0.604 & 0.045 &       \\
		        & 0.903 & 0.97  & 0.91  & 0.859 & 0.927 & 0.049 &       & 0.902 & 0.363 & 0.002 &       \\
		        & 0.922 & 0.898 & 0.899 & 0.71  & 0.725 & 0.736 &       & 0.837 & 0.048 & 0.051 &       \\
		        & 0.359 & 0.931 & 0.942 & 0.864 & 0     & 0     & 0.656 & 0.902 & 0     & 0     & 0.045 \\ \hline
		$n=200$ & 0.908 & 0.971 & 0.94  & 0.898 & 0.942 & 0.159 &       & 0.878 & 0.795 & 0.125 &       \\
		        & 0.92  & 0.972 & 0.946 & 0.891 & 0.923 & 0.106 &       & 0.872 & 0.707 & 0.074 &       \\
		        & 0.888 & 0.975 & 0.959 & 0.906 & 0.892 & 0.06  &       & 0.908 & 0.277 & 0.031 &       \\
		        & 0.888 & 0.909 & 0.828 & 0.783 & 0.773 & 0.705 &       & 0.673 & 0.032 & 0.27  &       \\
		        & 0.299 & 0.954 & 0.903 & 0.899 & 0     & 0     & 0.766 & 0.882 & 0     & 0     & 0.537 \\ \hline
		$n=300$ & 0.9   & 0.971 & 0.926 & 0.921 & 0.94  & 0.183 &       & 0.854 & 0.906 & 0.225 &       \\
		        & 0.94  & 0.971 & 0.938 & 0.896 & 0.94  & 0.087 &       & 0.854 & 0.917 & 0.072 &       \\
		        & 0.913 & 0.971 & 0.94  & 0.896 & 0.922 & 0.054 &       & 0.855 & 0.964 & 0.074 &       \\
		        & 0.93  & 0.923 & 0.864 & 0.786 & 0.771 & 0.735 &       & 0.712 & 0.551 & 0.575 &       \\
		        & 0.283 & 0.956 & 0.925 & 0.906 & 0     & 0     & 0.835 & 0.856 & 0     & 0     & 0.508 \\ \hline
		$n=400$ & 0.918 & 0.961 & 0.947 & 0.934 & 0.96  & 0.175 &       & 0.851 & 0.883 & 0.259 &       \\
		        & 0.927 & 0.951 & 0.938 & 0.92  & 0.955 & 0.063 &       & 0.839 & 0.88  & 0.076 &       \\
		        & 0.908 & 0.976 & 0.933 & 0.924 & 0.939 & 0.062 &       & 0.863 & 0.958 & 0.099 &       \\
		        & 0.911 & 0.922 & 0.874 & 0.813 & 0.825 & 0.771 &       & 0.695 & 0.952 & 0.789 &       \\
		        & 0.266 & 0.961 & 0.909 & 0.922 & 0.114 & 0     & 0.891 & 0.859 & 0     & 0     & 0.584 \\ \hline
	\end{tabular}
	
	\caption{Coverage proportions for $95\%$ smoothed bootstrap confidence intervals on the mean persistent Betti numbers; coverage is estimated using $N=1,000$ independent base samples with $B=500$ bootstrap samples each. True mean persistent Betti numbers are estimated using a large ($N=100,000$) number of independent samples from the true distribution. For each case, the values from top to bottom: Coverage proportions using ``Hpi.diag'', ``Hlscv.diag'', ``Hscv.diag'', ``adaptive'', and ``bw.silv'' bandwidth selectors, respectively. (see Section~\ref{section::simulation_study}) \label{table::sim_results}}
\end{table}

\section{Data Analysis} \label{section::data_analysis}
In this section we show how smoothed bootstrap estimation performs on a real dataset. We consider a selection of galaxies from the Sloan Digital Sky Survey \cite{Blanton2017}, chosen from a selection of sky with right ascension values between $100^\circ$ and $270^\circ$ and declination between $-7^\circ$ and $70^\circ$. Three slices of galaxies were considered, separated by redshift, a measure of radial distance from the solar system. The selections consist of galaxies with red-shift within $\p{0.025, 0.026}$, $\p{0.027, 0.028}$, and $\p{0.029, 0.030}$, respectively. These slices were chosen to investigate the topological properties of the cosmic web across time. In this case, due to the rough homogeneity of the web at large scales, few significant topological deviations are expected.

Subset limits were chosen to maintain computational feasibility and avoid measurement gaps. In an initial cleaning step, each slice was flattened using an area-preserving cylindrical projection and trimmed so that the slices share a common boundary with the same number of galaxies ($2374$) per slice. Angular units are converted to distances in Megaparsecs (Mpc) based on the redshift and Hubble's constant.

The distribution of galaxies in each dataset is modeled by a random sample from some bivariate probability distribution, where the location of each galaxy is drawn independently from the overall distribution. As a part of the model framework, the effect of gravitational interaction manifests via a macroscopic change in the matter distribution, rather than as dependency between individual galaxies.

Following the recommendation of \cite{Ferdosi2011}, we estimate the density of the matter distribution using the adaptive bandwidth selector described in \cite{Breiman1977}. This adaptive bandwidth selector was chosen to accommodate for the large variations in density present within astronomy data. The selectors considered in Section~\ref{section::simulation_study} do not perform well in this context, often oversmoothing by a large margin. A pilot density estimator was constructed based on the ``Hpi.diag'' plug-in bandwidth selector and a Gaussian kernel.

Visualizations of the density estimates are provided in Figure~\ref{figure::galaxy_diagrams}. Generally, the fit adequately captures the filament structures present in the raw data. Within the persistence diagrams, the mass of features present close to the main diagonal represents small-scale holes between neighboring galaxies, whereas features farther from the diagonal represent the large-scale holes formed by relatively disparate galaxies.

We apply the Vietoris-Rips complex to each of the slices, and calculate a selection of persistent Betti numbers in dimensions $q=0$ and $q=1$. The $0$-dimensional features summarize cluster and filament structure, whereas the $1$-dimensional features describe voids and depressions. The transformed datasets and persistence diagrams in dimension $q=1$ can be seen in Figure~\ref{figure::galaxy_diagrams}. We consider the Betti numbers $\beta_0^r$ and $\beta_1^r$, as well as the persistent Betti numbers $\beta_1^{r, r+1}$ for $r=3, ..., 30$ Mpc. Filtration parameters for the persistent Betti numbers were chosen to lie close to the diagonal $r = s$, excluding features with a lifetime less than $1$ Mpc. We use bootstrap estimation to construct nominal $98\%$ confidence intervals for the population mean values, both pointwise and simultaneous within each regime across $r=3, ..., 30$ Mpc. The number of bootstrap replicates used was $B = 20,000$, with results seen in Figure~\ref{figure::galaxy_boot}.

In feature dimension $q=0$, the curves show similar behavior across the slices. Consistent with our empirical results, similar Betti curves are expected when the within-filament matter distribution and overall frequency of filaments for each sample are equal. For feature dimension $q=1$, more variation is present. However, as can be seen from the bootstrap confidence intervals, much of this variation is explained by random fluctuation. For example, while a notable depression around the scale of $8$ Mpc exists for the third slice, it is still within the margins of error provided. From this analysis, we do not find significant differences in the topological properties of the three samples over the range of filtration parameters considered. The difference in topological structure seen within each pair of Betti curves is within the margin of error provided by the bootstrap confidence intervals, especially considering the wider simultaneous intervals.

The consistency shown in Section~\ref{section::bootstrap_results} for bootstrap estimation applies only for those features within the ``body" of topological features, being those occurring at a local scale. Features with large persistence or ones that appear at large diameter are not accounted for in this, as their relative weight is small within the persistent Betti numbers. As such, our analysis does not preclude differences in topology at a large relative scale, describing the largest galactic structures.

\begin{figure}
	\begin{center}
		\includegraphics[width=\textwidth]{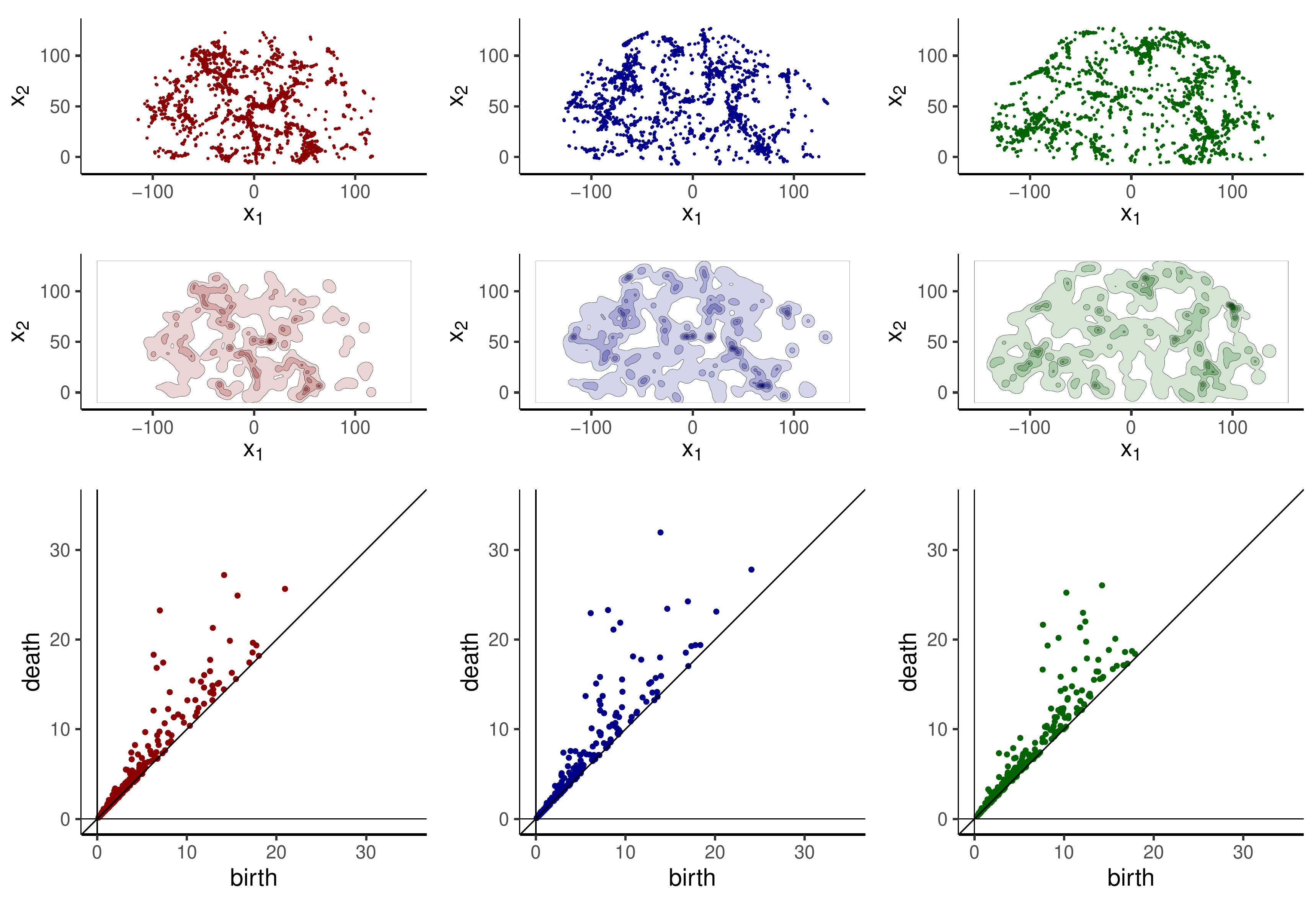}
	\end{center}
	\caption{Top row: Transformed point clouds. Middle row: Density estimates using adaptive bandwidth. Bottom row: Persistence diagrams in dimension $q=1$ for the Vietoris-Rips complex. Columns from left to right: Galaxies with redshifts within $\p{0.025, 0.026}$, $\p{0.027, 0.028}$, and $\p{0.029, 0.030}$, respectively. Axis units are given in Megaparsecs (Mpc). \label{figure::galaxy_diagrams}}
\end{figure}

\begin{figure}
	\begin{center}
		\includegraphics[width=\textwidth]{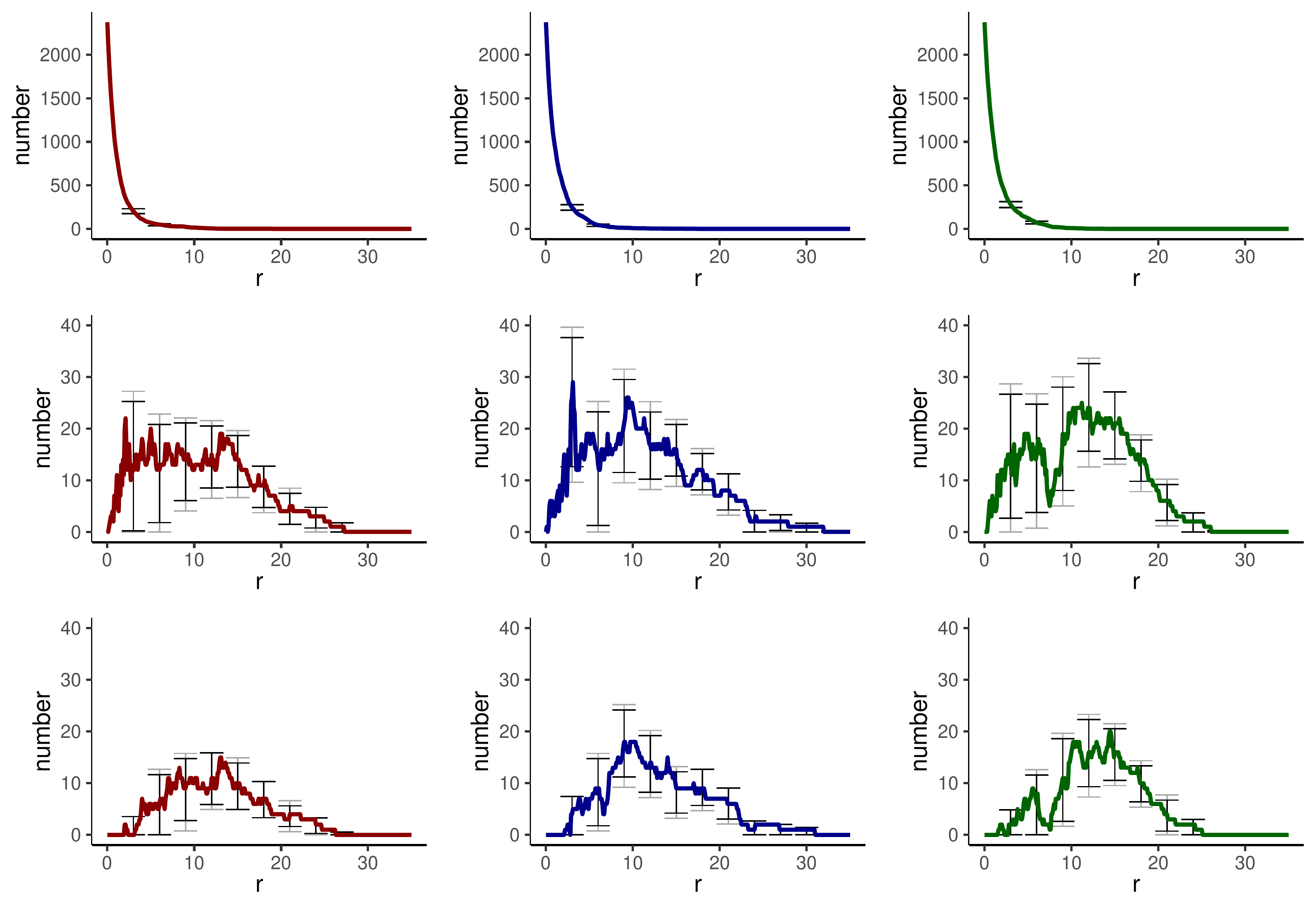}
	\end{center}
	\caption{Betti curves for the Vietoris-Rips complex. Top row: Betti numbers $\beta_0^r$.  Middle row: Betti numbers $\beta_1^r$. Bottom Row: persistent Betti numbers $\beta_1^{r, r+1}$. Columns correspond with those of Figure~\ref{figure::galaxy_diagrams}. Axis units are given in Megaparsecs (Mpc). For each of $r=3, ..., 30$ Mpc, simultaneous bootstrap confidence bands are given in gray, drawn from bootstrap samples of size $B=20,000$. Likewise, pointwise intervals are given in black. \label{figure::galaxy_boot}}
\end{figure}

\section{Discussion}
In this work we have shown the large-sample consistency of multivariate bootstrap estimation for a range of stabilizing statistics. This includes the persistent Betti numbers, the Euler characteristic, and the total edge length of the $k$-nearest neighbor graph. However, many open questions still remain.

In Section~\ref{section::nonparametric_bootstrap} it was argued that the standard nonparametric bootstrap may fail to directly reproduce the correct sampling distribution asymptotically for topological statistics like the persistent Betti numbers. However, there remains the possibility that a corrected version of the standard bootstrap could provide for consistency. As discussed in Section~\ref{section::nonparametric_bootstrap}, standard bootstrap sampling results in a fundamentally different point process limit at small scales. Previous stabilization results primarily consider Poisson and related processes, meaning a full theoretical treatment of the standard bootstrap would likely require reconstructing much of the previous stabilization and central limit theorem results for the alternative limiting process.

The results for the smoothed bootstrap presented here apply only in the multivariate setting, the obvious extension being to stochastic processes. Essential to a process-level result concerning the persistent Betti numbers would be a convenient tail bound for the radius of stabilization, which is yet unavailable. In the case of persistent Betti numbers, there is a strong relationship between the persistent Betti function and an empirical CDF in two dimensions. As such, there is much established theory in that regard which may be applied once stochastic equicontinuity is established.

In practice it is common that data comes not from a density in $\reals^d$, but instead from a manifold. It is suspected that a version of the results in this paper could apply in the manifold setting. However, this requires a bootstrap that adapts to a possibly unknown manifold structure, similar to that found in \cite{Kim2018}. Combined with the inherent challenges of working with manifolds, this extension presents many technical hurdles.

Furthermore, in this work we have shown only consistency for bootstrap estimation to a common limiting distribution. The rates of convergence in the $2$-Wasserstein distance regarding the persistent Betti numbers rely on the unknown tail properties of the corresponding radius of stabilization. Quantifying these tail properties is a challenging open problem, and seems to be a key step towards an eventual rate calculation, as well as the previously mentioned process-level result.

Finally, there are several statistics of interest, including those based on the Delaunay complex, which do not fit into the specific frameworks provided here. It may be that these statistics may still satisfy Theorem~\ref{theorem::bootstrap_general} in the general case, by techniques others than those provided here.

\section*{Acknowledgements}
Thank you to the reviewers for their helpful comments and thorough examination of this work.

Benjamin Roycraft was partially supported by the National Science Foundation (NSF), grant number DMS-1148643. Johannes Krebs was partially supported by the German Research Foundation (DFG), grant number KR-4977/2-1. Wolfgang Polonik was partially supported by the National Science Foundation (NSF), grant number DMS-2015575.

Funding for the Sloan Digital Sky Survey IV has been provided by the Alfred P. Sloan Foundation, the U.S. Department of Energy Office of Science,  and the Participating Institutions. SDSS-IV acknowledges support and resources from the Center for High-Performance Computing at the University of Utah. The SDSS web site is www.sdss.org.




\nocite{*}
\bibliographystyle{imsart-number} 


\appendix

\section{Altered Problem Settings} \label{section::b_bounded}

\subsection{$B$-Bounded Persistent Betti Numbers} \label{section::bounded_persistent_betti_numbers}

To effectively quantify the radius of stabilization for persistent Betti numbers, it is necessary to place controls on the size of possible cycles within a simplicial complex. Large loops extend the influence of a single point beyond the local region, and complicate statistical analysis. As such, we present the following definitions which eliminate any large loops. Note that the statistics of this appendix are presented for their convenient theoretical properties, not their practical significance. Let $S\in\tilde{\mathcal X}\p{\reals^d}$ and $K=K\p{S}$ be a simplicial complex with vertices in $S$. For a given chain of simplices $\sum_{i=1}^m\sigma_i\in C\p{K\p{S}}$, we have the diameter given by $\diam{\sum_{i=1}^m\sigma_i}:=\diam{\bigcup_{i=1}^m \sigma_i}$. Let the space of \textit{$B$-bounded cycles} of the complex $K$ be the vector space, denoted by $Z_{q,B}\p{K}$, spanned by cycles in $K$ with diameter no larger than $B$. We have $Z_{q,B}\p{K}:=\text{span}\set{x\in Z_q\p{K} \st \diam{x}\leq B}$. Likewise let the space of \textit{$B$-bounded boundaries} be $B_{q,B}\p{K}:=\text{span}\set{x\in B_q\p{K} \st \diam{x}\leq B}$.

The definitions presented here are directly inspired by a previous concept under the name ``$M$-bounded persistence" found in \cite{Biscio2020}, and may be viewed as a generalization thereof. In this previous work, it was shown that a diameter bound of this type is sufficient for establishing functional central limit theorems for persistent Betti numbers, and thus an extension is desireable. The original definition given in \cite{Biscio2020} is based on a correspondence between loops and connected components in the complement space, and does not apply to arbitrary simplicial complexes and feature dimensions.

The change in naming effected here is not meant to drawn a distinction between the two definitions, but purely to avoid overloading symbols within this paper. $M$ is used in this work to denote an upper bound for a density norm.

The $B$-bounded spaces obey many of the same properties as their original counterparts. We have $B_{q,M}\p{K}\subseteq Z_{q,M}\p{K}$. Thus we can define the \textit{$B$-bounded homology spaces} as $H_{q,B}\p{K}=Z_{q,B}\p{K}/B_{q,B}\p{K}$. It should be noted that these definitions allow for chains of unbounded diameter, so long as there exists a decomposition into a sum of bounded chains. Furthermore, for $B_{q,B}\p{K}$, the diameter control is on the chains $x\in B_q\p{K}$, not on a corresponding $y\in C_{q+1}\p{K}$ with $x=\partial y$. It is possible to have a chain with arbitrarily high diameter, whose boundary has diameter less than $B$.

We next define the analog of Betti numbers and persistent Betti numbers over a filtration of simplicial complexes in the bounded context. Given a filtration $\mathcal K=\set{K^r}_{r\in\reals}$, we have $B$-bounded analogs for the Betti numbers, persistent homology spaces, and persistent Betti numbers given by
\begin{align}
	\beta_{q, B}^r\p{\mathcal K} & := \dim\p{H_{q,B}\p{K^r}}\\
	& =\dim\p{Z_{q,B}\p{K^r}}-\dim\p{B_{q,B}\p{K^r}} \\
	H_{q,B}^{r,s}\p{\mathcal K} & := \frac{Z_{q,B}\p{K^r}}{Z_{q,B}\p{K^r}\cap B_{q,B}\p{K^s}} \\
	\beta_{q,B}^{r,s}\p{\mathcal K} & := \dim\p{H_{q,B}^{r,s}\p{\mathcal K}} \\
	& = \dim\p{Z_{q,B}\p{K^r}} - \dim\p{Z_{q,B}\p{K^r}\cap B_{q,B}\p{K^s}}.
\end{align}

Unfortunately, no direct analog of the Geometric Lemma~\ref{lemma::geometric} exists for $B$-bounded persistent Betti numbers. The addition of a positive simplex can add more than one dimension to $Z_{q,B}$. Consider $Z_q$ consisting of a single cycle with diameter above $B$ but below $2B$, meaning $Z_{q,B}=\set{0}$ initially. Now let the loop be split in two by a new simplex $\sigma$. Each piece may now be of diameter less than $B$, unlike the original. In this way a single simplex can increase the dimension of $Z_{q,B}$ by two or more. The same is true for the negative simplices. Consider the same setup, but now extend each simplex towards a distant point $x$ in a cone. In this case we have $Z_{q,B}=B_{q,B}=\set{0}$ initially. The inclusion of the simplex $\sigma\cup\set{x}$ will split the boundary space just as before into two bounded pieces.

Thus it becomes clear that we must utilize slightly different techniques when considering $B$-bounded persistence. We have the following inequality, the analog of the Geometric Lemma for $B$-bounded persistent Betti numbers.

\pagebreak

\begin{lemma} \label{lemma::geometric_bounded}
	Let $\mathcal J=\set{J^r}_{r\in\reals}$ and $\mathcal K=\set{K^r}_{r\in\reals}$ be filtrations of simplicial complexes with $J^r\subseteq K^r$ for all $r\in\reals$. Then
	\begin{align}
		\abs{\beta_{q,B}^{r,s}\p{\mathcal K}-\beta_{q,B}^{r,s}\p{\mathcal J}}  \leq \ & \max\set{\dim\p{\frac{Z_{q,B}\p{K^r}}{Z_{q,B}\p{J^r}}}, \dim\p{\frac{B_{q,B}\p{K^s}}{B_{q,B}\p{J^s}}}} \\
		\leq \ &\dim\p{\frac{Z_{q,B}\p{K^r}}{Z_{q,B}\p{J^r}}} + \dim\p{\frac{B_{q,B}\p{K^s}}{B_{q,B}\p{J^s}}}.
	\end{align}
\end{lemma}

\begin{proof}		
	\begin{align*}
		&\abs{\beta_{q,B}^{r,s}\p{\mathcal K}-\beta_{q,B}^{r,s}\p{\mathcal J}}  \\
		= \ & \abs{\dim\p{\frac{Z_{q,B}\p{K^r}}{Z_{q,B}\p{K^r}\cap B_{q,B}\p{K^s}}} - \dim\p{\frac{Z_{q,B}\p{J^r}}{Z_{q,B}\p{J^r}\cap B_{q,B}\p{J^s}}}} \\
		= \ & \abs{\dim\p{\frac{Z_{q,B}\p{K^r}+ B_{q,B}\p{K^r}}{Z_{q,B}\p{J^r}+ B_{q,B}\p{K^r}}}-\dim\p{\frac{Z_{q,B}\p{J^r}\cap B_{q,B}\p{K^s}}{Z_{q,B}\p{J^r}\cap B_{q,B}\p{J^s}}}} \\
		\leq \ & \max\set{\dim\p{\frac{Z_{q,B}\p{K^r}+ B_{q,B}\p{K^r}}{Z_{q,B}\p{J^r}+ B_{q,B}\p{K^r}}}, \dim\p{\frac{Z_{q,B}\p{J^r}\cap B_{q,B}\p{K^s}}{Z_{q,B}\p{J^r}\cap B_{q,B}\p{J^s}}}} \\
		\leq \ & \max\set{\dim\p{\frac{Z_{q,B}\p{K^r}}{Z_{q,B}\p{J^r}}}, \dim\p{\frac{B_{q,B}\p{K^s}}{B_{q,B}\p{J^s}}}} \\
		\leq \ & \dim\p{\frac{Z_{q,B}\p{K^r}}{Z_{q,B}\p{J^r}}} + \dim\p{\frac{B_{q,B}\p{K^s}}{B_{q,B}\p{J^s}}}.
	\end{align*}
\end{proof}

We make a note here about the difference between Lemma~\ref{lemma::geometric_bounded} and the Geometric Lemma~\ref{lemma::geometric}. While drawn from the same fundamental inequality, in the persistent Betti number case, we reduce to counting the simplices that are added when moving from one complex to the other. This reduction cannot be made in the $B$-bounded case, and we must count the number of additional linearly independent loops and boundaries. Different combinatorial techniques will be needed when applying each lemma, as can be seen in the proofs of Corollaries~\ref{theorem::bootstrap_pbn}, \ref{theorem::bootstrap_pbn_ball}, and \ref{theorem::bootstrap_bounded_pbn}.

\subsection{Stabilization Results}

We define the $q$-truncated Euler characteristics as
\begin{equation}
	\chi_q\p{K}:=\sum\limits_{k=0}^q\p{-1}^k \#\set{K_k}.
\end{equation}

We have the following stabilization lemmas for $B$-bounded persistent Betti numbers and $q$-truncated Euler characteristics. Since in both Lemmas~\ref{lemma::stabilization_bounded_pbn_increasing} and \ref{lemma::stabilization_bounded_pbn_ball} the radius of stabilization is a deterministic constant, \eqref{condition::stabilization_radius} is satisfied for any distribution $G$. The same is true in the following results for the truncated Euler characteristics.

\begin{lemma} \label{lemma::stabilization_bounded_pbn_increasing}
	Let $\mathcal K$ satisfy \eqref{condition::complex_increasing}. Then for any $B\geq 0$, $r\in \reals$, $s\in\reals$, $q\in \nats_0$, and $z\in\reals^d$, $\rho_z=2B$ is a locally determined radius of stabilization for $\beta_{q, B}^{r, s}\p{\mathcal K}$ centered at $z$.
\end{lemma}

\begin{proof} \label{proof::stabilization_bounded_pbn_increasing}
	Let $a\geq 2B$ and $S\in\mathcal X\p{\reals^d}$. We decompose $Z_{q, B}\p{K^r\p{\p{S\cap B_z\p{a}}\cup \set{z}}}$ into three spaces. Let $U_z$ be spanned by the generators of $Z_{q, B}\p{K^r\p{\p{S\cap B_z\p{a}}\cup \set{z}}}$ with $z$ as a vertex. Let $U_a$ be spanned by the generators with vertices within $B_z\p{a}\setminus B_z\p{2B}$. Finally, let $U_*$ be spanned by the generators without $z$ as a vertex and with no vertices within $B_z\p{a}\setminus B_z\p{2B}$. Since the generators of $Z_{q, B}\p{K^r\p{\p{S\cap B_z\p{a}}\cup \set{z}}}$ have diameter at most $B$, there are no generating cycles with vertices both at $z$ and in $B_z\p{a}\setminus B_z\p{2B}$.
	
	By \eqref{condition::complex_increasing} we have $Z_{q, B}\p{K^r\p{\p{S\cap B_z\p{a}}\cup \set{z}}}=U_z+ U_a+ U_*$, $Z_{q, B}\p{K^r\p{S\cap B_z\p{a}}} =U_a+ U_*$, $Z_{q, B}\p{K^r\p{\p{S\cap B_z\p{2B}}\cup \set{z}}}=U_z+ U_*$, and $Z_{q, B}\p{K^r\p{S\cap B_z\p{2B}}}=U_*$.
	
	Now for any cycle within $U_z$, the associated vertex set must lie within $B_z\p{B}$. Likewise, for any cycle in $U_a$, the associated vertex set must lie within $B_z\p{a}\setminus B_z\p{B}$. These vertex sets cannot intersect, thus $U_z\cap U_a=\set{0}$.
	
	Now, consider any vector spaces $X$, $Y$, and $Z$ such that $X\cap Y=\set{0}$. Because $X\cap Y\cap Z$ is a subspace of $X\cap Y$, it is also the trivial space $\set{0}$. We have 
	\begin{align}
		& \dim\p{X+Y+Z}-\dim\p{Y+Z} \\
		= \ & \dim\p{X}-\dim\p{X\cap Y}-\dim\p{X\cap Z}+\dim\p{X\cap Y\cap Z} \nonumber\\
		= \ & \dim\p{X}-\dim\p{X\cap Z} \nonumber\\
		= \ & \dim\p{X+Z}-\dim\p{Z}.
	\end{align}
	
	We use this result in each of the following. We have
	\begin{align}
		& \dim\p{Z_{q, B}\p{K^r\p{\p{S\cap B_z\p{a}}\cup \set{z}}}}-\dim\p{Z_{q, B}\p{K^r\p{S\cap B_z\p{a}}}} \\
		= \ & \dim\p{U_z+ U_a+ U_*}-\dim\p{U_a+ U_*} \nonumber\\
		= \ & \dim\p{U_z+ U_*}-\dim\p{U_*}\nonumber\\
		= \ & \dim\p{Z_{q, B}\p{K^r\p{\p{S\cap B_z\p{2B}}\cup \set{z}}}}-\dim\p{Z_{q, B}\p{K^r\p{S\cap B_0\p{2B}}}}.
	\end{align}
	
	A similar result holds for the boundaries. Let $V_z$, $V_a$, and $V_*$ be defined similarly to $U_z$, $U_a$, and $U_*$, respectively, instead using the generators of $B_{q, B}\p{K^s\p{\p{S\cap B_z\p{a}}\cup \set{z}}}$. Similarly $B_{q, B}\p{K^s\p{\p{S\cap B_z\p{a}}\cup \set{z}}}=V_z+V_a+V_*$, $B_{q, B}\p{K^s\p{\p{S\cap B_z\p{2B}}\cup \set{z}}}=V_z+V_*$, and we conclude $V_z\cap V_a=\set{0}$. Furthermore, we have $U_z\cap V_a= V_z\cap U_a=\set{0}$ by similar vertex-based arguments. Then $\p{U_z+V_z}\cap\p{U_a+V_a}=\set{0}$. We have
	\begin{align*}
		& \dim\p{Z_{q, B}\p{K^r\p{\p{S\cap B_z\p{a}}\cup \set{z}}}\cap B_{q, B}\p{K^s\p{\p{S\cap B_z\p{a}}\cup \set{z}}}} \\
		& - \dim\p{Z_{q, B}\p{K^r\p{S\cap B_z\p{a}}}\cap B_{q, B}\p{K^s\p{S\cap B_z\p{a}}}} \\
		& \\
		= \ & \dim\p{\p{U_z+U_a+U_*}\cap \p{V_z+V_a+V_*}} \\
		& - \dim\p{\p{U_a+U_*}\cap \p{V_a+V_*}} \\
		= \ & \dim\p{U_z+U_a+U_*}+\dim\p{V_z+V_a+V_*} - \dim\p{U_z+U_a+U_*+V_z+V_a+V_*}\\
		& - \dim\p{U_a+U_*}-\dim\p{V_a+V_*} + \dim\p{U_a+U_*+V_a+V_*} \\
		= \ & \dim\p{U_z+U_*} - \dim\p{U_*} + \dim\p{V_z+V_*} - \dim\p{V_*} \\
		& - \dim\p{U_z+U_*+V_z+V_*} + \dim\p{U_*+V_*} \\
		=\ & \dim\p{\p{U_z+U_*}\cap\p{V_z+V_*}} - \dim\p{U_*\cap V_*} \\
		= \ & \dim\p{Z_{q, B}\p{K^r\p{\p{S\cap B_z\p{2B}}\cup \set{z}}}\cap B_{q, B}\p{K^s\p{\p{S\cap B_z\p{2B}}\cup \set{z}}}} \\
		& - \dim\p{Z_{q, B}\p{K^r\p{S\cap B_z\p{2B}}}\cap B_{q, B}\p{K^s\p{S\cap B_z\p{2B}}}}.
	\end{align*}
	
	Combining these pieces, the $B$-bounded persistent Betti numbers must stabilize after a constant radius of $\rho_z=2B$.
\end{proof}

\begin{lemma} \label{lemma::stabilization_bounded_pbn_ball}
	Let $\mathcal K$ satisfy \eqref{condition::complex_local_ball}. Then for any $B\geq 0$, $r\in \reals$, $s\in\reals$, $q\geq 0$, and $z\in\reals^d$, $\rho_z=2\max\set{\phi\p{r}, \phi\p{s}}+2B$ is a locally determined radius of stabilization for $\beta_{q, B}^{r, s}\p{\mathcal K}$ centered at $z$.
\end{lemma}

\begin{proof} \label{proof::stabilization_bounded_pbn_ball}
	Denote by $\phi:=\max\p{\phi\p{r}, \phi\p{s}}$. Let $S\in\mathcal X\p{\reals^d}$. Furthermore, let $T$ be any finite multiset of points in $\reals^d$ with $S\cap B_z\p{2\phi+2B}\subseteq T$ and $y\notin B_{z}\p{2\phi+2B}$. We have the following partition. Let $U_z$, $U_y$, and $U_*$, respectively, be spanned by the generators of $Z_{q, B}\p{K^r\p{T}}$ having: a simplex within $B_z\p{\phi}$, a simplex within $B_y\p{\phi}$, or neither. Let $U_z^*$ be spanned by the generators of $Z_{q, B}\p{K^r\p{T\cup\set{z}}}$ have a simplex in $B_z\p{\phi}$. Finally, let $U_y^*$ be spanned by the generators of $Z_{q, B}\p{K^r\p{T\cup\set{y}}}$ have a simplex in $B_y\p{\phi}$.
	
	By \eqref{condition::complex_local_ball} we have $Z_{q, B}\p{K^r\p{T}}=U_z+ U_*+ U_y$, $Z_{q, B}\p{K^r\p{T\cup \set{z}}}=U_z^*+ U_*+ U_y$, $Z_{q, B}\p{K^r\p{T\cup \set{y}}}=U_z+ U_*+ U_y^*$, and $Z_{q, B}\p{K^r\p{T\cup \set{z, y}}}=U_z^*+ U_*+ U_y^*$.
	
	Now for any cycle within $U_z$, the associated vertex set must lie within $B_z\p{\phi+B}$. Likewise, for any cycle in $U_y$, the associated vertex set must lie within $B_y\p{\phi+B}$. Because $\|y-z\|>2\phi+2B$, these vertex sets cannot intersect, thus $U_z\cap U_y=\set{0}$. Likewise $U_z\cap U_y^*=U_z^*\cap U_y=U_z^*\cap U_y^*=\set{0}$.
	
	Now, for any vector spaces $X$, $X^*$, $Y$, and $Z$ such that $X\cap Y^*=X^*\cap Y^*=\set{0}$, we have
	\begin{align*}
		& \dim\p{X^*+Y^*+Z}-\dim\p{X+Y^*+Z} \\
		= \ & \dim\p{X}-\dim\p{X^*}+\dim\p{X\cap Y^*}+\dim\p{X\cap Z}-\dim\p{X^*\cap Y^*}-\dim\p{X^*\cap Z} \\
		& +\dim\p{X^*\cap Y^*\cap Z} - \dim\p{X\cap Y^*\cap Z}\\
		= \ & \dim\p{X}-\dim\p{X^*}+\dim\p{X\cap Z}-\dim\p{X^*\cap Z} \\
		= \ & \dim\p{X+Z}-\dim\p{X^*+Z}.
	\end{align*}
	
	Thus we have
	\begin{align*}
		& \dim\p{Z_{q, B}\p{K^r\p{T\cup \set{z, y}}}}-\dim\p{Z_{q, B}\p{K^r\p{T\cup \set{y}}}} \\
		= \ & \dim\p{U_z^*+ U_*+ U_y^*}-\dim\p{U_z+U_*+ U_y^*} \\
		= \ & \dim\p{U_z^*+ U_*}-\dim\p{U_z+U_*}\\
		= \ & \dim\p{U_z^*+ U_*+U_y}-\dim\p{U_z+U_*+U_y}\\
		= \ & \dim\p{Z_{q, B}\p{K^r\p{T\cup \set{z}}}}-\dim\p{Z_{q, B}\p{K^r\p{T}}}.
	\end{align*}
	
	A similar result holds for the boundaries. Let $V_z$, $V_z^*$, $V_y$, $V_y^*$, and $V_*$ be defined similarly to $U_z$, $U_z^*$, $U_y$, $U_y^*$, and $U_*$, respectively, instead using the generators of $B_{q, B}\p{K^r\p{T}}$,  $B_{q, B}\p{K^r\p{T\cup\set{y}}}$, and  $B_{q, B}\p{K^r\p{T\cup \set{y}}}$. Similarly $B_{q, B}\p{K^s\p{T}}=V_z+V_*+V_y$, $B_{q, B}\p{K^s\p{T\cup \set{z}}}=V_z^*+V_*+V_y$, $B_{q, B}\p{K^s\p{T\cup \set{y}}}=V_z+V_*+V_y^*$, and $\mybrk B_{q, B}\p{K^s\p{T\cup \set{z,y}}}=V_z^*+V_*+V_y^*$. We conclude $V_z\cap V_y=V_z\cap V_y^*=V_z^*\cap V_y=V_z^*\cap V_y^*=\set{0}$. Furthermore, we have $U_z\cap V_y= U_z\cap V_y^*=U_z^*\cap V_y=U_z^*\cap V_y^*=\set{0}$ and $V_z\cap U_y= V_z\cap U_y^*=V_z^*\cap U_y=V_z^*\cap U_y^*=\set{0}$ by similar vertex-based arguments. Thus $\p{U_z+V_z}\cap\p{U_y+V_y}=\p{U_z+V_z}\cap\p{U_y^*+V_y^*}=\p{U_z^*+V_z^*}\cap\p{U_y+V_y}=\p{U_z^*+V_z^*}\cap\p{U_y^*+V_y^*}=\set{0}$. We have
	\begin{align*}
		& \dim\p{Z_{q, B}\p{K^r\p{T\cup \set{z,y}}}\cap B_{q, B}\p{K^s\p{T\cup \set{z,y}}}} \\
		& - \dim\p{Z_{q, B}\p{K^r\p{T\cup\set{y}}}\cap B_{q, B}\p{K^s\p{T\cup\set{y}}}} \\
		= \ & \dim\p{\p{U_z^*+U_*+U_y^*}\cap \p{V_z^*+V_*+V_y^*}} - \dim\p{\p{U_z+U_*+U_y^*}\cap \p{V_z+V_*+V_y^*}} \\
		= \ & \dim\p{U_z^*+U_*+U_y^*}+\dim\p{V_z^*+V_*+V_y^*} - \dim\p{U_z^*+U_*+U_y^*+V_z^*+V_*+V_y^*}\\
		& - \dim\p{U_z+U_*+U_y^*} - \dim\p{V_z+V_*+V_y^*} + \dim\p{U_z+U_*+U_y^*+V_z+V_*+V_y^*} \\
		= \ & \dim\p{U_z^*+U_*+U_y} - \dim\p{U_z+U_*+U_y} +\dim\p{V_z^*+V_*+V_y} - \dim\p{V_z+V_*+V_y} \\
		& - \dim\p{U_z^*+U_*+U_y+V_z^*+V_*+V_y} + \dim\p{U_z+U_*+U_y+V_z+V_*+V_y} \\
		=\ & \dim\p{\p{U_z^*+U_*+U_y}\cap\p{V_z^*+V_*+V_y}} - \dim\p{\p{U_z+U_*+U_y}\cap\p{V_z+V_*+V_y}} \\
		= \ & \dim\p{Z_{q, B}\p{K^r\p{T\cup \set{z}}}\cap B_{q, B}\p{K^s\p{T\cup \set{z}}}} \\
		& - \dim\p{Z_{q, B}\p{K^r\p{T}}\cap B_{q, B}\p{K^s\p{T}}}.
	\end{align*}
	
	Thus, the addition of $y$ to $T$ does not change the add-$z$ cost. We proceed inductively. Starting with $S\cap B_z\p{2\phi +2B}$, for any $a>2\phi+2B$, the finitely many points of $\p{S\cap B_z\p{a}}\setminus \p{S\cap B_z\p{2\phi+2B}}$ may be added one at a time, while leaving the add-$z$ cost unchanged. Thus, the $B$-bounded persistent Betti numbers must stabilize after a constant radius of $\rho_z=2\phi+2B$.
\end{proof}

\begin{lemma} \label{lemma::stabilization_euler_adj}
	Let $K$ satisfy \eqref{condition::complex_increasing} and \eqref{condition::complex_local_adj}. Then for any $z\in\reals^d$ and $q\geq 0$, $\rho_z=\phi$ is a locally determined radius of stabilization for $\chi_q\p{K}$ centered at $z$.
\end{lemma}

\begin{proof} \label{proof::stabilization_euler_adj}
	Let $a\geq\phi$. By \eqref{condition::complex_increasing} and \eqref{condition::complex_local_adj}, we can partition $K\p{\p{S\cap B_{z}\p{a}}\cup \set{z}}$ into the sets \begin{align}
		U:= \ & \set{\sigma\in K\p{\p{S\cap B_{z}\p{a}}\cup \set{z}}\st z\in \sigma} \\
		V:= \ &\set{\sigma\in K\p{\p{S\cap B_{z}\p{a}}\cup \set{z}}\st \sigma\subset B_{z}\p{\phi}\setminus \set{z}} \\
		W:= \ & \set{\sigma\in K\p{\p{S\cap B_{z}\p{a}}\cup \set{z}}\st \sigma\cap B_{z}\p{a}\setminus B_z\p{\phi}\neq \emptyset}
	\end{align}
	
	Condition \eqref{condition::complex_local_adj} gives that no simplices may simultaneously have $z$ as a vertex and intersect $B_{z}\p{a}\setminus B_z\p{\phi}$, thus $U$, $V$, and $W$ indeed partition $K\p{\p{S\cap B_{z}\p{a}}\cup \set{z}}$. Condition \eqref{condition::complex_increasing} gives that the addition of $\set{z}$ and $S\cap\p{B_z\p{a}\setminus B_z\p{\phi}}$ to $S\cap B_z\p{\phi}$ may only introduce simplices to $K\p{S\cap B_z\p{\phi}}$ with vertices somewhere within $S\cap\p{B_z\p{a}\setminus B_z\p{\phi}}\cup \set{z}$, and thus not included in $V$. Therefore, $V\subseteq K\p{S\cap B_z\p{\phi}}$. Furthermore, since $S\cap B_z\p{\phi}\subset \p{S\cap B_z\p{a}}\cup \set{z}$, $K\p{S\cap B_z\p{\phi}}\subseteq V$. Thus we have $V=K\p{S\cap B_z\p{\phi}}$. Using similar arguments, condition \eqref{condition::complex_increasing} also gives $K\p{\p{S\cap B_z\p{\phi}}\cup \set{z}}=U\cup V$ and $K\p{S\cap B_z\p{a}}=V\cup W$.
	
	For $U_k$, $V_k$, and $W_k$ denoting the set of $k$-simplices contained in $U$, $V$, and $W$, respectively, the add-$z$ cost for the $q$-truncated Euler characteristic becomes
	\begin{align}
		& \chi_q\p{K\p{\p{S\cap B_z\p{a}}\cup \set{z}}}-\chi_q\p{K\p{S\cap B_z\p{a}}} \\
		= \ & \sum_{k=0}^q \p{-1}^k\#\set{K_k\p{\p{S\cap B_z\p{a}}\cup \set{z}}} - \sum_{k=0}^q\p{-1}^k\#\set{K_k\p{S\cap B_z\p{a}}} \\
		= \ & \sum_{k=0}^q \p{-1}^k\p{\#\set{U_k}+\#\set{V_k}+\#\set{W_k}} - \sum_{k=0}^q\p{-1}^k\p{\#\set{V_k}+\#\set{W_k}} \\
		= \ & \sum_{k=0}^q \p{-1}^k\p{\#\set{U_k}+\#\set{V_k}} - \sum_{k=0}^q\p{-1}^k\#\set{V_k} \\
		= \ & \sum_{k=0}^q \p{-1}^k\#\set{K_k\p{\p{S\cap B_z\p{\phi}}\cup \set{z}}} - \sum_{k=0}^q\p{-1}^k\#\set{K_k\p{S\cap B_z\p{\phi}}} \\
		= \ & \chi_q\p{K\p{\p{S\cap B_z\p{\phi}}\cup \set{z}}}-\chi_q\p{K\p{S\cap B_z\p{\phi}}}.
	\end{align}
	
	We see that $\chi_q\p{K}$ stabilizes after a constant radius of $\rho_z=\phi$, thus the local-determination criterion is immediately satisfied.
\end{proof}

\begin{lemma} \label{lemma::stabilization_euler_ball}
	Let $K$ satisfy \eqref{condition::complex_local_ball}. Then for any $z\in\reals^d$ and $q\geq0$, $\rho_z=2\phi$ is a locally determined radius of stabilization for $\chi_q\p{K}$ centered at $z$.
\end{lemma}

\begin{proof} \label{proof::stabilization_euler_ball}
	Let $z\in\reals^d$ and $S\in\mathcal X\p{\reals^d}$. Furthermore, let $T$ be a finite multiset of points in $\reals^d$ such that $S\cap B_z\p{2\phi}\subseteq T$. Let $y\notin B_z\p{2\phi}$. Consider the partition
	\begin{align}
		U:= \ & \set{\sigma\in K\p{T}\st \sigma\subset B_z\p{\phi}} \\
		U^*:= \ & \set{\sigma\in K\p{T\cup\set{z}}\st \sigma\subset B_z\p{\phi}} \\
		V:= \ & \set{\sigma\in K\p{T}\st \sigma\subset B_y\p{\phi}} \\
		V^*:= \ & \set{\sigma\in K\p{T\cup\set{y}}\st \sigma\subset B_y\p{\phi}} \\
		W:= \ & \set{\sigma\in K\p{T}\st \sigma\nsubseteq B_z\p{\phi} \AND \sigma\nsubseteq B_y\p{\phi}}.
	\end{align}
	
	Condition \eqref{condition::complex_local_ball} limits the influence of a single additional point on the complex to the ball of radius $\phi$ around it. $B_z\p{\phi}\cap B_z\p{\phi}=\emptyset$ because $\|y-z\|>2\phi$. Thus we have $K\p{T}=U\cup W\cup V$, $K\p{T\cup\set{z}}=U^*\cup W\cup V$, $K\p{T\cup\set{y}}=U\cup W\cup V^*$, and $K\p{T\cup\set{y, z}}=U^*\cup W\cup V^*$.
	
	For $U_{k}$, $U_{k}^*$, $V_{k}$, $V_{k}^*$, and $W_k$ denoting the set of $k$-simplices contained in $U$, $U^*$, $V$, $V^*$, and $W$, respectively, the add-$z$ cost for the $q$-truncated Euler characteristic becomes
	\begin{align*}
		& \chi_q\p{K\p{T\cup \set{y,z}}}-\chi_q\p{K\p{T\cup \set{y}}} \\
		= \ & \sum_{k=0}^q \p{-1}^k\#\set{K_k\p{T\cup\set{y, z}}} - \sum_{k=0}^q\p{-1}^k\#\set{K_k\p{T\cup\set{y}}} \\
		= \ & \sum_{k=0}^q \p{-1}^k\p{\#\set{U_{k}^*}+\#\set{W_k}+\#\set{V_{k}^*}} - \sum_{k=0}^q\p{-1}^k\p{\#\set{V^*_k}+\#\set{W_k}} \\
		= \ & \sum_{k=0}^q \p{-1}^k\p{\#\set{U^*_k}+\#\set{W_k}+\#\set{V_k}} - \sum_{k=0}^q\p{-1}^k\p{\#\set{V_k}+\#\set{W_k}} \\
		= \ & \sum_{k=0}^q \p{-1}^k\#\set{K_k\p{T\cup\set{z}}} - \sum_{k=0}^q\p{-1}^k\#\set{K_k\p{T}} \\
		= \ & \chi_q\p{K\p{T\cup \set{z}}}-\chi_q\p{K\p{T}}.
	\end{align*}
	
	We see that the addition of $\set{y}$ does not change the add-$z$ cost. Starting with $T=S\cap B_z\p{\phi}$, for any radius $a>2\phi$, $S\cap\p{B_z\p{a}\setminus B_z\p{\phi}}$ consists of finitely many points, which may be added to $S\cap B_z\p{\phi}$ in succession while leaving the add-$z$ cost unchanged. We conclude that $\rho_z=2\phi$ is a radius of stabilization for $\chi_q\p{K}$, and is locally-determined by virtue of being constant.
\end{proof}

\subsection{Bootstrap Results}

Here we give bootstrap convergence results for the altered statistics defined in Appendix~\ref{section::bounded_persistent_betti_numbers}. For given vectors of birth and death times, $\vec r=\p{r_i}_{i=1}^k$ and $\vec s=\p{s_i}_{i=1}^k$, let $\beta_{q, B}^{\vec r, \vec s}=\p{\beta_{q,B}^{r_i, s_i}}_{i=1}^k$ denote the multivariate function whose components are the $B$-bounded persistent Betti numbers evaluated at each pair of birth and death times. Likewise, for a vector of filtration times $\vec{r}=\p{r_i}_{i=1}^k$, let $\chi_q^{\vec r}$ denote the multivariate function giving the $q$-truncated Euler characteristic at each time $r_i$, with $\chi^{\vec r}_q\p{\mathcal K}:=\p{\chi_q\p{K^{r_i}}}_{i=1}^k$.

The following apply for $F\in \mathcal P\p{\reals^d}$ with density $f$ such that $\|f\|_p<\infty$ for some $p>2$, as specified. $F$ and $\hat F_n$ are such that $\hat F_n$ has density $\hat f_n$, $\| \hat f_n-f\|_1 \to 0$, and $\| \hat f_n-f \|_p \rightarrow 0$ in probability (resp.\ $a.s.$). Let $\mathbf X_n=\set{X_i}_{i=1}^n\iid F$ and $\p{m_n}_{n\in\nats}$ such that $\lim_{n\rightarrow\infty}m_n=\infty$. $\mathbf X_{m_n}^*=\set{X_i^*}_{i=1}^{m_n}\iid \hat F_n\big|\mathbf X_n$ is a bootstrap sample and $G$ a multivariate distribution. Recalling the conclusion of Theorem~\ref{theorem::bootstrap_general}, for a multivariate statistic $\vec\psi$:
\begin{statement} \label{conclusion::bootstrap_appendix}
	\begin{equation*}
		\frac{1}{\sqrt{n}}\p{\vec{\psi}\p{\sqrt[d]{n}\mathbf X_n}-\E{\vec{\psi}\p{\sqrt[d]{n}\mathbf X_n}}}\overset{d}{\rightarrow}G 
	\end{equation*}
	\centering{if and only if}
	\begin{equation*}
		\frac{1}{\sqrt{{m_n}}}\p{\vec{\psi}\p{\sqrt[d]{{m_n}}\mathbf X_{m_n}^*}-\E{\vec{\psi}\p{\sqrt[d]{{m_n}}\mathbf X_{m_n}^*}\big| \mathbf X_n}  }\overset{d}{\rightarrow}G \text{ in probability (resp. a.s.)}.
	\end{equation*}
\end{statement}

\begin{corollary}\label{theorem::bootstrap_bounded_pbn}
	Let $q\geq 0$ and $p>2q+3$. Let $\mathcal K$ be a filtration of simplicial complexes satisfying \eqref{condition::complex_increasing}. Then for any given $\vec r$, $\vec s$, and $B>0$, Statement~\ref{conclusion::bootstrap_appendix} holds for $\beta_{q, B}^{\vec r, \vec s}$.
\end{corollary}

Note the difference in necessary conditions between Corollaries~\ref{theorem::bootstrap_pbn} and \ref{theorem::bootstrap_bounded_pbn}. Corollary~\ref{theorem::bootstrap_pbn} notably requires a translation-invariant simplicial complex, along with the elimination of small loops via \eqref{condition::complex_minimum}. Corollary~\ref{theorem::bootstrap_bounded_pbn} imposes relatively few assumptions on the underlying simplicial complex. As a general statement, it can be seen that the $B$-bounded persistent Betti numbers defined here are better behaved than the unbounded persistent Betti numbers. Furthermore, the $B$-bounded persistent Betti numbers allow for an explicit rate calculation for the 2-Wasserstein metric in Proposition~\ref{proposition::w2_general}, see Appendix~\ref{appendix::bootstrap_applied} for details. For the unbounded persistent Betti numbers, this rate is stated implicitly in terms of the unknown tail probability for the radius of stabilization.

\begin{proof} \label{proof::bootstrap_bounded_pbn}
	Let $\mathbf Y_n=\set{Y_i}_{i=1}^n$ be iid and $Y'$ an independent copy. For a given $q\geq 0$, $B\geq 0$, and $r,s\in\reals$, we will show that $B_{q, B}^{r,s}\p{\mathcal K}$ satisfies assumption \eqref{condition::local_count}.
	
	Applying Lemma~\ref{lemma::geometric_bounded}, we must bound above the number of linearly independent $B$-bounded $q$-cycles and $q$-boundaries added when $\set{\sqrt[d]{n}Y'}$ is included with the sample $\sqrt[d]{n}\mathbf Y_n$. We start by considering the cycles. By \eqref{condition::complex_increasing} the addition of $\sqrt[d]{n}Y'$ will only introduce simplices to the complex having $\sqrt[d]{n}Y'$ as a vertex. As such, any $B$-bounded cycles in $Z_q\p{K^r\p{\sqrt[d]{n}\p{\mathbf Y_n\cup\set{Y'}}}}$ not having $\sqrt[d]{n}Y'$ as a vertex must already be in $Z_q\p{K^r\p{\sqrt[d]{n}\mathbf Y_n}}$, and thus in $Z_{q, B}\p{K^r\p{\sqrt[d]{n}\mathbf Y_n}}$. Thus, we must only bound the possible number of linearly independent $B$-bounded cycles within $K^r\p{\sqrt[d]{n}\p{\mathbf \cup \set{Y'}}}$ which have $\sqrt[d]{n}Y'$ as a vertex.
	
	Let $I_n:=\sum_{i=1}^n\ind{\|Y_i-Y'\|\leq B/\sqrt[d]{n}}$ be the number of sample points falling within $B$ of $\sqrt[d]{n}$. 
	
	We will construct a worst-case scenario. For any simplicial complexes $J\subseteq K$, we have that $Z_{q, B}\p{J}\subseteq Z_{q, B}\p{K}$. The addition of more simplices to $K^r\p{\sqrt[d]{n}\p{\mathbf Y_n\cup \set{Y'}}}$ having $\sqrt[d]{n}Y'$ as a vertex may increase the dimension of $Z_{q,B}\p{K^r\p{\sqrt[d]{n}\p{\mathbf Y_n\cup \set{Y'}}}}$, but will not alter $Z_{q,B}\p{K^r\p{\sqrt[d]{n}\mathbf Y_n}}$. As a worst case, we assume $K^r$ is such that all possible simplices containing $\sqrt[d]{n}Y'$ are included. Thus, for any simplex $\sigma\in K^r\p{\sqrt[d]{n}\mathbf Y_n}$ such that $\sigma\subseteq \p{\sqrt[d]{n}\mathbf Y_n}\cap B_{\sqrt[d]{n}Y'}\p{B}$ and $\diam{\sigma}\leq B$, $\partial\p{\sigma\cup\set{\sqrt[d]{n}Y'}}$ has diameter at most $B$, contains $\sqrt[d]{n}Y'$ as a vertex, and is a cycle within $Z_r\p{K^r\p{\sqrt[d]{n}\p{\mathbf Y_n\cup\set{Y'}}}}$. Let
	\begin{equation*}
		U:=\set{\partial\p{\sigma\cup\set{\sqrt[d]{n}Y'}}\st \sigma\subseteq \p{\sqrt[d]{n}\mathbf Y_n}\cap B_{\sqrt[d]{n}Y'}\p{B}\AND \#\set{\sigma}=q+1}.
	\end{equation*}
	
	Now consider $x$ to be any cycle in $Z_q\p{K^r\p{\sqrt[d]{n}\p{\mathbf Y_n\cup \set{Y'}}}}$ with diameter at most $B$ and a vertex at $\sqrt[d]{n}Y'$. For every simplex $\sigma$ of $x$ not containing $\sqrt[d]{n}Y'$ as a vertex, we add $\partial\p{\sigma\cup\set{\sqrt[d]{n}Y'}}$ to $x$, $\partial\p{\sigma\cup\set{\sqrt[d]{n}Y'}}$ necessarily having diameter less than $B$. This operation cannot add any new vertices to $x$, and thus cannot increase the total cycle diameter. What remains after completing these additions is either $0$ or a cycle $x'$ whose simplices all contain $\sqrt[d]{n}Y'$ as a vertex, the latter being an impossibility. Thus, any $B$-bounded cycle in $Z_q\p{K^r\p{\sqrt[d]{n}\p{\mathbf Y_n\cup \set{Y'}}}}$ having a vertex at $\sqrt[d]{n}Y'$ can be written as a linear combination of $B$-bounded elements from $U$. For $I_n:=\sum_{i=1}^n\ind{\|Y_i-Y'\|\leq B/\sqrt[d]{n}}$, we arrive at a worst case bound of
	\begin{equation}
		\dim\p{\frac{Z_{q,B}\p{K^r\p{\sqrt[d]{n}\p{\mathbf Y_n\cup\set{Y'}}}}}{Z_{q, B}\p{K^r\p{\sqrt[d]{n}\mathbf Y_n}}}}\leq \#\set{U}=\binom{I_n}{q+1}.
	\end{equation}
	
	A similar argument for the boundaries yields
	\begin{equation}
		\dim\p{\frac{B_{q,B}\p{K^s\p{\sqrt[d]{n}\p{\mathbf Y_n\cup\set{Y'}}}}}{B_{q, B}\p{K^s\p{\sqrt[d]{n}\mathbf Y_n}}}}\leq \#\set{U}=\binom{I_n}{q+1}.
	\end{equation}
	
	For any $a>2$, via Lemma~\ref{lemma::geometric_bounded} we have
	\begin{align*}
		& \abs{\beta_{q, B}^{r, s}\p{\mathcal K\p{\sqrt[d]{n}\p{\mathbf Y_n\cup \set{Y'}}}}-\beta_{q, B}^{r, s}\p{\mathcal K\p{\sqrt[d]{n}\mathbf Y_n}}}^a \nonumber \\
		\leq \ & \max\set{\dim\p{\frac{Z_{q,B}\p{K^r\p{\sqrt[d]{n}\p{\mathbf Y_n\cup\set{Y'}}}}}{Z_{q, B}\p{K^r\p{\sqrt[d]{n}\mathbf Y_n}}}}, \dim\p{\frac{B_{q,B}\p{K^s\p{\sqrt[d]{n}\p{\mathbf Y_n\cup\set{Y'}}}}}{B_{q, B}\p{K^s\p{\sqrt[d]{n}\mathbf Y_n}}}}}^a \\
		\leq \ & \binom{I_n}{q+1}^a \nonumber \\
		\leq \ & \frac{1}{\p{\p{q+1}!}^{a}}I_n^{a\p{q+1}} \nonumber \\
		\leq \ & \frac{1}{\p{\p{q+1}!}^{a}}\p{I_n^{a\p{q+1}}+1}. \nonumber
	\end{align*}
	
	Here $R=B$, $U_a=1/\p{\p{q+1}!}^{a}$, and $u_a=a\p{q+1}$. \eqref{condition::expectation} is then satisfied via Lemma~\ref{lemma::condition_expectation}. \eqref{condition::stabilization_radius} is satisfied via Lemma~\ref{lemma::stabilization_bounded_pbn_increasing}, in this case with a constant radius of stabilization of $2B$. An application of Theorem~\ref{theorem::bootstrap_general} gives the desired result.
	
	In this case, given that the radius of stabilization is a known constant, an explicit rate for $\gamma_\epsilon$ in Proposition~\ref{proposition::w2_general} can be calculated. Details omitted, from the proof of Proposition~\ref{proposition::w2_general} we have $\delta_\epsilon=B^{d}\epsilon^{\frac{p-2}{p-1}}$ up to constant factors. For $p<\infty$, using $a=\p{p-1}/\p{q+1}$ we achieve an optimal rate for $\gamma_\epsilon$ of
	\begin{equation}
		O\p{B^{d\p{1-\frac{2q+2}{p-1}}}\p{1+B^{d\p{2q+2}}}\epsilon^{\frac{p-2}{p-1}\p{1-\frac{2q+2}{p-1}}}}.
	\end{equation}
	
	For $p=\infty$, using $a_\epsilon=2-\log\p{\delta_\epsilon}$ we achieve an optimal rate of
	\begin{equation}
		O\p{\epsilon B^{d\p{2q+3}}\p{\frac{-\log\p{B^d\epsilon}}{\log\p{-\log\p{B^d\epsilon}}}}^{2q+2}}.
	\end{equation}
\end{proof}

\begin{corollary} \label{theorem::bootstrap_euler}
	Let $q\geq 0$ and $p>2q+1$. Let $\mathcal K$ be a filtration of simplicial complexes satisfying \eqref{condition::complex_increasing} and \eqref{condition::complex_local_adj}. Then for any given $\vec r$, Statement~\ref{conclusion::bootstrap_appendix} holds for $\chi_q^{\vec r}$.
\end{corollary}

\begin{proof} \label{proof::bootstrap_euler}
	We will show that assumption \eqref{condition::local_count} is satisfied for $\psi=\chi^r_q\p{\mathcal K}:=\chi_q\p{K^r}$ given $r\in\reals$. Let $\mathbf Y_n=\set{Y_i}_{i=1}^n$ be an iid sample in $\reals^d$, with $Y'$ an independent copy.
	
	By \eqref{condition::complex_increasing}, and \eqref{condition::complex_local_adj} it suffices to count those simplices within $B_{\sqrt[d]{n}Y'}\p{\phi\p{r}}$ having a vertex at $\sqrt[d]{n}Y'$. Let $I_n=\sum_{i=1}^n\ind{\|Y_i-Y'\|\leq \phi\p{r}/\sqrt[d]{n}}$. For any $a>2$, we have
	\begin{align*}
		& \abs{\chi^r_q\p{\mathcal K\p{\sqrt[d]{n}\p{\mathbf Y_n\cup\set{Y'}}}}-\chi^r_q\p{\mathcal K\p{\sqrt[d]{n}\mathbf Y_n}}}^a \\
		= \ & \abs{\sum\limits_{k=0}^q\p{-1}^k\#\set{K^r_k\p{\sqrt[d]{n}\p{\mathbf Y_n\cup\set{Y'}}}}-\sum\limits_{k=0}^q\p{-1}^k\#\set{K^r_k\p{\sqrt[d]{n}\mathbf Y_n}}}^a \\
		= \ & \abs{\sum\limits_{k=0}^q\p{-1}^k\p{\#\set{K^r_k\p{\sqrt[d]{n}\p{\mathbf Y_n\cup\set{Y'}}}}-\#\set{K^r_k\p{\sqrt[d]{n}\mathbf Y_n}}}}^a \\
		= \ & \abs{\sum\limits_{k=0}^q\p{-1}^k\#\set{K^r_k\p{\sqrt[d]{n}\p{\mathbf Y_n\cup\set{Y'}}}\setminus K_k^r\p{\sqrt[d]{n}\mathbf Y_n}}}^a \\
		\leq \ & \p{\sum\limits_{k=0}^q \binom{I_n}{k}}^a \\
		\leq \ & \p{\sum\limits_{k=0}^q \frac{I_n^k}{k!}}^a \\
		\leq \ & \p{\sum\limits_{k=0}^q \frac{I_n^q}{k!}}^a \\
		\leq \ & \p{eI_n^q}^a \\
		\leq \ & e^a\p{1+I_n^{aq}}.
	\end{align*}
	
	Here $R=\phi\p{r}$, $U_a=e^a$, and $u_a=aq$, satisfying \eqref{condition::local_count}. \eqref{condition::expectation} then follows from Lemma~\ref{lemma::condition_expectation} for $p\geq qa+1>2q+1$. \eqref{condition::stabilization_radius} is satisfied via Lemma~\ref{lemma::stabilization_euler_adj} with a constant radius of stabilization $\phi\p{r}$. \eqref{condition::stabilization} is satisfied via Lemma~\ref{lemma::condition_stabilization}. An application of Theorem~\ref{theorem::bootstrap_general} gives the final result.
	
	For the rate in Proposition~\ref{proposition::w2_general}, for $p<\infty$ we have that $\delta_\epsilon=\epsilon^{\frac{p-2}{p-1}}$ up to constant factors. Using $a=\p{p-1}/q$ we achieve a final rate for $\gamma_\epsilon$ of
	\begin{equation}
		O\p{\epsilon^{{\frac{p-2}{p-1}}\p{1-\frac{2q}{p-1}}}}.
	\end{equation}
	
	For $p=\infty$, using $a=a_\epsilon=2-\log\p{\epsilon}$ we achieve a final rate of
	\begin{equation}
		O\p{\epsilon\p{\frac{-\log\p{\epsilon}}{\log\p{-\log\p{\epsilon}}}}^{2q}}.
	\end{equation}
\end{proof}

\begin{corollary} \label{theorem::bootstrap_euler_ball}
	Let $q\geq 0$ and $p>2q+3$. Let $\mathcal K$ be a filtration of simplicial complexes satisfying \eqref{condition::complex_local_ball}. Then for any given $\vec r$, Statement~\ref{conclusion::bootstrap_appendix} holds for $\chi_q^{\vec r}$.
\end{corollary}

\begin{proof} \label{proof::bootstrap_euler_ball}
	The proof follows exactly that of Corollary~\ref{theorem::bootstrap_euler}. Let $\mathbf Y_n=\set{Y_i}_{i=1}^n$ be an iid sample in $\reals^d$, with $Y'$ an independent copy.
	
	By \eqref{condition::complex_local_ball} it suffices to consider simplices within $B_{\sqrt[d]{n}Y'}\p{\phi\p{r}}$. Let
	\begin{equation*}
		I_n=\sum_{i=1}^n\ind{\|Y_i-Y'\|\leq \phi\p{r}/\sqrt[d]{n}}.
	\end{equation*}
	
	For any $a>2$, we have
	\begin{align*}
		& \chi^r_q\p{\mathcal K\p{\sqrt[d]{n}\p{\mathbf Y_n\cup\set{Y'}}}}-\chi_q^r\p{\mathcal K\p{\sqrt[d]{n}\mathbf Y_n}} \\
		= \ & \sum\limits_{k=0}^q\p{-1}^k\#\set{K^r_k\p{\sqrt[d]{n}\p{\mathbf Y_n\cup\set{Y'}}}}-\sum\limits_{k=0}^q\p{-1}^k\#\set{K^r_k\p{\sqrt[d]{n}\mathbf Y_n}} \\
		= \ & \sum\limits_{k=0}^q\p{-1}^k\p{\#\set{K^r_k\p{\sqrt[d]{n}\p{\mathbf Y_n\cup\set{Y'}}}}-\#\set{K^r_k\p{\sqrt[d]{n}\mathbf Y_n}}} \\
		= & \sum\limits_{k=0}^q\p{-1}^k\p{\#\set{K^r_k\p{\sqrt[d]{n}\p{\mathbf Y_n\cup\set{Y'}}}\setminus K^r_k\p{\sqrt[d]{n}\mathbf Y_n}}} \\
		&  - \sum\limits_{k=0}^q\p{-1}^k\p{\#\set{K^r_k\p{\sqrt[d]{n}\mathbf Y_n}\setminus K^r_k\p{\sqrt[d]{n}\p{\mathbf Y_n\cup\set{Y'}}}}}.
	\end{align*}
	
	Any simplices added by the inclusion of $\sqrt[d]{n}Y'$ may contain $\sqrt[d]{n}Y'$ as a vertex, and any removed simplices must only have vertices within $\sqrt[d]{n}\mathbf Y_n$. We bound the possible simplices in each dimension. Thus for any $a>2$
	\begin{align*}
		& \abs{\chi_q^r\p{\mathcal K\p{\sqrt[d]{n}\p{\mathbf Y_n\cup\set{Y'}}}}-\chi_q^r\p{\mathcal K\p{\sqrt[d]{n}\mathbf Y_n}}}^a \\
		\leq \ & \p{\sum\limits_{k=0}^q \binom{I_n}{k+1} + \sum\limits_{k=0}^{q} \binom{I_n+1}{k+1}}^a \\
		\leq \ & \p{2\sum\limits_{k=0}^{q} \binom{I_n+1}{k+1}}^a \\
		\leq \ & \p{2\sum\limits_{k=0}^q\frac{\p{I_n+1}^{k+1}}{\p{k+1}!}}^a \\
		\leq \ & \p{2\sum\limits_{k=0}^q\frac{\p{I_n+1}^{q+1}}{\p{k+1}!}}^a \\
		\leq \ & \p{2\p{e-1}}^a\p{I_n+1}^{a\p{q+1}}\\
		\leq \ & 2^{a\p{q+2}-1}\p{e-1}^a\p{1+I_n^{a\p{q+1}}}.
	\end{align*}
	
	Here $R=\phi\p{r}$, $U_a\leq 2^{a\p{q+2}-1}\p{e-1}^a$, and $u_a=a\p{q+1}$, satisfying \eqref{condition::local_count}. \eqref{condition::expectation} is then satisfied via Lemma~\ref{lemma::condition_expectation} for $p\geq a\p{q+1}+1>2q+3$. \eqref{condition::stabilization_radius} is satisfied via Lemma~\ref{lemma::stabilization_euler_adj} with a constant radius of stabilization $\phi\p{r}$. \eqref{condition::stabilization} is satisfied via Lemma~\ref{lemma::condition_stabilization}. An application of Theorem~\ref{theorem::bootstrap_general} gives the final result.
	
	For the rate in Proposition~\ref{proposition::w2_general}, for $p<\infty$ we have that $\delta_\epsilon=\epsilon^{\frac{p-2}{p-1}}$ up to constant factors. Using $a=\p{p-1}/\p{q+1}$ we achieve a final rate for $\gamma_\epsilon$ of
	\begin{equation}
		O\p{\epsilon^{{\frac{p-2}{p-1}}\p{1-\frac{2q+2}{p-1}}}}.
	\end{equation}
	
	For $p=\infty$, using $a=a_\epsilon=2-\log\p{\epsilon}$ we achieve a final rate of
	\begin{equation}
		O\p{\epsilon\p{\frac{-\log\p{\epsilon}}{\log\p{-\log\p{\epsilon}}}}^{2q+2}}.
	\end{equation}
\end{proof}

\section{Proofs of Main Results} \label{appendix::proofs}

\subsection{Necessary Inequalities} \label{appendix::inequalities}

Throughout these proofs, we will make ample use of the H\"older, Jensen, and Minkowsky inequalities, along with the following. For brevity, these inequalities may be used implicitly and in combination. For $m\in\nats$, $\set{x_i}_{i=1}^m\subset\reals$, and $k\geq 1$,
\begin{equation}
	\abs{\sum_{i=1}^m x_i}^k\leq m^{k-1}\p{\sum_{i=1}^m\abs{x_i}^k}.
\end{equation}
Likewise for $0\leq k\leq 1$
\begin{equation}
	\abs{\sum_{i=1}^m x_i}^k\leq \sum_{i=1}^m\abs{x_i}^k.
\end{equation}
Next, for any density $f$ and $1\leq j\leq k\leq\infty$
\begin{equation}
	\|f\|_j^j\leq \|f\|_k^{\p{j-1}\frac{k}{k-1}}.
\end{equation}
Finally, for any set $A\subseteq \reals^d$ with $\abs{A}$ the Lebesgue measure of $A$ and $k\geq 1$
\begin{equation}
	\abs{\int_{A}f\p{x}\de x}^k\leq \abs{A}^{k-1}\int_{A}\abs{f\p{x}}^k\de x.
\end{equation}

Furthermore, in each of the following, we use the simplified notation \begin{equation}
	H_n\p{\mathbf S, \mathbf T}=\psi\p{\sqrt[d]{n}\mathbf S} -\psi\p{\sqrt[d]{n}\mathbf T}
\end{equation}
for the change in the statistic $\psi$ when the underlying scaled point cloud is altered. In the multivariate case, given $\vec \psi=\p{\psi_j}_{j=1}^k$ we use the notation $\vec H_n\p{\mathbf S, \mathbf T}=\p{H_{n,j}\p{\mathbf S, \mathbf T}}_{j=1}^k$, where $H_{n,j}=\psi_j\p{\sqrt[d]n\mathbf S}-\psi_j\p{\sqrt[d]{n}\mathbf T}$.

\subsection{Proofs of Section~\ref{section::stabilization}}

\begin{proposition}[Proposition~\ref{proposition::almost_sure_in_probability}]
	For $\mathbf S$ a simple point process taking values in $\mathcal X\p{\reals^d}$, let $\psi$ stabilize on $\mathbf S$ almost surely. Then $\psi$ stabilizes on $\mathbf S$ in probability. 
\end{proposition}

\begin{proof}
	Let $\rho$ be a radius of stabilization satisfying Definition~\ref{definition::stabilization_almost_sure}. Likewise, let $D^\infty$ be a corresponding terminal addition cost. For any $\rho\p{\mathbf S}\leq l<\infty$, $D\p{\mathbf S\cap B_z\p{l}} =D\p{\mathbf S\cap B_z\p{\rho\p{\mathbf S}}} \mybrk =D^\infty\p{\mathbf S}$. Thus $\set{D\p{\mathbf S\cap B_z\p{l}}\neq D^\infty\p{\mathbf S}}\subseteq \set{\rho\p{\mathbf S}>l}$, and consequently $\mybrk \probout{D\p{\mathbf S\cap B_z\p{l}} \allowbreak \neq D^\infty\p{\mathbf S}}\leq\probout{\rho\p{\mathbf S}>l}\rightarrow 0$. We see that $\psi$ stabilizes in probability on $\mathbf S$ with terminal addition cost $D^\infty\p{\mathbf S}$.
\end{proof}

\begin{proposition}[Proposition~\ref{proposition::locally_determinined_minimum}]
	For $\mathcal R$ the space of locally-determined radii of stabilization for $\psi$ centered at $z\in\reals^d$, let $\rho^*\colon\mathcal X\p{\reals^d}\rightarrow \brk{0, \infty}$ such that $\rho^*\p{S}=\inf_{\rho\in \mathcal R}\rho\p{S}$. Then $\rho^*$ is a locally determined radius of stabilization for $\psi$ centered at $z$.
\end{proposition}

\begin{proof}
	If all possible radii are infinite, the result follows trivially. Else for $S, T\in \mathcal X\p{\reals^d}$ suppose $\rho^*\p{S}<\infty$ with $S \cap B_z\p{\rho^*\p{S}}=T \cap B_z\p{\rho^*\p{S}}$. Since $S$ and $T$ have no accumulation points, for any $\epsilon>0$ sufficiently small, we have $S \cap B_z\p{\rho^*\p{S}+\epsilon}=T \cap B_z\p{\rho^*\p{S}+\epsilon}$. There exists a locally determined radius of stabilization $\rho$ such that $\rho\p{S} \leq \rho^*\p{S}+\epsilon$. As $S \cap B_z\p{\rho^*\p{S}+\epsilon}=T \cap B_z\p{\rho^*\p{S}+\epsilon}$ with $\rho\p{S}\leq \rho^*\p{S}+\epsilon$, we have that $S \cap B_z\p{\rho\p{S}}=T \cap B_z\p{\rho\p{S}}$. Thus $\rho\p{S}=\rho\p{T}$ by the local-determination criterion. Then $\rho^*\p{T}\leq \rho\p{T}=\rho\p{S}\leq \rho^*\p{S}+\epsilon$. Since the choice of $\epsilon$ was arbitrary, we have $\rho^*\p{T}\leq \rho^*\p{S}$. Thus, $S \cap B_z\p{\rho^*\p{T}}=T \cap B_z\p{\rho^*\p{T}}$. By similar arguments, $\rho^*\p{S}\leq \rho^*\p{T}$. Combining, $\rho^*\p{S}=\rho^*\p{T}$ must hold, and the result follows.
\end{proof}

\subsection{Proofs of Section~\ref{section::technical_results}}

\begin{lemma}[Lemma~\ref{lemma::condition_expectation}]
	For $p>2$, let $\psi$ satisfy \eqref{condition::local_count} with $u_a\leq p-1$ for some $a>2$. Then for any $M<\infty$, $\psi$ satisfies \eqref{condition::expectation} for $\mathcal C_{p, M}\p{\reals^d}$.
\end{lemma}

\begin{proof} \label{proof::condition_expectation}
	Let $R>0$ and $a>2$ be as given such that $u_a\leq p-1$. Define $I_n := \mybrk \#\set{\mathbf Y_n\cap B_{Y'}\p{R/\sqrt[d]{n}}}=\#\set{\p{\sqrt[d]{n}\mathbf Y_n}\cap B_{\sqrt[d]{n}Y'}\p{R}}$. Conditional on $Y'$, $I_n$ follows a binomial distribution with expectation $n\int_{B_{Y'}\p{R/\sqrt[d]{n}}}g\p{y}\de y$, where $g$ is a density of $G$. By \eqref{condition::local_count}, we have that
	\begin{align}
		& \E{\abs{\psi\p{\sqrt[d]{n}\p{\mathbf Y_n\cup\set{Y'}}}-\psi\p{\sqrt[d]{n}\mathbf Y_n}}^{a}} \\
		\leq \ & \E{U_a\p{1+I_n^{u_a}}} \\
		\leq \ & U_a\p{1+\E{I_n^{u_a}}}.
	\end{align}
	
	Via Corollary~3 in \cite{Latala1997}, there is a universal constant $K$ such that the conditional $u_a$-th moment of $I_n$ is at most
	\begin{align}
		& \p{K\frac{u_a}{\log\p{u_a}}}^{u_a}\max\set{n\int_{B_{Y'}\p{\frac{R}{\sqrt[d]{n}}}}g\p{y}\de y , \p{n\int_{B_{Y'}\p{\frac{R}{\sqrt[d]{n}}}}g\p{y}\de y}^{u_a}} \nonumber \\
		\leq \ & \p{K\frac{u_a}{\log\p{u_a}}}^{u_a}\p{n\int_{B_{Y'}\p{\frac{R}{\sqrt[d]{n}}}}g\p{y}\de y + \p{n\int_{B_{Y'}\p{\frac{R}{\sqrt[d]{n}}}}g\p{y}\de y}^{u_a}}. \nonumber
	\end{align}
	
	Removing the conditioning on $Y'$, for $V_d$ the volume of a unit ball in $\reals^d$, we have
	\begin{align}
		& \int_{\reals^d} n\p{\int_{B_x\p{\frac{R}{\sqrt[d]{n}}}}g\p{y}\de y}g\p{x}\de x \nonumber \\
		= \ & \int_{\reals^d} \int_{B_0\p{R}}g\p{x+\frac{t}{\sqrt[d]{n}}}g\p{x}\de t\de x \nonumber \\
		= \ &  \int_{B_0\p{R}}\int_{\reals^d}g\p{x+\frac{t}{\sqrt[d]{n}}}g\p{x}\de x\de t \nonumber \\
		\leq \ & V_dR^d\|g\|_2^2 \nonumber \\
		\leq \ & V_dR^d\|g\|_p^{\frac{p}{p-1}} \nonumber \\
		\leq \ & V_dR^dM^{\frac{p}{p-1}} \nonumber
	\end{align}
	
	and
	\begin{align}
		& \int_{\reals^d} \p{n\int_{B_x\p{\frac{R}{\sqrt[d]{n}}}}g\p{y}\de y}^{u_a}g\p{x}\de x \nonumber \\
		= \ & \int_{\reals^d}\p{\int_{B_0\p{R}}g\p{x+\frac{t}{\sqrt[d]{n}}}\de t}^{u_a}g\p{x}\de x \nonumber \\
		\leq \ & \p{V_d R^d}^{u_a-1}\int_{\reals^d}\int_{B_0\p{R}}g\p{x+\frac{t}{\sqrt[d]{n}}}^{u_a}g\p{x}\de t\de x \nonumber \\
		= \ & \p{V_d R^d}^{u_a-1}\int_{B_0\p{R}}\int_{\reals^d}g\p{x+\frac{t}{\sqrt[d]{n}}}^{u_a}g\p{x}\de x\de t \nonumber \\
		\leq \ & \p{V_dR^d}^{u_a}\|g\|_{u_a+1}^{u_a+1} \nonumber \\
		\leq \ & \p{V_dR^d}^{u_a}\|g\|_p^{\frac{p}{p-1}u_a} \nonumber \\
		\leq \ & \p{V_dR^dM^{\frac{p}{p-1}}}^{u_a}.
	\end{align}
	
	Combining, we have
	\begin{align}
		& \E{\abs{\psi\p{\sqrt[d]{n}\p{\mathbf Y_n\cup\set{Y'}}}-\psi\p{\sqrt[d]{n}\mathbf Y_n}}^{a}} \\
		\leq \ & U_a\p{1+\p{K\frac{u_a}{\log\p{u_a}}}^{u_a}\p{V_dR^dM^{\frac{p}{p-1}} + \p{V_dR^dM^{\frac{p}{p-1}}}^{u_a}}}.
	\end{align}
	
	Since this bound does not depend on $G$ or $n$, \eqref{condition::expectation} is satisfied by $\psi$ for $\mathcal C_{p, M}\p{\reals^d}$.
\end{proof}

\begin{lemma}[Lemma~\ref{lemma::condition_stabilization}]
	Let $\psi$ satisfy \eqref{condition::stabilization_radius} for $F\in C_{p, M}\p{\reals^d}$. Then $\psi$ satisfies \eqref{condition::stabilization} for $C_{p, M}\p{\reals^d}$, $F$, $b=\p{p-2}/\p{d\p{p-1}}$, and any $\p{l_\epsilon}_{\epsilon>0}$ such that $\lim_{\epsilon\rightarrow 0}l_\epsilon\epsilon^{\p{p-2}/\p{d\p{p-1}}} =0$ and $\lim_{\epsilon\rightarrow0}l_\epsilon=\infty$.
\end{lemma}

\begin{proof} \label{proof::condition_stabilization}
	
	Let $\set{X_i}_{i\in\nats}\iid F$ with $X'\sim F$ an independent copy. Likewise, for $G\in\mathcal C_{p, M}\p{\reals^d}\cap B_F\p{\epsilon; d_{\text{TV}}}$, let $\set{Y_i}_{i\in\nats}\iid G$ with $Y'\sim G$ an independent copy. Denote $\mathbf X_n:=\set{X_i}_{i=1}^n$. As $d_{\text{TV}}\p{F, G}\leq \epsilon$, it may be assumed that $\set{\p{X_i, Y_i}}_{i\in\nats}$ are iid with $\prob{X_i\neq Y_i}\leq \epsilon$ for all $i\in\nats$.
	
	Let $\p{l_\epsilon}_{\epsilon>0}$ be such that $\lim_{\epsilon\rightarrow0}l_\epsilon\epsilon^{\p{p-2}/\p{d\p{p-1}}}=0$. Define the following sets:
	\begin{align}
		A_Y & :=\set{Y'=X'} \\
		B_{Y, l_\epsilon} & :=\set{\mathbf Y_n\cap B_{X'}\p{\frac{l_\epsilon}{\sqrt[d]{n}}} = \mathbf X_n\cap B_{X'}\p{\frac{l_\epsilon}{\sqrt[d]{n}}}} \\
		C_{l_\epsilon} & :=\set{\rho_{\sqrt[d]{n}X'}\p{\sqrt[d]{n}\mathbf X_n}\leq l_\epsilon}.
	\end{align}
	
	By the local-definition criterion, Definition~\ref{definition::local_determined}, we have the following inclusion:
	\begin{equation*}
		A_Y\cap B_{Y, l_\epsilon}\cap C_{X, l_\epsilon}\subseteq\set{\rho_{\sqrt[d]{n}Y'}\p{\sqrt[d]{n}\mathbf Y_n}\leq l_\epsilon}.
	\end{equation*}
	Then
	\begin{align}
		& \probout{\rho_{\sqrt[d]{n}Y'}\p{\sqrt[d]{n}\mathbf Y_n}> l_\epsilon} \nonumber\\
		\leq \ & \probout{A_Y^c\cup B_{Y, l_\epsilon}^c\cup C_{l_\epsilon}^c} \nonumber\\
		\leq \ & \prob{A_Y^c}+\prob{B_{Y, l_\epsilon}^c}+\probout{C_{l_\epsilon}^c}.
	\end{align}
	
	Bounding each piece, $\prob{A_Y^c}=\prob{X'\neq Y'}\leq \epsilon$. Likewise, by \eqref{condition::stabilization_radius} we have $\probout{C_{l_\epsilon}^c}=\probout{\rho_{\sqrt[d]{n}X'}\p{\sqrt[d]{n}\mathbf X_n}> l_\epsilon}\leq p_\epsilon$, with $p_\epsilon$ not depending on $G$ or $n$ such that $\lim_{\epsilon\rightarrow0}p_\epsilon=0$. It thus remains to be shown that $B_{Y, l_\epsilon}^c$ occurs with small probability, uniformly in $n$ and $G$.
	
	As in \eqref{proof::condition_stabilization::1} in the proof of Proposition~\ref{proposition::w2_general}, the probability that $\mathbf X_n$ and $\mathbf Y_n$ coincide within $B_{X'}\p{l_\epsilon/\sqrt[d]{n}}$ is at most $2M^{\frac{p}{p-1}}V_d{l_\epsilon}^d\epsilon^{\frac{p-2}{p-1}}$. Thus we have that $\prob{B_{Y, l_\epsilon}^c}\leq \mybrk 2M^{\frac{p}{p-1}}V_d{l_\epsilon}^d\epsilon^{\frac{p-2}{p-1}}$. The bound does not depend on $G$ or $n$, with
	\begin{equation}
		\lim_{\epsilon\rightarrow 0}2M^{\frac{p}{p-1}}V_dl_\epsilon^d\epsilon^{\frac{p-2}{p-1}}=2M^{\frac{p}{p-1}}V_d\p{\lim_{\epsilon\rightarrow0}l_\epsilon\epsilon^{\frac{p-2}{d\p{p-1}}}}^d=0.
	\end{equation}
	
	Finally, by the definition of a radius of stabilization we have that
	\begin{align}
		& \prob{D_{\sqrt[d]{n}Y'}\p{\p{\sqrt[d]{n}\mathbf Y_n}\cap B_{\sqrt[d]{n}}\p{l_\epsilon}}\neq D_{\sqrt[d]Y'}\p{\mathbf Y_n}} \\
		\leq \ & \probout{\rho_{\sqrt[d]{n}Y'}\p{\sqrt[d]{n}\mathbf Y_n}> l_\epsilon} \\
		\leq \ & \epsilon + p_\epsilon + 2M^{\frac{p}{p-1}}V_dl_\epsilon^d\epsilon^{\frac{p-2}{p-1}}.
	\end{align}
	
	Here the final quantity does not depend on $G$ or $n$, and goes to $0$ as $\epsilon\rightarrow 0$. Thus \eqref{condition::stabilization} is satisfied.
	
\end{proof}

\begin{lemma}[Lemma~\ref{lemma::stabilization_poisson}]
	Let $F\in C_{p, M}$ with $p>2$ and $M<\infty$. Let $\rho_0$ be a locally-determined radius of stabilization for $\psi$ centered at $0$. Suppose that for any given $a,b \in \p{0, \infty}$, and $\delta>0$, there exists an $L_{a, b, \delta}<\infty$ and a measurable set $A_{a, b, \delta}$ with $\mybrk \rho_0^{-1}\p{\left(L_{a, b, \delta}, \infty\right]}\subseteq A_{a, b, \delta}$ such that
	\begin{equation}
		\sup_{\lambda\in\brk{a, b}}\probout{\rho_0\p{P_{\lambda}}>L_{a,b,\delta}} \leq \sup_{\lambda\in\brk{a, b}}\prob{\mathbf P_{\lambda}\in A_{a, b, \delta}}\leq \delta.
	\end{equation}
	Then for any $\delta>0$ there exists an $n_\delta<\infty$ and $L_\delta<\infty$ such that
	\begin{equation}
		\sup\limits_{n\geq n_\delta}\probout{\rho_0\p{\mathbf X_n-X'}>L_\delta}\leq \delta.
	\end{equation}
\end{lemma}

\begin{proof}
	We consider $n\geq n_0$. Define two independent sets of random variables $\p{U_i}_{i=1}^\infty\iid F$ and $\p{U_i^*}_{i=1}^\infty\iid F$. For $N\sim\rpois{n}$, denote by $\mathbf P_n$ the Poisson process given by $\set{U_i}_{i=1}^N$, having intensity $nf$ over $\reals^d$. We will couple this Poisson process to $\mathbf X_n$. $\set{U_i}_{i=1}^{N\vee n}\cup \set{U_i^*}_{i=1}^{\p{n-N}^+}$ has the same distribution as $\mathbf X_n$, thus we assume that the two random variables are equal. For a given random variable $U_i$ or $U_i^*$ and $L>0$, the probability of falling within $B_{X'}\p{L/\sqrt[d]{n}}$ is bounded, as shown below. Applying the Cauchy-Schwartz inequality, we have
	\begin{align}
		& \int_{\reals^d}\int_{B_x\p{\frac{L}{\sqrt[d]{n}}}}f\p{y}f\p{x}\de y\de x \nonumber \\
		= \ & \int_{\reals^d}\int_{B_0\p{\frac{L}{\sqrt[d]{n}}}}f\p{x+t}f\p{x}\de t\de x \nonumber \\
		= \ & \int_{B_0\p{\frac{L}{\sqrt[d]{n}}}}\int_{\reals^d}f\p{x+t}f\p{x}\de x\de t \nonumber \\
		\leq \ & \frac{V_dL^d}{n}\|f\|_2^2 \nonumber \\
		\leq \ & \frac{V_dL^d}{n}\|f\|_p^{\frac{p}{p-1}} \nonumber\\
		\leq \ & \frac{V_dL^dM^{\frac{p}{p-1}}}{n}. \nonumber
	\end{align}
	
	The expected number of points within $B_{X'}\p{L/\sqrt[d]{n}}$ that contribute to $\mathbf P_n\triangle \mathbf X_n$ is then at most
	\begin{equation}
		\E{\abs{N-n}\frac{V_dL^dM^{\frac{p}{p-1}}}{n}}\leq\frac{V_dL^dM^{\frac{p}{p-1}}}{n}\sqrt{\Var{N}}\leq \frac{M^{\frac{p}{p-1}}V_dL^d}{\sqrt{n}}. \label{proof::stabilization_poisson::1}
	\end{equation}
	As the number of differing points is an integer-valued random variable, this expectation bounds the probability that $\mathbf X_n$ and $\mathbf P_n$ differ within $B_{X'}\p{L/\sqrt[d]{n}}$. For a fixed value of $L$ and sufficiently large $n$, the bound can be made arbitrarily small.
	
	Next, we will couple the Poisson process $\mathbf P_n$ with a conditionally homogeneous approximation. We construct the following coupling: Let $\mathbf T$ be a homogeneous Poisson process on $\reals^d\times \left[0, \infty\right)$ with unit intensity. The point process given by $\set{U_i\st \p{U_i, T_i}\in \mathbf T, T_i\leq nf\p{U_i}}$ is then a nonhomogeneous Poisson process with intensity $nf$. We can safely assume that this process equals $\mathbf P_n$. Define the point process $\mathbf H_n:=\set{U_i\st \p{U_i, T_i}\in\mathbf T\text{ and } T_i\leq nf\p{X'}}$.
	
	Conditional on $X'$, $\mathbf H_n$ is a homogeneous Poisson process with intensity $nf\p{X'}$. The number of observations within $B_{X'}\p{L/\sqrt[d]{n}}$ that contribute to $\mathbf P_n\triangle \mathbf H_n$ follows a Poisson distribution with rate parameter
	\begin{equation}
		\int_{B_{X'}\p{\frac{L}{\sqrt[d]{n}}}}\abs{nf\p{y}-nf\p{X'}}\de y
	\end{equation}
	Removing the conditioning on $X'$, the expected number is
	\begin{equation}
		\int_{\reals^d}\p{n\int_{B_{x}\p{\frac{L}{\sqrt[d]{n}}}}\abs{f\p{y}-f\p{x}}\de y}f\p{x}	\de x \label{proof::stabilization_poisson::2}
	\end{equation}
	
	As the expectation above is an upper bound for the probability that $\mathbf P_n$ and $\mathbf H_n$ fail to coincide within $B_{X'}\p{L/\sqrt[d]{n}}$, we show that this quantity can be made arbitrarily small. Consider $C$, the set of Lebesgue points of $f$. We have that $C^c$ has Lebesgue measure $0$ by the Lebesgue differentiation theorem. By the definition of a Lebesgue point, we may write
	\begin{equation}
		C=\bigcap\limits_{\gamma>0}\bigcup\limits_{\Delta>0}\bigcap\limits_{\delta\leq\Delta}\set{x\in\reals\st \frac{\int_{B_x\p{\delta}}\abs{f\p{y}-f\p{x}}\de y}{V_d\delta^d}\leq \gamma}
	\end{equation}
	
	Here $V_d$ denotes the volume of a unit ball in $\reals^d$. Now as $f$ is a density, it may be shown that $\int_{B_x\p{\delta}}\abs{f\p{y}-f\p{x}}\de y/V_d\delta^d$ is a jointly continuous function of $x$ and $\delta$, and therefore it is measurable. Via the continuity with respect to $\delta$, we need only consider rational $\delta\leq\Delta$, because the rationals are dense in the reals. Thus,
	\begin{equation*}
		C_{\Delta, \gamma}:=\bigcap\limits_{\delta\leq \Delta}\set{x\in\reals\st \frac{\int_{B_x\p{\delta}}\abs{f\p{y}-f\p{x}}\de y}{V_d\delta^d}\leq \gamma}
	\end{equation*}
	is a countable intersection of measurable sets. Finally, by the Archimedean principle and other standard calculus arguments, we may assume $\gamma$ and $\Delta$ also come from a countable class, $\set{1/n: n\in\nats}$, for example. Let $C_{\gamma}:=\cup_{\Delta>0}C_{\Delta, \gamma}$. We have that $C_{\delta, \gamma}$ and $C_\gamma$ are measurable with $\lim_{\Delta\rightarrow 0}C_{\delta, \gamma}^c=C_{\gamma}^c$ and $\lim_{\gamma\rightarrow 0}C_{ \gamma}^c=C^c$. By continuity of measure, the Lebesgue measure of $C_\gamma^c$ must go to $0$, as well for $\int_{C_\gamma^c}f\p{x}\de x$. We decompose the integral in \eqref{proof::stabilization_poisson::2} as follows. For any integer $1 < a \leq p-1$, an application of H\"older's inequality gives
	
	\begin{align}
		& \int_{\reals^d}\p{n\int_{B_{x}\p{\frac{L}{\sqrt[d]{n}}}}\abs{f\p{y}-f\p{x}}\de y}f\p{x}\de x \nonumber \\
		= \ & \int_{\reals^d}\p{n\int_{B_{x}\p{\frac{L}{\sqrt[d]{n}}}}\abs{f\p{y}-f\p{x}}\de y}f\p{x}\ind{x\in C_{\frac{L}{\sqrt[d]{n}},\gamma}}f\p{x}\de x \nonumber \\
		& + \int_{\reals^d}\p{n\int_{B_{x}\p{\frac{L}{\sqrt[d]{n}}}}\abs{f\p{y}-f\p{x}}\de y}f\p{x}\ind{x\in C_{\frac{L}{\sqrt[d]{n}},\gamma}^c}f\p{x}\de x \nonumber \\
		\leq \ & \gamma V_d L^d \label{proof::stabilization_poisson::3}\\
		& + \p{\int_{\reals^d}\p{\int_{B_{0}\p{L}}\abs{f\p{x+\frac{t}{\sqrt[d]{n}}}-f\p{x}} \de t}^a f\p{x}\de x}^{\frac1a}\prob{X'\in C_{\frac{L}{\sqrt[d]{n}},\gamma}^c}^{1-\frac1a}\nonumber.
	\end{align}
	
	For the integral above
	\begin{align}
		& \int_{\reals^d}\p{\int_{B_{0}\p{L}}\abs{f\p{x+\frac{t}{\sqrt[d]{n}}}-f\p{x}}	\de t}^a f\p{x}\de x \nonumber \\
		\leq & \p{V_dL^d}^{a-1}\int_{\reals^d}\int_{B_{0}\p{L}}\abs{f\p{x+\frac{t}{\sqrt[d]{n}}}-f\p{x}}^a	f\p{x}\de t\de x \nonumber \\
		= \ & \p{V_dL^d}^{a-1}\int_{B_{0}\p{L}}\int_{\reals^d}\abs{f\p{x+\frac{t}{\sqrt[d]{n}}}-f\p{x}}^a	f\p{x}\de t\de x \nonumber \\
		\leq \ & 2^{a-1}\p{V_dL^d}^{a-1}\int_{B_{0}\p{L}}\int_{\reals^d}\p{f\p{x+\frac{t}{\sqrt[d]{n}}}^a+f\p{x}^a}f\p{x}\de t\de x \nonumber \\
		\leq \ & \p{2V_dL^d}^a\|f\|_{a+1}^{a+1} \nonumber \\
		\leq \ & \p{2V_dL^d}^a\|f\|_p^{\frac{p}{p-1}a} \nonumber \\
		\leq \ & \p{2V_dL^dM^{\frac{p}{p-1}}}^a. \nonumber
	\end{align}
	
	Thus \eqref{proof::stabilization_poisson::3} is at most
	\begin{align}
		& \gamma V_d L^d + 2V_dL^dM^{\frac{p}{p-1}}\prob{X'\in C_{\frac{L}{\sqrt[d]{n}},\gamma}^c}^{1-\frac1a} \\
		\leq \ &\gamma V_d L^d + 2V_dL^dM^{\frac{p}{p-1}}\prob{X'\in C_{\frac{L}{\sqrt[d]{n}},\gamma}^c}^{\frac{p-2}{p-1}}. \label{proof::stabilization_poisson::4}
	\end{align}
	This provides a bound for the probability that $\mathbf P_n$ and $\mathbf H_n$ fail to coincide within $B_{X'}\p{L/\sqrt[d]{n}}$. The bound holds in the limiting $p=\infty$ case and can be made arbitrarily small for $\gamma$ sufficiently small and $n$ sufficiently large. $\gamma$ can be chosen as a function of $F$, $L$, and $n$ to provide the tightest bound, but this requires specific knowledge of $f$. Combining with the previous steps, we have coupled $\mathbf X_n$ and $\mathbf H_n$ to be equal with arbitrarily high probability.
	
	Now for $\eta, \zeta>0$ define $D_{*,\eta}=f^{-1}\p{\left[\eta, \infty\right)}$ and $D^*_{\zeta}=f^{-1}\p{\left[0, \zeta\right]}$. For $\eta$ sufficiently small and $\zeta$ sufficiently large, $\prob{X'\in D_{*,\eta}^c}$ and $\prob{X'\in {D^*_{\zeta}}^c}$ can be made arbitrarily small.
	
	By assumption, for any given $\eta$, $\zeta$, and $\nu>0$ there is an $L_{\eta, \zeta, \nu}$ and a measurable set $A_{\eta,\zeta,\nu}$ such that for any homogenous Poisson process $\mathbf Q_\lambda$ on $\reals^d$ with intensity $\lambda$ bounded between $\eta$ and $\zeta$, be have $\p{\rho_0}^{-1}\p{\left(L_{\eta, \zeta, \nu}, \infty\right]}\subseteq A_{\eta,\zeta,\nu}$ and $\probout{\rho_0\p{\mathbf Q_\lambda}> L_{\eta,\zeta,\nu}}\leq \prob{\mathbf Q_\lambda\in A_{\eta,\zeta,\nu}}\leq\nu$. $L_{\eta,\zeta,\nu}$ is possibly increasing as $\eta\rightarrow 0$, $\zeta\rightarrow\infty$, and $\nu\rightarrow 0$.
	
	As $\sqrt[d]{n}\p{\mathbf H_n-X'}$ is a homogeneous Poisson process, conditional on $X'$, $\in D_{*, \eta}\cup D^*_\zeta$, we have
	\begin{align}
		& \probout{\rho_0\p{\sqrt[d]{n}\p{\mathbf H_n-X'}}> L_{\eta,\zeta,\nu}\big| \ X'\in D_{*, \eta}\cup D^*_\zeta} \\
		\leq \ & \prob{\sqrt[d]{n}\p{\mathbf H_n-X'}\in A_{\eta,\zeta,\nu}\big| \ X'\in D_{*, \eta}\cup D^*_\zeta} \\
		= \ & \E{\prob{\sqrt[d]{n}\p{\mathbf H_n-X'}\in A_{\eta,\zeta,\nu}|X'}\big| \ X'\in D_{*, \eta}\cup D^*_\zeta} \\
		\leq \ & \nu.
	\end{align}
	
	Combining the pieces and letting $L=L_{\eta,\zeta,\nu}$, we have that
	\begin{align}
		& \probout{\rho_0\p{\mathbf X_n-X'}>L_{\eta,\zeta,\nu}} \\
		= \ & \probout{\rho_0\p{\mathbf H_n-X'}>L_{\eta,\zeta,\nu}|X'\in A_{*, \eta}\cup A^*_\zeta}\prob{X'\in D_{*, \eta}\cup D^*_\zeta} \\
		& + \prob{\mathbf X_n \cap B_{X'}\p{L_{\eta,\zeta,\nu}}\neq \mathbf P_n \cap B_{X'}\p{L_{\eta,\zeta,\nu}}} \\
		& + \prob{\mathbf P_n \cap B_{X'}\p{L_{\eta,\zeta,\nu}}\neq \mathbf H_n \cap B_{X'}\p{L_{\eta,\zeta,\nu}}} + \prob{X'\in D_{*,\eta}^c} + \prob{X'\in {D^*_\zeta}^c} \\
		\leq \ & \nu + \frac{M^{\frac{p}{p-1}}V_dL_{\eta,\zeta,\nu}^d}{\sqrt{n}} + \gamma V_d L_{\eta,\zeta,\nu}^d + 2V_dL_{\eta,\zeta,\nu}^dM^{\frac{p}{p-1}}\prob{X'\in C_{\frac{L_{\eta,\zeta,\nu}}{\sqrt[d]{n}},\gamma}^c}^{\frac{p-2}{p-1}} \\
		& + \prob{X'\in D_{*,\eta}^c} + \prob{X'\in {D^*_\zeta}^c}.
	\end{align}
	
	As $\eta, \nu \rightarrow 0$ and $\zeta\rightarrow\infty$, $L_{\eta,\zeta,\nu}$ can become unbounded. Let $\gamma\rightarrow 0$ and choose $n_0$ suitably large to ensure that the entire expression goes to $0$. The result follows.
	
\end{proof}

\pagebreak

\begin{proposition}[Proposition~\ref{proposition::w2_general}]
	For $p>2$ and $M<\infty$, let $\psi$ satisfy \eqref{condition::expectation} and \eqref{condition::stabilization} for $\mathcal C_{p, M}\p{\reals^d}$, $F\in \mathcal C_{p, M}\p{\reals^d}$, and some $a>2$. Then for any $G\in\mathcal C_{p, M}\p{\reals^d}\cap B_F\p{\epsilon, d_{\text{TV}}}$, there exist iid coupled random variables $\p{\p{X_i, Y_i}}_{i\in\nats}$ such that $\mathbf X_n=\set{X_i}_{i=1}^n\iid F$, $\mathbf Y_n=\set{Y_i}_{i=1}^n\iid G$, and
	\begin{equation}
		\sup_{n\in \nats} \Var{\frac{1}{\sqrt{n}}\p{\psi\p{\sqrt[d]{n}\mathbf X_n}-\psi\p{\sqrt[d]{n}\mathbf Y_n}}}\leq \gamma_{\epsilon}.
	\end{equation}
	The value $\gamma_\epsilon$ does not depend on $G$ and satisfies $\lim_{\epsilon\rightarrow 0}\gamma_\epsilon=0$.
\end{proposition}

\begin{proof} \label{proof::w2_general}
	We refer to Appendix~\ref{appendix::inequalities} for a reference list of the general inequalities used here. Our proof technique is inspired by that of Proposition~5.4 in \cite{Krebs2019}. We expand using a martingale difference sequence (MDS). Let $\set{\p{X_i, Y_i}}_{i=1}^\infty$ be iid such that $\set{X_i}_{i=1}^\infty\iid F$ and $\set{Y_i}_{i=1}^\infty\iid G$. Each pair $\p{X_i, Y_i}$ can be identically coupled such that $\prob{X_i\neq Y_i}=d_{TV}\p{F, G} \leq \epsilon$. For $\mathbf X_j:=\set{X_i}_{i=1}^j$, $\mathbf Y_j:=\set{Y_i}_{i=1}^j$, and $\sigma$ denoting a generated sigma algebra, let $\mathcal F_j:=\sigma\set{\mathbf X_j, \mathbf Y_j}$ with $\mathcal F_0:=\set{\Omega, \emptyset}$. For $\p{X', Y'}$ an independent copy of the $\p{X_i, Y_i}$, let
	\begin{align*}
		\mathbf X_{n,j}'&:=\set{X_1,..., X_{j-1}, X', X_{j+1}, ..., X_n}\\
		\mathbf Y_{n,j}'&:=\set{Y_1,..., Y_{j-1}, Y', Y_{j+1}, ..., Y_n}.
	\end{align*}
	
	We apply the condensed notation $H_n\p{\mathbf S, \mathbf T}=\psi\p{\sqrt[d]{n}\mathbf S} -\psi\p{\sqrt[d]{n}\mathbf T}$. Using the orthogonality of a MDS and the conditional version of Jensen's inequality,
	\begin{align}
		& \Var{\frac{1}{\sqrt{n}}H_n\p{\mathbf X_n, \mathbf Y_n}} \label{proof::w2_general::1} \\
		= \ & \frac1n\E{\abs{\sum_{j=1}^n \E{H_n\p{\mathbf X_n, \mathbf Y_n}|\mathcal F_j}-\E{H_n\p{\mathbf X_n, \mathbf Y_n}\big|\mathcal F_{j-1}}}^2} \nonumber \\
		= \ & \frac1n\E{\abs{\sum_{j=1}^n \E{H_n\p{\mathbf X_n, \mathbf Y_n} - H_n\p{\mathbf X_{n,j}', \mathbf Y_{n,j}'}\big|\mathcal F_j}}^2} \nonumber \\
		= \ & \frac1n\sum_{j=1}^n \E{\E{H_n\p{\mathbf X_n, \mathbf Y_n} - H_n\p{\mathbf X_{n,j}', \mathbf Y_{n,j}'}\big|\mathcal F_j}^2} \nonumber \\
		\leq \ & \E{\abs{H_n\p{\mathbf X_n, \mathbf Y_n} - H_n\p{\mathbf X_{n,j}', \mathbf Y_{n,j}'}}^2}. \label{proof::w2_general::3} 	
	\end{align}
	
	The above holds for any $1\leq j\leq n$. We have an upper bound for \eqref{proof::w2_general::3} of
	\begin{align}
		& 2\E{\abs{H_n\p{\mathbf X_n\cup X', \mathbf X_n} - H_n\p{\mathbf Y_n\cup Y', \mathbf Y_n}}^2} \nonumber \\
		& + 2\E{\abs{H_n\p{\mathbf X_n\cup X', \mathbf X_{n,j}'} - H_n\p{\mathbf Y_n\cup Y', \mathbf Y_{n,j}'}}^2} \nonumber \\
		=\ & 4\E{\abs{H_n\p{\mathbf X_n\cup X', \mathbf X_n} - H_n\p{\mathbf Y_n\cup Y', \mathbf Y_n}}^2}. \label{proof::w2_general::2}
	\end{align}
	
	\pagebreak
	
	We will decompose the expectation in \eqref{proof::w2_general::2} using the stabilization of $\psi$. Let $L>0$, and define the following sets. Note that when all four are satisfied, $H_n\p{\mathbf X_n\cup X', \mathbf X_n}=H_n\p{\mathbf Y_n\cup Y', \mathbf Y_n}$. 
	\begin{align}
		A_Y & :=\set{Y'=X'} \\
		B_{Y, L} & :=\set{\mathbf Y_n\cap B_{X'}\p{\frac{L}{\sqrt[d]{n}}} = \mathbf X_n\cap B_{X'}\p{\frac{L}{\sqrt[d]{n}}}} \\
		C_{X, L} & :=\set{D^\infty_{\sqrt[d]{n}X'}\p{\p{\sqrt[d]{n}\mathbf X_n}\cap B_{\sqrt[d]{n}X'}\p{L}}=D_{\sqrt[d]{n}X'}\p{\sqrt[d]{n}\mathbf X_n}} \\
		C_{Y, L} & :=\set{D^\infty_{\sqrt[d]{n}Y'}\p{\p{\sqrt[d]{n}\mathbf Y_n}\cap B_{\sqrt[d]{n}Y'}\p{L}}=D_{\sqrt[d]{n}Y'}\p{\sqrt[d]{n}\mathbf Y_n}}.
	\end{align}
	
	Let $C_{X, L*}\subseteq C_{X, L}$ and $C_{Y, L*}\subseteq C_{Y, L}$ be measurable. We decompose the expectation in \eqref{proof::w2_general::2} along these events into
	
	\begin{align}
		& \E{\abs{H_n\p{\mathbf X_n\cup X', \mathbf X_n} - H_n\p{\mathbf Y_n\cup Y', \mathbf Y_n}}^2\ind{A_Y\cap B_{Y, L}\cap C_{X, L*}\cap C_{Y, L*}}} \nonumber \\
		& \ + \E{\abs{H_n\p{\mathbf X_n\cup X', \mathbf X_n} - H_n\p{\mathbf Y_n\cup Y', \mathbf Y_n}}^2\ind{A_Y^c\cup B_{Y, L}^c\cup C_{X,L*}^c\cup C_{Y,L*}^c}} \nonumber \\
		= \ & \E{\abs{H_n\p{\mathbf X_n\cup X', \mathbf X_n} - H_n\p{\mathbf Y_n\cup Y', \mathbf Y_n}}^2\ind{A_Y^c\cup B_{Y, L}^c\cup C_{X,L*}^c\cup C_{Y,L*}^c}}. \nonumber
	\end{align}
	
	Let $a>2$ satisfy \eqref{condition::expectation}. H\"older's inequality gives an upper bound of
	\begin{align}
		\norm[\frac{a}{2}]{\abs{H_n\p{\mathbf X_n\cup X', \mathbf X_n} - H_n\p{\mathbf Y_n\cup Y', \mathbf Y_n}}^2}\prob{A_Y^c\cup B_{Y, L}^c\cup	C_{X,L*}^c\cup C_{Y,L*}^c}^{1-\frac{2}{a}}. \nonumber
	\end{align}
	
	As the choice of $C_{X, L*}$ and $C_{Y, L*}$ was arbitrary, the expectation in \eqref{proof::w2_general::2} is at most
	\begin{align}
		& \norm[\frac{a}{2}]{\abs{H_n\p{\mathbf X_n\cup X', \mathbf X_n} - H_n\p{\mathbf Y_n\cup Y', \mathbf Y_n}}^{2}}\probout{A_Y^c\cup B_{Y, L}^c\cup	C_{X,L}^c\cup C_{Y,L}^c}^{1-\frac{2}{a}} \nonumber \\
		\leq & \norm[\frac{a}{2}]{\abs{H_n\p{\mathbf X_n\cup X', \mathbf X_n} - H_n\p{\mathbf Y_n\cup Y', \mathbf Y_n}}^{2}} \nonumber \\
		& \times\max\set{\prob{A_Y^c}+\prob{B_{Y, L}^c}+\probout{C_{X,L}^c}+\probout{C_{Y,L}^c}, 1}^{1-\frac{2}{a}}. \nonumber
	\end{align}
	
	Consider the norm in the final expression above. We have an upper bound of
	\begin{align}
		& 2\p{\norm[\frac{a}{2}]{H_n\p{\mathbf X_n\cup X', \mathbf X_n}^2} + \norm[\frac{a}{2}]{H_n\p{\mathbf Y_n\cup Y', \mathbf Y_n}^2}} \leq 4E_a^{\frac{2}{a}}.
	\end{align}
	
	This final quantity does not depend on $\epsilon$, $G$, or $n$. It thus remains to show that, for a certain choice of $L$ and as $\epsilon\rightarrow 0$, that each of the events $A_Y^c$, $B_{Y, L}^c$, $C_{X,L}^c$, and $C_{Y,L}^c$ can be made to occur with small outer probability, uniformly in $G$ and $n$. For $A^c$, this is satisfied because $\prob{X'\neq Y'}\leq\epsilon$.
	
	We then consider $B_{Y, L}^c$. The sample pairs which contribute to $\mathbf X_n\cap B_{X'}\p{L/\sqrt[d]{n}}$ but not $\mathbf Y_n\cap B_{X'}\p{L/\sqrt[d]{n}}$ are those $\p{X_i, Y_i}$ for which $X_i\neq Y_i$ and either $\|X_i-X'\|\leq L/\sqrt[d]{n}$ or $\|Y_i-X'\|\leq L/{\sqrt[d]{n}}$. Conditional on $X'$, their count follows a binomial distribution with expectation at most $n\prob{X_i\neq Y_i}\int_{B_{X'}\p{L/\sqrt[d]{n}}}\tilde f\p{y}+\tilde{g}\p{y}\de y$. Here $\tilde f$ and $\tilde g$ are the densities of $X_i$ and $Y_i$ conditional on the event $\set{X_i\neq Y_i}$. These densities can be shown to exist via the absolute continuity of $F$ and $G$ with respect to the Lebesgue measure on $\reals^d$. Subsequently, we have that $\|\tilde f\|_p\leq \|f\|_p/\prob{X_i\neq Y_i}\leq M/\prob{X_i\neq Y_i}$ and $\|\tilde g\|_p\leq M/\prob{X_i\neq Y_i}$. Removing the conditioning on $X'$, via H\"older's inequality the expected number of pairs which contribute to $\mathbf X_n\triangle \mathbf Y_n$ within $B_{X'}\p{L/\sqrt[d]{n}}$ is at most
	\begin{align}
		& \int_{\reals^d}\p{n\prob{X_i\neq Y_i}\int_{B_{x}\p{\frac{L}{\sqrt[d]{n}}}}\tilde f\p{y}+\tilde g\p{y}\de y}f\p{x}\de x  \nonumber \\
		= \ & \prob{X_i\neq Y_i}\int_{\reals^d}\int_{B_0\p{L}}\p{\tilde f\p{x+\frac{t}{\sqrt[d]{n}}}+\tilde g\p{x+\frac{t}{\sqrt[d]{n}}}}f\p{x}\de t\de x \nonumber \\
		= \ & \prob{X_i\neq Y_i}\int_{B_0\p{L}}\int_{\reals^d}\p{\tilde f\p{x+\frac{t}{\sqrt[d]{n}}}+\tilde g\p{x+\frac{t}{\sqrt[d]{n}}}}f\p{x}\de x\de t \nonumber \\
		\leq \ & \prob{X_i\neq Y_i}V_d L^d\p{\|\tilde f\|_{\frac{p}{p-1}}+\|\tilde g\|_{\frac{p}{p-1}}}\|f\|_p \nonumber \\
		\leq \ & 2\prob{X_i\neq Y_i}MV_d L^d\p{\frac{M}{\prob{X_i\neq Y_i}}}^{\frac{1}{p-1}} \nonumber \\
		\leq \ & 2M^{\frac{p}{p-1}}V_dL^d\epsilon^{\frac{p-2}{p-1}}. \label{proof::condition_stabilization::1}
	\end{align}
	
	This final expression provides an upper bound on $\prob{B_{Y, L}^c}$. Let $\p{l_\epsilon}_{\epsilon>0}$ satisfy \eqref{condition::stabilization} and $L=l_\epsilon$. We have that $\prob{B_{Y, l_\epsilon}^c}\leq 2M^{\frac{p}{p-1}}V_d{l_\epsilon}^d\epsilon^{\frac{p-2}{p-1}}\rightarrow 0$. By \eqref{condition::stabilization}, both $\probout{C_{X, l_\epsilon}^c}$ and $\probout{C_{Y, l_\epsilon}^c}$ are bounded above by a quantity $p_\epsilon$ such that $\lim_{\epsilon\rightarrow 0}p_\epsilon=0$. 
	
	Let $\delta_\epsilon=\min\set{\epsilon+2M^{\frac{p}{p-1}}V_d{l_\epsilon}^d\epsilon^{\frac{p-2}{p-1}}+2p_\epsilon, 1}$ be the derived upper bound for $\mybrk \probout{A_Y^c\cup B_{Y, {l_\epsilon}}^c\cup C_{X, l_\epsilon}^c\cup C_{Y, l_\epsilon}^c}$. We achieve a final upper bound for \eqref{proof::w2_general::1} of
	\begin{equation}
		16E_a^{\frac{2}{a}}\delta_\epsilon^{1-\frac{2}{a}}.
	\end{equation}
	
	If \eqref{condition::stabilization} is satisfied for many $\p{l_\epsilon}_{\epsilon>0}$ such that $\lim_{\epsilon\rightarrow0}l_\epsilon\epsilon^{\p{p-2}/\p{d\p{p-1}}}=0$, $l_\epsilon$ can be further chosen to optimize the rate of $\delta_\epsilon$, provided a rate for $p_\epsilon$. Furthermore, if \eqref{condition::expectation} is satisfied for more than one $a>2$, $a=a_\epsilon$ may be chosen to optimize the final rate as $\epsilon\rightarrow 0$. Such considerations depend on the specifics of the statistic $\psi$ and the density assumptions used.
\end{proof}

\pagebreak

\subsection{Proofs of Section~\ref{section::bootstrap}}

\begin{theorem}[Theorem~\ref{theorem::bootstrap_general}]
	Let $F\in \mathcal P\p{\reals^d}$ with density $f$ such that $\|f\|_p<\infty$ for some $p>2$. Furthermore, let $F$ and $\hat f_n$ be such that $\| \hat f_n-f \|_1 \to 0$ and $\|\hat f_n-f \|_p \to 0$ in probability (resp. a.s.). Suppose $\vec{\psi}\colon \tilde{\mathcal X}\p{\reals^d}\rightarrow\reals^k$ has component functions $\psi_j\colon \tilde{\mathcal X}\p{\reals^d}\rightarrow\reals$, $1\leq j\leq k$ satisfying \eqref{condition::expectation} and \eqref{condition::stabilization} for $\mathcal C_{p, M}\p{\reals^d}$, $M>\|f\|_p$, $F$, and $b=\p{p-2}/\p{d\p{p-1}}$. Then for a sample $\mathbf X_n=\set{X_i}_{i=1}^n\iid F$, $\p{m_n}_{n\in\nats}$ such that $\lim_{n\rightarrow \infty }m_n=\infty$, a bootstrap sample $\mathbf X_{m_n}^*=\set{X_i^*}_{i=1}^{m_n}\iid \hat F_n|\mathbf X_n$, and a multivariate distribution $G$,
	\begin{equation*}
		\frac{1}{\sqrt{n}}\p{\vec{\psi}\p{\sqrt[d]{n}\mathbf X_n}-\E{\vec{\psi}\p{\sqrt[d]{n}\mathbf X_n}}}\overset{d}{\rightarrow}G
	\end{equation*}
	\centering{if and only if}
	\begin{equation*}
		\frac{1}{\sqrt{{m_n}}}\p{\vec{\psi}\p{\sqrt[d]{{m_n}}\mathbf X_{m_n}^*}-\E{\vec{\psi}\p{\sqrt[d]{{m_n}}\mathbf X_{m_n}^*}\big| \mathbf X_n}}\overset{d}{\rightarrow}G \text{ in probability (resp. a.s.)}.
	\end{equation*}
\end{theorem}

\begin{proof} \label{proof::bootstrap_general}
	For any bounded, Lipschitz function $v\colon\reals^k\rightarrow\reals$, consider the functional given by $V_{m_n}:=\E{v\p{\vec{H}_{m_n}\p{\mathbf Y_{m_n}}}}$, where $\mathbf Y_{m_n}=\set{Y_i}_{i=1}^{m_n}$ is an iid sample, and the functional takes as input the shared distribution of the $Y_i$. Let $v$ be bounded within $\brk{-L,L}$ with a Lipschitz constant of $L$. First assuming that $\vec{H}_n\p{\mathbf X_n}\overset{d}{\rightarrow} G$, we have $V_n\p{F}\rightarrow \int_{\reals} v\de G$.
	
	Now, let $\mathbf X_{m_n}'=\set{X_{{m_n},i}'}_{i=1}^{m_n}\iid F$ be an independent copy of $\mathbf X_{m_n}=\set{X_i}_{i=1}^{m_n}$. Furthermore, as in the proof of Proposition~\ref{proposition::w2_general}, $\mathbf X_{m_n}'$ and $\mathbf X_{m_n}^*$ can be coupled so that $\mybrk \prob{X_{{m_n},i}'\neq X_i^*} =d_{TV}\p{F, \hat F_n}=\|\hat f_n-f\|_1/2$, conditional on $\mathbf X_n$. Via Proposition~\ref{proposition::w2_general} and Chebyshev's inequality, for $\delta>0$ we have almost surely that
	\begin{align}
		& V_{m_n}\p{\hat F_n} \nonumber\\
		= \ & \E{v\p{\vec{H}_{m_n}\p{\mathbf X_{m_n}^*}}\big|\mathbf X_n} \nonumber\\
		= \ & \E{v\p{\vec{H}_{m_n}\p{\mathbf X_{m_n}^*}}\ind{\|\vec{H}_{m_n}\p{\mathbf X_{m_n}^*}-\vec{H}_{m_n}\p{\mathbf X_{m_n}'}\|\leq \delta}\big|\mathbf X_n} \nonumber\\
		& + \E{v\p{\vec{H}_{m_n}\p{\mathbf X_{m_n}^*}}\ind{\|\vec{H}_{m_n}\p{\mathbf X_{m_n}^*}-\vec{H}_{m_n}\p{\mathbf X_{m_n}'}\|> \delta}\big|\mathbf X_n} \nonumber\\
		\leq \ & \E{v\p{\vec{H}_{m_n}\p{\mathbf X_{m_n}'}}+L\delta\big|\mathbf X_n} \nonumber\\
		& + L\p{2-\delta}\sum_{j=1}^{m_n}\prob{\abs{H_{{m_n}, j}\p{\mathbf X_{m_n}^*}-H_{{m_n}, j}\p{\mathbf X_{m_n}'}}> \frac{\delta}{\sqrt{k}}\big|\mathbf X_n} \nonumber\\
		\leq \ & \E{v\p{\vec{H}_{m_n}\p{\mathbf X_{m_n}'}}+L\delta\big|\mathbf X_n} \nonumber\\
		& + L\p{2-\delta}\p{\sum_{j=1}^k\prob{\abs{H_{{m_n}, j}\p{\mathbf X_{m_n}^*}-H_{{m_n}, j}\p{\mathbf X_{m_n}'}}> \frac{\delta}{\sqrt{k}}\big|\mathbf X_n}}\ind{\|\hat f_n\|_p\leq M} \nonumber\\
		& +L\p{2-\delta}\ind{\|\hat f_n\|_p> M} \nonumber\\
		\leq \ & \E{v\p{\vec{H}_{m_n}\p{\mathbf X_{m_n}'}}}+L\delta \nonumber\\
		& + L\p{2-\delta}\p{\p{\sum_{j=1}^k\frac{k\gamma_{\|\hat f_n-f\|_1/2, j}}{\delta^2}}\ind{\|\hat f_n\|_p\leq M}+ \ind{\|\hat f_n\|_p> M}}\nonumber.
	\end{align}
	
	Here $\gamma_{\|\hat f_n-f\|_1/2, j}$ is as given in Proposition~\ref{proposition::w2_general} applied to $\psi_j$ for $\epsilon=\|\hat f_n-f\|_1/2$. Similarly, almost surely
	\begin{align}
		& V_{m_n}\p{\hat F_n} \\
		\geq \ & \E{v\p{\vec{H}_{m_n}\p{\mathbf X_{m_n}'}}}-L\delta \\
		& - L\p{2-\delta}\p{\p{\sum_{i=1}^k\frac{k\gamma_{\|\hat f_n-f\|_1/2, j}}{\delta^2}}\ind{\|\hat f_n\|_p\leq M}+ \ind{\|\hat f_n\|_p> M}}.
	\end{align}
	
	As $\|\hat f_n-f\|_p\rightarrow 0$ and $M>\|f\|_p$, we have that the lower bound for $V_{m_n}\p{\hat F_n}$ converges to $\int_{\reals}v\de G - L\delta$ and the upper bound converges to $\int_{\reals}v\de G + L\delta$, either in probability or a.s., depending on assumptions. Since this holds for any $\delta>0$, we have that $V_{m_n}\p{\hat F_n}\rightarrow \int_{\reals}v\de G$ in probability (or a.s.).
	
	Now we will show the converse direction. Let $\mathbf X_{m_n}^*$ and $\mathbf X_{m_n}'$ be as previously defined. We have
	\begin{align}
		& V_{m_n}\p{F} \nonumber\\
		= \ & \E{\E{v\p{\vec{H}_{m_n}\p{\mathbf X_{m_n}'}}\big|\mathbf X_n}} \nonumber\\
		\leq \ & \E{\E{v\p{\vec{H}_{m_n}\p{\mathbf X_{m_n}^*}}\big|\mathbf X_n}}+L\delta \nonumber\\
		& + L\p{2-\delta}\E{\min\set{\sum_{i=1}^k\frac{k\gamma_{\|\hat f_n-f\|_1/2, j}}{\delta^2}, 1}\ind{\|\hat f_n\|_p\leq M}+ \ind{\|\hat f_n|_p> M}}\nonumber
	\end{align}
	and
	\begin{align}
		& V_{m_n}\p{F} \nonumber\\
		\geq \ & \E{\E{v\p{\vec{H}_{m_n}\p{\mathbf X_{m_n}^*}}\big|\mathbf X_n}}-L\delta \nonumber\\
		& - L\p{2-\delta}\E{\min\set{\sum_{i=1}^k\frac{k\gamma_{\|\hat f_n-f\|_1/2, j}}{\delta^2}, 1}\ind{\|\hat f_n\|_p\leq M}+ \ind{\|\hat f_n\|_p> M}}\nonumber.
	\end{align}
	
	Each expectation involves only bounded random variables, thus the lower bound converges to $\int_\reals v\de G-L\delta$ and the upper bound to $\int_\reals v\de G+L\delta$, assuming $\E{v\p{\vec{H}_{m_n}\p{\mathbf X_{m_n}^*}}\big|\mathbf X_n}\rightarrow \int_{\reals}v\de G$. This holds if the assumed convergence is either in probability or almost sure. Since this holds for any $\delta>0$, we have $V_{m_n}\p{F}\rightarrow \int_\reals v\de G$. Since our initial choice of $v$ was arbitrary, the desired result follows.
\end{proof}

\subsection{Proofs of Section~\ref{section::stabilization_results}}

\begin{lemma}[Lemma~\ref{lemma::stabilization_pbn}]
	Let $F\in\mathcal C_{p, M}\p{\reals^d}$ for some $p>2$ and $M<\infty$, and let $\mathcal K=\set{K^r}_{r\in\reals}$ be a filtration of simplicial complexes satisfying \eqref{condition::complex_translate}, \eqref{condition::complex_local_ball}, and \eqref{condition::complex_minimum}. Then for any $r\in \reals$, $s\in\reals$, and $q\geq0$, $\beta_q^{r, s}\p{\mathcal K}$ satisfies \eqref{condition::stabilization_radius} for $F$.
\end{lemma}

\begin{proof} \label{proof::stabilization_pbn}
	We start by defining a crude locally-determined radius of stabilization. Let $K$ be either $K^r$ or $K^s$. Denote $\phi=\max\set{\phi\p{r}, \phi\p{s}}$ as given by \eqref{condition::complex_local_ball}. For $z\in\reals^d$, $S\in\mathcal X\p{\reals^d}$, and $a>\phi$, consider the connected components in $K\p{S\cap B_z\p{a}}$ and $K\p{\p{S\cap B_z\p{a}}\cup\set{z}}$ with at least one simplex entirely contained within $B_z\p{\phi}$. By \eqref{condition::complex_local_ball}, if these components are entirely contained within $B_z\p{a-\phi}$, no simplices will be added or removed from them within $K\p{S\cap B_z\p{b}}$ or $K\p{\p{S\cap B_z\p{b}}\cup\set{z}}$ for any $b>a$. This property holds for both $K^s$ or $K^r$. The persistent Betti numbers are additive with respect to connected components, thus the add-$z$ cost is entirely defined by those components altered by the inclusion of $z$, which necessarily must include one simplex within $B_z\p{\phi}$. As such, $a$ is a locally determined radius of stabilization for $S$ in this case. Any changes to the simplices outside of $a$ must contribute to different connected components, and thus do not influence the add-$z$ cost.
	
	Now, $\mathbf X_n$ contains $n$ total points. Including one point within $B_z\p{\phi}$, the longest possible chain of $n$ connected points reaches at most a radius of $n\phi$. Therefore, $\rho_{\sqrt[d]{n}X'}\p{\sqrt[d]{n}\mathbf X_n}=\p{n+1}\phi$ is a locally-determined radius of stabilization on $\sqrt[d]{n}\mathbf X_n$ centered at $\sqrt[d]{n}X'$, as shown in the previous paragraph. However, since this radius grows with $n$, it alone cannot provide for the desired result.
	
	Given \eqref{condition::complex_local_ball} and \eqref{condition::complex_minimum}, by Theorem~4.3 in \cite{Krebs2019} and the proof thereof, there exists a locally-determined radius of stabilization $\rho_0^*$ for $\beta_{q}^{r, s}\p{\mathcal K}$ centered at $0$ such that the conditions of Lemma~\ref{lemma::stabilization_poisson} are satisfied. It must be noted that the original statement of the referenced lemma does not give this result directly. However, a careful analysis of the provided proof yields this more general result with minimal additions, and is not restated here. By \eqref{condition::complex_translate}, we may define a radius of stabilization $\rho_z^*$ for $\beta_q^{r, s}$ centered at $z\in\reals^d$ with $\rho_z^*\p{S}=\rho_0^*\p{S-z}$. Thus, for any $\delta>0$, there exists an $L_\delta<\infty$ and $n_\delta<\infty$ such that $\prob{\rho^*_{\sqrt[d]{n}X'}\p{\sqrt[d]{n}\mathbf X_n}}\leq \delta$ for all $n\geq N_\delta$.
	
	Denote by $P_z\p{S}$ the union of all connected components in either $K\p{S}$ or $K\p{S\cup\set{0}}$ with at least one simplex entirely contained within $B_z\p{\phi}$. For any center point $z\in\reals^d$, define $\rho_z:\mathcal X\rightarrow \brk{0, \infty}$ with $\rho_z\p{S}=\min\set{\diam{P_z\p{S}}+\phi, \rho^*\p{S-z}}$. We have that $\rho_z$ is a locally-determined radius of stabilization.
	
	For $n< n_\delta$, we have that $\rho_{\sqrt[d]{n}X'}\p{\sqrt[d]{n}\mathbf X_n}\leq \p{n_\delta+1}\phi$ almost surely. For $n\geq n_\delta$, $\rho_{\sqrt[d]{n}X'}\p{\sqrt[d]{n}\mathbf X_n}\leq \rho^*_{\sqrt[d]{n}X'}\p{\sqrt[d]{n}\mathbf X_n}\leq L_{\delta}$ with probability at least $1-\delta$. Therefore $\mybrk \sup_{n\in\nats}\prob{\rho_{\sqrt[d]{n}X'}\p{\sqrt[d]{n}\mathbf X_n}> \max\set{L_\delta, \p{n_\delta+1}\phi}}\leq \delta$, and the result follows.
\end{proof}

\subsection{Proofs of Section~\ref{section::bootstrap_results}} \label{appendix::bootstrap_applied}

\begin{corollary}[Corollary~\ref{theorem::bootstrap_pbn}]
	Let $q\geq 0$ and $p>2q+3$. Let $\mathcal K$ be a filtration of simplicial complexes satisfying \eqref{condition::complex_increasing}, \eqref{condition::complex_translate}, \eqref{condition::complex_local_adj}, and \eqref{condition::complex_minimum}. Then for any given $\vec r$, $\vec s$, Statement~\ref{conclusion::bootstrap} holds for $\beta_q^{\vec r, \vec s}$.
\end{corollary}

\begin{proof} \label{proof::bootstrap_pbn}
	For given $r,s\in\reals$, we will verify that assumption \eqref{condition::local_count} is satisfied for $\psi=\beta_q^{r, s}\p{\mathcal K}$. Let $\mathbf Y_n=\set{Y_i}_{i=1}^n$ be iid and $Y'$ an independent copy. By the Geometric Lemma~\ref{lemma::geometric}, a bound for the change in persistent Betti numbers when $\set{\sqrt[d]{n}Y'}$ is added to $\sqrt[d]{n}\mathbf Y_n$ is given by the number of new simplices introduced to the corresponding complexes. By \eqref{condition::complex_increasing}, \eqref{condition::complex_local_adj}, it suffices to count the number of points within $\phi:=\max\set{\phi\p{r}, \phi\p{s}}$ of $\sqrt[d]{n}Y'$, the combinations of which include any possible new simplices. Let $I_n=\sum_{i=1}^n \ind{\|Y_i-Y'\|\leq \phi/\sqrt[d]{n}}$. For any $a>2$ we have
	\begin{align*}
		& \abs{\beta_{q}^{r, s}\p{\mathcal K\p{\sqrt[d]{n}\p{\mathbf Y_n\cup \set{Y'}}}}-\beta_{q}^{r, s}\p{\mathcal K\p{\sqrt[d]{n}\mathbf Y_n}}}^a \nonumber \\
		\leq \ & \big\lvert\#\set{K_q^r\p{\sqrt[d]{n}\p{\mathbf Y_n\cup \set{Y'}}}\setminus K_q^r\p{\sqrt[d]{n}\mathbf Y_n}} \\
		& +\#\set{K_{q+1}^s\p{\sqrt[d]{n}\p{\mathbf Y_n\cup \set{Y'}}}\setminus K_{q+1}^s\p{\sqrt[d]{n}\mathbf Y_n}}\big\rvert^a \\
		\leq \ & \abs{\binom{I_n}{q}+\binom{I_n}{q+1}}^a \\
		= \ & \binom{I_n+1}{q+1}^a \\
		\leq \ & \frac{1}{\p{\p{q+1}!}^a}\p{I_n+1}^{a\p{q+1}} \\
		\leq \ & \frac{2^{a\p{q+1}-1}}{\p{\p{q+1}!}^a}\p{I_n^{a\p{q+1}}+1}.
	\end{align*}
	
	In this case $R=\phi$, $U_a\leq 2^{a\p{q+1}-1}/\p{\p{q+1}!}^{a}$ and $u_a=a\p{q+1}$. \eqref{condition::expectation} then follows from Lemma~\ref{lemma::condition_expectation} for $p\geq a\p{q+1}+1>2q+3$. As \eqref{condition::complex_increasing} and \eqref{condition::complex_local_adj} together imply \eqref{condition::complex_local_ball}, \eqref{condition::stabilization_radius} is satisfied as shown in Lemma~\ref{lemma::stabilization_pbn}. Then \eqref{condition::stabilization} follows from Lemma~\ref{lemma::condition_stabilization}. Finally an application of Theorem~\ref{theorem::bootstrap_general} gives the desired result.
	
	Referring to Proposition~\ref{proposition::w2_general} and the proof thereof, for $p<\infty$, using $a=\p{p-1}/\p{q+1}$ we achieve an optimal rate for $\gamma_\epsilon$ of
	\begin{equation}
		O\p{\delta_\epsilon^{1-\frac{2q+2}{p-1}}}.
	\end{equation}
	
	Details of the calculation are omitted here. For $p=\infty$, using $a=a_\epsilon=2-\log\p{\delta_\epsilon}$ we achieve an optimal rate of
	\begin{equation}
		O\p{\delta_\epsilon \p{\frac{-\log\p{\delta_\epsilon}}{\log\p{-\log\p{\delta_\epsilon}}}}^{2q+2}}.
	\end{equation}
	
	Both of these rates depend on $\delta_\epsilon$, the upper bound for the total probability found in the proof of Proposition~\ref{proposition::w2_general}. The techniques found in the proofs of Lemma~\ref{lemma::condition_expectation} and Proposition~\ref{proposition::w2_general} allow for a bound on $\delta_\epsilon$, provided a tail bound for $\sup_{n\in\nats}\rho_0\p{\sqrt[d]{n}\p{\mathbf Y_n-Y'}}$. At this time, such a bound is unavailable, thus no explicit rate calculation is possible. 
\end{proof}

\begin{corollary}[Corollary~\ref{theorem::bootstrap_pbn_ball}]
	Let $q\geq 0$ and $p>2q+5$. Let $\mathcal K$ be a filtration of simplicial complexes satisfying \eqref{condition::complex_translate}, \eqref{condition::complex_local_ball}, and \eqref{condition::complex_minimum}. Then for any given $\vec r$, $\vec s$, Statement~\ref{conclusion::bootstrap} holds for $\beta_q^{\vec r, \vec s}$.
\end{corollary}

\begin{proof} \label{proof::bootstrap_pbn_ball}
	
	The proof follows exactly that of Corollary~\ref{theorem::bootstrap_pbn}, thus we will omit many replicated details. Let $\mathbf Y_n=\set{Y_i}_{i=1}^n$ be iid and $Y'$ an independent copy. Define $\phi:=\max\set{\phi\p{r}, \phi\p{s}}$. 
	
	Since we do not assume \eqref{condition::complex_increasing} in this case, the addition of $\sqrt[d]{n}Y'$ to the complex may both add and remove simplices, but only within  $B_{\sqrt[d]{n}Y'}\p{\phi}$ by \eqref{condition::complex_local_ball}. Any additional simplices may have $\sqrt[d]{n}Y'$ as a vertex, whereas any removed simplices may only have vertices within $\sqrt[d]{n}\mathbf Y_n$. For $I_n=\sum_{i=1}^n \ind{\|Y_i-Y'\|\leq \phi/\sqrt[d]{n}}$, via the Geometric Lemma~\ref{lemma::geometric} we have
	\begin{align*}
		& \abs{\beta_{q}^{r, s}\p{\mathcal K\p{\sqrt[d]{n}\p{\mathbf Y_n\cup \set{Y'}}}}-\beta_{q}^{r, s}\p{\mathcal K\p{\sqrt[d]{n}\mathbf Y_n}}} \\
		\leq \ & \abs{\beta_{q}^{r, s}\p{\mathcal K\p{\sqrt[d]{n}\p{\mathbf Y_n\cup \set{Y'}}}\cup \mathcal K\p{\sqrt[d]{n}\mathbf Y_n}}-\beta_{q}^{r, s}\p{\mathcal K\p{\sqrt[d]{n}\p{\mathbf Y_n\cup \set{Y'}}}}}\\
		& + \abs{\beta_{q}^{r, s}\p{\mathcal K\p{\sqrt[d]{n}\p{\mathbf Y_n\cup \set{Y'}}}\cup \mathcal K\p{\sqrt[d]{n}\mathbf Y_n}}-\beta_{q}^{r, s}\p{\mathcal K\p{\sqrt[d]{n}\mathbf Y_n}}} \\
		\leq \ & \#\set{K_q^r\p{\sqrt[d]{n}\mathbf Y_n}\setminus K_q^r\p{\sqrt[d]{n}\p{\mathbf Y_n\cup \set{Y'}}}} \\
		& + \#\set{K_{q+1}^s\p{\sqrt[d]{n}\mathbf Y_n}\setminus K_{q+1}^s\p{\sqrt[d]{n}\p{\mathbf Y_n\cup \set{Y'}}}} \\
		& + \#\set{K_q^r\p{\sqrt[d]{n}\p{\mathbf Y_n\cup \set{Y'}}}\setminus K_q^r\p{\sqrt[d]{n}\mathbf Y_n}} \\
		& + \#\set{K_{q+1}^s\p{\sqrt[d]{n}\p{\mathbf Y_n\cup \set{Y'}}}\setminus K_{q+1}^s\p{\sqrt[d]{n}\mathbf Y_n}} \\
		\leq \ & \binom{I_n}{q+1} + \binom{I_n}{q+2} + \binom{I_n+1}{q+1} + \binom{I_n+1}{q+2} \\
		\leq \ & 2\binom{I_n+2}{q+2} \\
		\leq \ & \frac{2^{q+3}}{\p{q+2}!}\p{I_n+1}^{q+2}.
	\end{align*}
	
	Thus for any $a>2$,
	\begin{equation*}
		\abs{\beta_{q}^{r, s}\p{\mathcal K\p{\sqrt[d]{n}\p{\mathbf Y_n\cup \set{Y'}}}}-\beta_{q}^{r, s}\p{\mathcal K\p{\sqrt[d]{n}\mathbf Y_n}}}^a \leq \frac{2^{\p{a+1}\p{q+2}}}{\p{\p{q+2}!}^a}\p{I_n^{a\p{q+2}}+1}.
	\end{equation*}
	
	\eqref{condition::local_count} is satisfied for $R=\phi$, $U_a=\p{2^{\p{a+1}\p{q+2}}}/{\p{\p{q+2}!}^a}$, and $u_a=a\p{q+2}$. Thus for $p\geq a\p{q+2}+1>2q+5$, \eqref{condition::expectation} follows by Lemma~\ref{lemma::condition_expectation}. \eqref{condition::stabilization_radius} and thus \eqref{condition::stabilization} follow from Lemmas~\ref{lemma::stabilization_pbn} and \ref{lemma::condition_stabilization}, respectively. An application of Theorem~\ref{theorem::bootstrap_general} gives the final result.
	
	For the rate in Proposition~\ref{proposition::w2_general}, for $p<\infty$, using $a=\p{p-1}/\p{q+2}$ we achieve an optimal rate for $\gamma_\epsilon$ of
	\begin{equation}
		O\p{\delta_\epsilon^{1-\frac{2q+4}{p-1}}}.
	\end{equation}
	
	For $p=\infty$, using $a_\epsilon=2-\log\p{\delta_\epsilon}$ we achieve an optimal rate of
	\begin{equation}
		O\p{\delta_\epsilon \p{\frac{-\log\p{\delta_\epsilon}}{\log\p{-\log\p{\delta_\epsilon}}}}^{2q+4}}.
	\end{equation}
\end{proof}

\begin{corollary}[Corollary~\ref{theorem::bootstrap_euler_nontrunc}]
	Let $m<\infty$ and $p>2m+3$. Let $\mathcal K$ be a filtration of simplicial complexes satisfying \eqref{condition::complex_increasing}, \eqref{condition::complex_translate}, \eqref{condition::complex_local_adj}, \eqref{condition::complex_minimum}, and \eqref{condition::complex_higher}. Then for any given $\vec r$, Statement~\ref{conclusion::bootstrap} holds for $\chi^{\vec r}$.
\end{corollary}

\begin{corollary}[Corollary~\ref{theorem::bootstrap_euler_nontrunc_ball}]
	Let $m<\infty$ and $p>2m+5$. Let $\mathcal K$ be a filtration of simplicial complexes satisfying \eqref{condition::complex_translate}, \eqref{condition::complex_local_ball}, \eqref{condition::complex_minimum}, and \eqref{condition::complex_higher}. Then for any given $\vec r$, Statement~\ref{conclusion::bootstrap} holds for $\chi^{\vec r}$.
\end{corollary}

\begin{proof}
	We prove together Corollaries~\ref{theorem::bootstrap_euler_nontrunc} and \ref{theorem::bootstrap_euler_nontrunc_ball}. Recall that the Euler characteristic $\chi$ can be written as an alternating (finite) sum of the Betti numbers when \eqref{condition::complex_higher} holds. As mentioned after the proposition statement, since Proposition~\ref{proposition::w2_general} holds for the Betti numbers in dimensions $0\leq q\leq m$ under the assumed conditions (see the proofs of Corollaries~\ref{theorem::bootstrap_pbn} and \ref{theorem::bootstrap_pbn_ball}), then the same holds for their (alternating) sum, namely the Euler characteristic. The proof of Theorem~\ref{theorem::bootstrap_general} applies without alteration.
\end{proof}

\begin{corollary}[Corollary~\ref{theorem::bootstrap_knn}]
	Let $p>2$. Furthermore, let $F\in \mathcal D_{\gamma, r_0}\p{C}$ and $\mybrk\ind{\hat F_n\in \mathcal D_{\gamma, r_0}\p{C}}\rightarrow 1$ in probability (resp. a.s.). Then Statement~\ref{conclusion::bootstrap} holds for $l_{\text{NN}, k}$.
\end{corollary}

\begin{proof}
	
	First, we will show that $\E{\abs{l_{\text{NN}, k}\p{\sqrt[d]{n}\p{\mathbf Y_n\cup\set{Y'}}}-l_{\text{NN}, k}\p{\sqrt[d]{n}\mathbf Y_n}}^a}$ is uniformly bounded for $G\in\mathcal D_{\gamma, r_0}\p{C}$ and $Y', Y_1, ..., Y_n\iid G$. Denote by $A_{k+1}$ the $k+1$ nearest neighbors of $\sqrt[d]{n}Y'$ in $\sqrt[d]{n}{\mathbf Y_n}$. Denote by $B_k$ the set of points in $\sqrt[d]{n}{\mathbf Y_n}$ for which $\sqrt[d]{n}Y'$ is among the $k$ nearest neighbors.
	
	It may be shown that $\#\set{B_k}\leq C_{d,k}$, where $C_{d, k}$ is a constant depending only on the dimension $d$ and $k$. To show this, consider a cone of angle $\pi/6$ whose point lies on $\sqrt[d]{n}Y'$. For $y_1, ..., y_k$ the $k$ closest points of $B_k$ to $\sqrt[d]{n}Y'$ within the cone, it follows from basic geometric arguments that any point lying within the cone, but outside a radius of $\max\set{\|y_i-\sqrt[d]{n}Y'\|}_{i=1}^k$ from $\sqrt[d]{n}Y'$ must be closer to each of $y_1, ..., y_k$ than to $\sqrt[d]{n}Y'$. Thus, any cone of this type may contain at most $k$ points of $B_n$. Since $\reals^d$ may be covered by finitely many of these cones, there must exist the required bound $C_{d, k}$.
	
	Now, consider the points of $A_{k+1}$ and $B_k$. Let $R_{k+1, n}:=\max\set{\|y-\sqrt[d]{n}Y'\|:y\in A_n}$. For any point $y$ in $B_n$, the distance to each point of $A_n$ is at most $\|y-\sqrt[d]{n}Y'\|+R_{k+1, n}$ by the triangle inequality. In this case, the introduction of $\sqrt[d]{n}Y'$ to the sample may reduce the contribution to $l_{\text{NN}, k}$ from the points in $B_n$ by at most
	\begin{equation*}
		l_{\text{NN}, k}\p{\sqrt[d]{n}\mathbf Y_n}-l_{\text{NN}, k}\p{\sqrt[d]{n}\p{\mathbf Y_n\cup\set{Y'}}}\leq C_{d, k}R_{k+1, n}.
	\end{equation*}
	
	Likewise, the contribution of $\sqrt[d]{n}Y'$ is bounded by
	\begin{equation*}
		l_{\text{NN}, k}\p{\sqrt[d]{n}\p{\mathbf Y_n\cup\set{Y'}}}-l_{\text{NN}, k}\p{\sqrt[d]{n}\mathbf Y_n}\leq k R_{k, n}\leq k R_{k+1, n}.
	\end{equation*}
	
	Thus, we proceed by bounding $\E{R_{k+1, n}^a}$. For any $G\in\mathcal D_{\gamma, r_0}$
	\begin{equation*}
		\E{\int\limits_{r_0}^\infty \prob{\|Y_j-Y'\|>\sqrt[a]{r}\big| \ Y'}^n\de r} \leq \diam{C}\p{1-\gamma r_0^{\frac{d}{a}}}^n.
	\end{equation*}
	
	We apply a bound similar to Theorem~7 in \cite{Singh2016}. In the statement of the referenced theorem, it is assumed that the above quantity is bounded by $C_T/n$ for an appropriate constant $C_T$. Here, we may improve that to an exponential bound. Consequently, we have 
	
	\begin{align}
		&\E{R_{k+1, n}^a}\leq \p{\frac{k+1}{\gamma}}^{\frac{a}{d}} +\diam{C}n^{\frac{a}{d}}\p{1-\gamma r_0^{\frac{d}{a}}}^n + \frac{a\p{e/\p{k+1}}^{k+1}}{d\p{\gamma }^{\frac{a}{d}}}\int\limits_{k+1}^\infty e^{-y}y^{k+\frac{a}{d}}\de y.
	\end{align}
	
	For any $a<\infty$, this quantity limits to a constant with $n\rightarrow\infty$, thus admitting a constant upper bound which holds for all $n\in\nats$, satisfying \eqref{condition::expectation}.
	
	The required stabilization properties \eqref{definition::stabilization_almost_sure} are first established for a unit-intensity homogeneous Poisson process via Lemma~6.1 in \cite{Penrose2001}. Let $\rho$ denote the minimal locally-determined radius of stabilization for $l_{\text{NN}, k}$. Let $\mathbb P_\lambda$ denote a homogeneous Poisson process with intensity $\lambda$. By the scaling properties of $l_{\text{NN}, k}$, we have $\rho_0\p{\mathbb P_\lambda}=\rho_0\p{\mathbb P_1/\sqrt[d]{\lambda}}=\rho_0\p{\mathbb P_1}/\sqrt[d]{\lambda}$. Thus, $\probout{\rho_0\p{\mathbb P_\lambda}>L}=\probout{\rho_0\p{\mathbb P_1}>\sqrt[d]{\lambda} L}$. For any $\lambda>1$, $\probout{\rho_0\p{\mathbb P_\lambda}>L}\leq\probout{\rho_0\p{\mathbb P_1}>L}$. Likewise, for any $\lambda_*<1$, we may choose $L_\delta$ such that $\probout{\rho_0\p{\mathbb P_1}>\sqrt[d]{\lambda_*}L_\delta}\leq \delta$. Then $\prob{\rho_0\p{\mathbb P_\lambda}>L_\delta}\leq \delta$ for all $\lambda\in\left[\lambda_*,\infty\right)$. Stabilization then extends to the binomial sampling setting via Lemma~\ref{lemma::stabilization_poisson} and the translation invariance of $l_{\text{NN}, k}$. We have for any $\delta>0$ that there exists an $n_\delta<\infty$ and $L_\delta^*<\infty$ such that $\probout{\rho_{\sqrt[d]{n}Y'}\p{\sqrt[d]{n}\mathbf Y_n}>L_\delta^*}\leq \delta$. Both quantities do not depend specifically on $G$.
	
	When restricted to $C$, we have an absolute upper bound of $\diam{C}\sqrt[d]{n}$ for the radius of stabilization, as all points will fall inside of $C$ almost surely. We set $L_\delta= \mybrk \max\{\diam{C}\sqrt[d]{n_\delta}, L_\delta^*\}$. Then $\probout{\rho_{\sqrt[d]{n}Y'}\p{\sqrt[d]{n}\mathbf Y_n}>L_\delta}\leq \delta$ for all $n\in\nats$, satisfying \eqref{condition::stabilization_radius}.
	
	We now have the required pieces to prove bootstrap convergence. Although $\mathcal C_{p, M}\cap\mathcal D_{\gamma,r_0}\p{C}$ is only a subset of $\mathcal C_{p, M}$, the proof and conclusion of Proposition~\ref{proposition::w2_general} still apply. Likewise, the proof of Theorem~\ref{theorem::bootstrap_general} is easily altered to include the additional condition $\ind{\hat F_n\in D_{\gamma, r_0}\p{C}}\rightarrow 1$. We omit details here.
\end{proof}

\section{$L_p$ Consistency of Kernel Density Estimators} \label{appendix::lp_norm}

In this section we discuss the $L_p$-norm consistency of the kernel density estimator under very mild conditions. To the best of our knowledge, the exact proof of this result could not be found in the kernel density literature, though it employs well-known results from probability theory. In the context of our smoothed bootstrap procedure, the $L_p$-norm convergence assumption of the KDE follows as a direct consequence of the following theorem. Notably, the necessary assumptions for $L_p$-norm convergence for the KDE are strictly weaker than those of Theorem~\ref{theorem::bootstrap_general}.

For $Q$ a kernel with $\int_{\reals^d} Q\p{x}\de x = 1$, define $Q_h\p{x}:=Q\p{x/h}/h^d$. Let $F$ be a probability distribution on $\reals^d$ with corresponding density $f$ and $\set{X_i}_{i\in\nats}\iid F$. The kernel density estimator for $f$ with bandwidth $h$ is
\begin{equation}
	\hat f_{n, h}\p{x} := \frac{1}{n}\sum\limits_{i=1}^n Q_h\p{x-X_i}
\end{equation}

\begin{proposition} \label{theorem::lp_norm}
	Given $p\geq 2$, let $\norm[p]{Q}<\infty$ and $\norm[p]{f}<\infty$. Then for any $h_n$ such that $\lim_{n\rightarrow\infty} h_n = \infty$ and $\lim_{n\rightarrow\infty}n^{p/\p{2d\p{p-1}}}h_n = \infty$
	\begin{equation}
		\norm[p]{\hat f_{n, h_n}-f}\overset{p}{\rightarrow} 0
	\end{equation}
\end{proposition}

If further $\sum_{n\in\nats} 1/\p{n^{p/2}h_n^{d\p{p-1}}}<\infty$
\begin{equation}
	\norm[p]{\hat f_{n, h_n} - f} \overset{a.s.}{\rightarrow} 0
\end{equation}

\begin{proof} \label{proof::lp_norm}
	
	The expectation of $\hat f_{n, h_n}$ is $Q_{h_n}*f$, where $*$ denotes the convolution operator. We expand the $L_p$-norm using the triangle inequality.
	\begin{equation}
		\norm[p]{\hat f_{n, h_n}-f}\leq \norm[p]{\hat f_{n, h_n} - Q_{h_n}*f}+\norm[p]{Q_{h_n}*f-f} \label{proof::lp_norm::1}
	\end{equation}
	
	Because $\int_{\reals^d}Q_{h_n}\p{x}\de x = 1$ and $\norm[p]{f}<\infty$, the second term goes to $0$ with $h_n\rightarrow 0$ via Theorem 8.14 in \cite{folland_1999}. We focus on the first term of \eqref{proof::lp_norm::1}.
	\begin{align}
		\E{\int\abs{\hat f_{n, h_n}\p{x} - \p{Q_{h_n}*f}\p{x}}^p\de x} & = \int\E{\abs{\hat f_{n, h_n}\p{x}- \p{Q_{h_n}*f}\p{x}}^p}\de x \\
		& = \frac1{n^p}\int\E{\abs{\sum\limits_{i=1}^n Y_i\p{x}}^p}\de x
	\end{align}
	where $Y_i\p{x}:=Q_{h_n}\p{x-X_i}-\p{Q_{h_n}*f}\p{x}$ are iid mean-zero random variables.
	
	We symmetrize using independent Radamacher random variables $\set{e_i}_{i\in\nats}$, letting $Z_i\p{x}:=e_iY_i\p{x}$. We have that $\E{\abs{\sum_{i=1}^nY_i\p{x}}^p}\leq 2^p\E{\abs{\sum_{i=1}^n Z_i\p{x}}^p}$. By Corollary 3 in \cite{Latala1997}, there exists a universal constant $C<\infty$ such that, for any $j\in\nats$
	\begin{align}
		\E{\abs{\sum_{i=1}^n Z_i\p{x}}^p} \leq \ &  C^p\p{\frac{p}{\log p}}^p\max\set{\p{n\E{\abs{Z_j\p{x}}^2}}^{\frac{p}{2}}, n\E{\abs{Z_j\p{x}}^p}} \\
		= \ & C^p\p{\frac{p}{\log p}}^p\max\set{\p{n\E{\abs{Y_j\p{x}}^2}}^{\frac{p}{2}}, n\E{\abs{Y_j\p{x}}^p}} \nonumber \\
		\leq \ & C^p\p{\frac{p}{\log p}}^p\max\set{n^{\frac{p}{2}}\E{\abs{Y_j\p{x}}^p}, n\E{\abs{Y_j\p{x}}^p}} \nonumber \\
		= \ & C^p\p{\frac{p}{\log p}}^pn^{\frac{p}{2}}\E{\abs{Y_j\p{x}}^p}.
	\end{align}
	
	Then
	\begin{equation}
		\E{\int\abs{\hat f_{n, h_n}\p{x} - \p{Q_{h_n}*f}\p{x}}^p\de x} \leq \frac{2^pC^p}{n^{\frac p2}}\p{\frac{p}{\log p}}^p\int\E{\abs{Y_j\p{x}}^p}\de x.
	\end{equation}
	
	\begin{align}
		\int\E{\abs{Y_j\p{x}}^p}\de x= \ & \E{\int \abs{Y_j\p{x}}^p\de x} \\
		= \ & \int\int\abs{Q_{h_n}\p{x-y}-\p{Q_{h_n}*f}\p x}^p f\p{y}\de x\de y \nonumber \\
		\leq \ & 2^{p-1}\int\int\p{\abs{Q_{h_n}\p{x-y}}^p+\abs{\p{Q_{h_n}*f}\p{x}}^p}f\p{y}\de x \de y \nonumber \\
		=\ &  2^{p-1}\p{\norm[p]{Q_{h_n}}^p+\norm[p]{Q_{h_n}*f}^p} \nonumber \\
		\leq &  2^{p}\norm[p]{Q_{h_n}}^p \nonumber \\
		= \ & \frac{2^p}{\p{h_n^d}^{p-1}}\norm[p]{Q}^p.
	\end{align}
	
	The last inequality follows from Young's inequality for convolutions, given that $\norm[1]{f} = 1$, $f$ being a probability density.
	
	\begin{align}
		\E{\int\abs{\hat f_{n, h_n}\p{x} - \p{Q_{h_n}*f}\p{x}}^p\de x} \leq \ 4^pC^p\p{\frac{p}{\log p}}^p\frac{\norm[p]{Q}^p}{\p{n^{\frac{p}{2d\p{p-1}}}h_n}^{d\p{p-1}}}
	\end{align}
	
	As $\lim_{n\rightarrow\infty}n^{p/\p{2d\p{p-1}}}h_n = \infty$ by assumption, this final bound goes to $0$ with $n\rightarrow \infty$. For any $\epsilon>0$, Markov's inequality gives
	\begin{align}
		\prob{\norm[p]{\hat f_{n, h_n} - Q_{h_n}*f}\geq\epsilon} = \ & \prob{\norm[p]{\hat f_{n, h_n} - Q_{h_n}*f}^p\geq\epsilon^p}\\
		\leq \ &  \frac{\E{\norm[p]{\hat f_{n, h_n} - Q_{h_n}*f}^p}}{\epsilon^p}
	\end{align}
	
	As was shown earlier, the right hand side goes to $0$, thus $\norm[p]{\hat f_{n, h_n} - Q_{h_n}*f}\overset{p}{\rightarrow} 0$. As $\norm[p]{Q_{h_n}*f-f}\rightarrow 0$, an application of Slutsky's theorem gives the final result. If $\sum_{n\in\nats} 1/\p{n^{p/2}h_n^{d\p{p-1}}}<\infty$, the almost sure result follows from Borel-Cantelli. $h_n=n^{-(p-1)/2}\p{\log\p{n}}^2$ satisfies this criterion.
	
\end{proof}

\section{Details of Simulation Study} \label{appendix::simulation_details}

Provided here are the data generating functions, written in pseudocode, for the simulation study of Section~\ref{section::simulation_study}. Each generator below corresponds to a distribution $F_1$-$F_7$ in Table~\ref{table::distributions}. A description is included, explaining each case in more detail. In all of the following, $\mathbb S^{d-1}$ denotes the unit sphere in $\reals^{d}$, $B_z\p{r}$ the ball with radius $r$ around $z$, and $\runif{S}$ the uniform distribution on the set $S$. $\rnorm{\mu, \sigma^2}$ denotes the normal distribution with mean $\mu$ and variance $\sigma^2$, and $\rexp{\lambda}$ is the exponential distribution with rate parameter $\lambda$. $\text{Cauchy}\p{\lambda}$ denotes the Cauchy distribution with scale parameter $\lambda$, and $\p{\cdot}$ is used to show vector concatenation.

\begin{figure}[H]
	\textbf{Generator 1:}
	\vspace{.5\baselineskip}
	\hrule
	\vspace{.5\baselineskip}
	\begin{algorithmic}[1]
		\State $\theta\sim\runif{\mathbb S^1}$
		\State $S\sim\runif{\set{-1, 1}}$
		\State $R\sim\runif{\brk{0, 1}}$
		
		\Return $X=\theta R^{.9S}$
	\end{algorithmic}
	\vspace{.5\baselineskip}
	\hrule
	\vspace{.5\baselineskip}
	\textit{$F_1$ is radially symmetric around the origin, and the radius is such that the random variable is unbounded, and the $L_8$ norm of the overall density is finite. Furthermore, the density approaches infinity near the origin. This case is chosen so as to test the assumptions of Corollary~\ref{theorem::bootstrap_pbn} with regards to the required norm bound.}
\end{figure}

\begin{figure}[H]
	\textbf{Generator 2:}
	\vspace{.5\baselineskip}
	\hrule
	\vspace{.5\baselineskip}
	\begin{algorithmic}[1]
		\State $\theta\sim\runif{\mathbb S^1}$
		\State $S\sim\runif{\set{-1, 1}}$
		\State $R\sim\runif{\brk{0, 1}}$
		
		\Return $X=\theta R^{.55S}$
	\end{algorithmic}
	\vspace{.5\baselineskip}
	\hrule
	\vspace{.5\baselineskip}
	\textit{$F_2$ is radially symmetric around the origin, and the radius is such that the random variable is unbounded. The $L_2$ norm of the overall density is finite, but the $L_8$ norm is infinite. As with distribution $F_1$, the density approaches infinity near the origin. This case violates the assumptions of Corollary~\ref{theorem::bootstrap_pbn}.}
\end{figure}

\begin{figure}[H]
	\textbf{Generator 3:}
	\vspace{.5\baselineskip}
	\hrule
	\vspace{.5\baselineskip}
	\begin{algorithmic}[1]
		\State $\theta\sim\runif{\mathbb S^1}$
		\State $X_1, X_2\sim \rnorm{0, .04}$
		
		\Return $\theta + \p{Y_1, Y_2}$.
	\end{algorithmic}
	\vspace{.5\baselineskip}
	\hrule
	\vspace{.5\baselineskip}
	\textit{$F_3$ represents a ring in $\reals^2$, combined with additive Gaussian noise. The variance parameter is chosen small enough so that the ring structure is not lost within the additive noise.}
\end{figure}

\begin{figure}[H]
	\textbf{Generator 4:}
	\vspace{.5\baselineskip}
	\hrule
	\vspace{.5\baselineskip}
	\begin{algorithmic}[1]
		\State $\theta\sim\runif{B_0\p{1}}$
		\State $X_1, X_2, X_3\sim \rnorm{0, .01}$
		
		\Return $\theta + \p{Y_1, Y_2, Y_3}$.
	\end{algorithmic}
	\vspace{.5\baselineskip}
	\hrule
	\vspace{.5\baselineskip}
	\textit{$F_4$ is the uniform distribution on the unit ball in $\reals^3$, with a small amount of additive noise included to slightly smooth the boundary at radius $1$.}
\end{figure}

\begin{figure}[H]
	\textbf{Generator 5:}
	\vspace{.5\baselineskip}
	\hrule
	\vspace{.5\baselineskip}
	\begin{algorithmic}[1]
		\State $X\sim \runif{\begin{Bmatrix} \p{0.38741799, 0.24263535, 0.09535272} \\ \p{0.25147839, 0.63824409, 0.62425101} \\ \p{0.73988542, 0.80749034, 0.84972394} \\ \p{0.26811913, 0.35911205, 0.08316547} \\ \p{0.65954757, 0.04704809, 0.02113341} \end{Bmatrix}}$
		\State $Y_1, Y_2, Y_3\sim \rexp{25}$
		
		\Return $X + \p{Y_1, Y_2, Y_3}$.
	\end{algorithmic}
	\vspace{.5\baselineskip}
	\hrule
	\vspace{.5\baselineskip}
	\textit{$F_5$ consists of $5$ clusters, one around each of the provided opints in $\reals^3$. Exponential noise is included to test the effects of heavier tails on the final coverage probability. The rate parameter was chosen large enough so that the $5$ clusters remain distinct after noise addition.}
\end{figure}

\begin{figure}[H]
	\textbf{Generator 6:}
	\vspace{.5\baselineskip}
	\hrule
	\vspace{.5\baselineskip}
	\begin{algorithmic}[1]
		\State $\theta\sim \runif{\mathbb S^2}$
		\State $Y_1, ..., Y_5\sim \text{Cauchy}\p{.1}$
		
		\Return $\p{\theta, 0, 0} + \p{Y_1, ..., Y_5}$
	\end{algorithmic}
	\vspace{.5\baselineskip}
	\hrule
	\vspace{.5\baselineskip}
	\textit{$F_6$ represents a $2$-dimensional unit sphere embedded in a higher dimension $\reals^6$. We have included additive Cauchy noise to investigate the effects of very heavy tails.}
\end{figure}

\begin{figure}[H]
	\textbf{Generator 7:}
	\vspace{.5\baselineskip}
	\hrule
	\vspace{.5\baselineskip}
	\begin{algorithmic}[1]
		\State $\p{\theta_1, \theta_2}\sim \runif{\mathbb S^1}$
		\State $S \sim \runif{\set{-1, 1}}$
		\State $Y_1, ..., Y_{10}\sim\rnorm{0, .04}$
		
		\Return $\p{\theta_1 + S, \theta_2, 0, ..., 0} + \p{Y_1, ..., Y_{10}}$
	\end{algorithmic}
	\vspace{.5\baselineskip}
	\hrule
	\vspace{.5\baselineskip}
	\textit{$F_7$ represents a dual ring, or figure-$8$ embedded in $\reals^{10}$. Full-dimensional Gaussian noise is added, with variance chosen small enough so that the dual rings are not closed upon noise addition. $F_7$ is included to illustrate the effects of the ``curse of dimensionality'' expected in higher dimensions.}
\end{figure}


\begin{thebibliography}{55}

\bibitem{Adler2017}
\begin{barticle}[author]
\bauthor{\bsnm{Adler},~\bfnm{Robert~J.}\binits{R.~J.}},
  \bauthor{\bsnm{Agami},~\bfnm{Sarit}\binits{S.}} \AND
  \bauthor{\bsnm{Pranav},~\bfnm{Pratyush}\binits{P.}}
(\byear{2017}).
\btitle{Modeling and Replicating Statistical Topology and Evidence for {CMB}
  Nonhomogeneity}.
\bjournal{Proc. Natl. Acad. Sci. USA}
\bvolume{114}
\bpages{11878--11883}.
\bdoi{10.1073/pnas.1706885114}
\bmrnumber{3725115}
\end{barticle}
\endbibitem

\bibitem{Aldous1992}
\begin{barticle}[author]
\bauthor{\bsnm{Aldous},~\bfnm{David}\binits{D.}} \AND
  \bauthor{\bsnm{Steele},~\bfnm{J.~Michael}\binits{J.~M.}}
(\byear{1992}).
\btitle{Asymptotics for {E}uclidean Minimal Spanning Trees on Random Points}.
\bjournal{Probab. Theory Related Fields}
\bvolume{92}
\bpages{247--258}.
\bdoi{10.1007/BF01194923}
\bmrnumber{1161188}
\end{barticle}
\endbibitem

\bibitem{Arsuaga2015}
\begin{barticle}[author]
\bauthor{\bsnm{Arsuaga},~\bfnm{Javier}\binits{J.}},
  \bauthor{\bsnm{Borrman},~\bfnm{Tyler}\binits{T.}},
  \bauthor{\bsnm{Cavalcante},~\bfnm{Raymond}\binits{R.}},
  \bauthor{\bsnm{Gonzalez},~\bfnm{Georgina}\binits{G.}} \AND
  \bauthor{\bsnm{Park},~\bfnm{Catherine}\binits{C.}}
(\byear{2015}).
\btitle{Identification of Copy Number Aberrations in Breast Cancer Subtypes
  Using Persistence Topology}.
\bjournal{Microarrays}
\bvolume{4}
\bpages{339--369}.
\bdoi{10.3390/microarrays4030339}
\end{barticle}
\endbibitem

\bibitem{Biscio2020}
\begin{barticle}[author]
\bauthor{\bsnm{Biscio},~\bfnm{Christophe A.~N.}\binits{C.~A.~N.}},
  \bauthor{\bsnm{Chenavier},~\bfnm{Nicolas}\binits{N.}},
  \bauthor{\bsnm{Hirsch},~\bfnm{Christian}\binits{C.}} \AND
  \bauthor{\bsnm{Svane},~\bfnm{Anne~Marie}\binits{A.~M.}}
(\byear{2020}).
\btitle{Testing Goodness of Fit for Point Processes Via Topological Data
  Analysis}.
\bjournal{Electron. J. Stat.}
\bvolume{14}
\bpages{1024--1074}.
\bdoi{10.1214/20-EJS1683}
\bmrnumber{4067816}
\end{barticle}
\endbibitem

\bibitem{Blanton2017}
\begin{barticle}[author]
\bauthor{\bsnm{{Blanton}},~\bfnm{M.~R.}\binits{M.~R.}},
  \bauthor{\bsnm{{Bershady}},~\bfnm{M.~A.}\binits{M.~A.}},
  \bauthor{\bsnm{{Abolfathi}},~\bfnm{B.}\binits{B.}},
  \bauthor{\bsnm{{Albareti}},~\bfnm{F.~D.}\binits{F.~D.}},
  \bauthor{\bsnm{{Allende Prieto}},~\bfnm{C.}\binits{C.}},
  \bauthor{\bsnm{{Almeida}},~\bfnm{A.}\binits{A.}},
  \bauthor{\bsnm{{Alonso-Garc{\'{\i}}a}},~\bfnm{J.}\binits{J.}},
  \bauthor{\bsnm{{Anders}},~\bfnm{F.}\binits{F.}},
  \bauthor{\bsnm{{Anderson}},~\bfnm{S.~F.}\binits{S.~F.}},
  \bauthor{\bsnm{{Andrews}},~\bfnm{B.}\binits{B.}} \AND \bauthor{\bparticle{et}
  \bsnm{al.}}
(\byear{2017}).
\btitle{{S}loan {D}igital {S}ky {S}urvey {IV}: Mapping the {M}ilky {W}ay,
  Nearby Galaxies, and the Distant Universe}.
\bjournal{Astronomical Journal}
\bvolume{154}
\bpages{28}.
\bdoi{10.3847/1538-3881/aa7567}
\end{barticle}
\endbibitem

\bibitem{Bobrowski2015}
\begin{barticle}[author]
\bauthor{\bsnm{Bobrowski},~\bfnm{Omer}\binits{O.}} \AND
  \bauthor{\bsnm{Mukherjee},~\bfnm{Sayan}\binits{S.}}
(\byear{2015}).
\btitle{The Topology of Probability Distributions on Manifolds}.
\bjournal{Probab. Theory Related Fields}
\bvolume{161}
\bpages{651--686}.
\bdoi{10.1007/s00440-014-0556-x}
\bmrnumber{3334278}
\end{barticle}
\endbibitem

\bibitem{Boissonnat2018}
\begin{bbook}[author]
\bauthor{\bsnm{Boissonnat},~\bfnm{Jean-Daniel}\binits{J.-D.}},
  \bauthor{\bsnm{Chazal},~\bfnm{Fr\'{e}d\'{e}ric}\binits{F.}} \AND
  \bauthor{\bsnm{Yvinec},~\bfnm{Mariette}\binits{M.}}
(\byear{2018}).
\btitle{Geometric and Topological Inference}.
\bseries{Cambridge Texts in Applied Mathematics}.
\bpublisher{Cambridge University Press, Cambridge}.
\bdoi{10.1017/9781108297806}
\bmrnumber{3837127}
\end{bbook}
\endbibitem

\bibitem{Breiman1977}
\begin{barticle}[author]
\bauthor{\bsnm{Breiman},~\bfnm{Leo}\binits{L.}},
  \bauthor{\bsnm{Meisel},~\bfnm{William}\binits{W.}} \AND
  \bauthor{\bsnm{Purcell},~\bfnm{Edward}\binits{E.}}
(\byear{1977}).
\btitle{Variable Kernel Estimates of Multivariate Densities}.
\bjournal{Technometrics}
\bvolume{19}
\bpages{135--144}.
\bdoi{10.1080/00401706.1977.10489521}
\end{barticle}
\endbibitem

\bibitem{Bubenik2015}
\begin{barticle}[author]
\bauthor{\bsnm{Bubenik},~\bfnm{Peter}\binits{P.}}
(\byear{2015}).
\btitle{Statistical Topological Data Analysis Using Persistence Landscapes}.
\bjournal{J. Mach. Learn. Res.}
\bvolume{16}
\bpages{77--102}.
\bmrnumber{3317230}
\end{barticle}
\endbibitem

\bibitem{Bubenik2007}
\begin{barticle}[author]
\bauthor{\bsnm{Bubenik},~\bfnm{Peter}\binits{P.}} \AND
  \bauthor{\bsnm{Kim},~\bfnm{Peter~T.}\binits{P.~T.}}
(\byear{2007}).
\btitle{A Statistical Approach to Persistent Homology}.
\bjournal{Homology Homotopy Appl.}
\bvolume{9}
\bpages{337--362}.
\bmrnumber{2366953}
\end{barticle}
\endbibitem

\bibitem{Camara2016}
\begin{barticle}[author]
\bauthor{\bsnm{Camara},~\bfnm{Pablo~G.}\binits{P.~G.}},
  \bauthor{\bsnm{Rosenbloom},~\bfnm{Daniel I.~S.}\binits{D.~I.~S.}},
  \bauthor{\bsnm{Emmett},~\bfnm{Kevin~J.}\binits{K.~J.}},
  \bauthor{\bsnm{Levine},~\bfnm{Arnold~J.}\binits{A.~J.}} \AND
  \bauthor{\bsnm{Rabadan},~\bfnm{Raul}\binits{R.}}
(\byear{2016}).
\btitle{Topological Data Analysis Generates High-Resolution, Genome-wide Maps
  of Human Recombination}.
\bjournal{Cell Systems}
\bvolume{3}
\bpages{83--94}.
\bdoi{10.1016/j.cels.2016.05.008}
\end{barticle}
\endbibitem

\bibitem{Chazal2018}
\begin{bincollection}[author]
\bauthor{\bsnm{Chazal},~\bfnm{Fr\'{e}d\'{e}ric}\binits{F.}} \AND
  \bauthor{\bsnm{Divol},~\bfnm{Vincent}\binits{V.}}
(\byear{2018}).
\btitle{The Density of Expected Persistence Diagrams and Its Kernel Based
  Estimation}.
In \bbooktitle{34th {I}nternational {S}ymposium on {C}omputational {G}eometry}.
\bseries{LIPIcs. Leibniz Int. Proc. Inform.}
\bvolume{99}
\bpages{Art. No. 26, 15}.
\bpublisher{Schloss Dagstuhl. Leibniz-Zent. Inform., Wadern}.
\bmrnumber{3824270}
\end{bincollection}
\endbibitem

\bibitem{Chazal2015a}
\begin{barticle}[author]
\bauthor{\bsnm{Chazal},~\bfnm{F.}\binits{F.}},
  \bauthor{\bsnm{Fasy},~\bfnm{B.~T.}\binits{B.~T.}},
  \bauthor{\bsnm{Lecci},~\bfnm{F.}\binits{F.}},
  \bauthor{\bsnm{Rinaldo},~\bfnm{A.}\binits{A.}},
  \bauthor{\bsnm{Singh},~\bfnm{A.}\binits{A.}} \AND
  \bauthor{\bsnm{Wasserman},~\bfnm{L.}\binits{L.}}
(\byear{2015}).
\btitle{On the Bootstrap for Persistence Diagrams and Landscapes}.
\bjournal{Modeling and Analysis of Information Systems}
\bvolume{20}
\bpages{111--120}.
\bdoi{10.18255/1818-1015-2013-6-111-120}
\end{barticle}
\endbibitem

\bibitem{Chazal2015}
\begin{barticle}[author]
\bauthor{\bsnm{Chazal},~\bfnm{Fr\'{e}d\'{e}ric}\binits{F.}},
  \bauthor{\bsnm{Fasy},~\bfnm{Brittany~Terese}\binits{B.~T.}},
  \bauthor{\bsnm{Lecci},~\bfnm{Fabrizio}\binits{F.}},
  \bauthor{\bsnm{Rinaldo},~\bfnm{Alessandro}\binits{A.}} \AND
  \bauthor{\bsnm{Wasserman},~\bfnm{Larry}\binits{L.}}
(\byear{2015}).
\btitle{Stochastic Convergence of Persistence Landscapes and Silhouettes}.
\bjournal{J. Comput. Geom.}
\bvolume{6}
\bpages{140--161}.
\bmrnumber{3323391}
\end{barticle}
\endbibitem

\bibitem{Chazal2017}
\begin{bmisc}[author]
\bauthor{\bsnm{Chazal},~\bfnm{Frédéric}\binits{F.}} \AND
  \bauthor{\bsnm{Michel},~\bfnm{Bertrand}\binits{B.}}
(\byear{2017}).
\btitle{An Introduction to Topological Data Analysis: Fundamental and Practical
  Aspects for Data Scientists}.
\end{bmisc}
\endbibitem

\bibitem{Chen2015}
\begin{bmisc}[author]
\bauthor{\bsnm{Chen},~\bfnm{Yen-Chi}\binits{Y.-C.}},
  \bauthor{\bsnm{Wang},~\bfnm{Daren}\binits{D.}},
  \bauthor{\bsnm{Rinaldo},~\bfnm{Alessandro}\binits{A.}} \AND
  \bauthor{\bsnm{Wasserman},~\bfnm{Larry}\binits{L.}}
(\byear{2015}).
\btitle{Statistical Analysis of Persistence Intensity Functions}.
\end{bmisc}
\endbibitem

\bibitem{Chung2019}
\begin{bmisc}[author]
\bauthor{\bsnm{Chung},~\bfnm{Yu-Min}\binits{Y.-M.}} \AND
  \bauthor{\bsnm{Lawson},~\bfnm{Austin}\binits{A.}}
(\byear{2019}).
\btitle{Persistence Curves: A Canonical Framework for Summarizing Persistence
  Diagrams}.
\end{bmisc}
\endbibitem

\bibitem{Crawford2019}
\begin{barticle}[author]
\bauthor{\bsnm{Crawford},~\bfnm{Lorin}\binits{L.}},
  \bauthor{\bsnm{Monod},~\bfnm{Anthea}\binits{A.}},
  \bauthor{\bsnm{Chen},~\bfnm{Andrew~X.}\binits{A.~X.}},
  \bauthor{\bsnm{Mukherjee},~\bfnm{Sayan}\binits{S.}} \AND
  \bauthor{\bsnm{Rabad{\'{a}}n},~\bfnm{Ra{\'{u}}l}\binits{R.}}
(\byear{2019}).
\btitle{Predicting Clinical Outcomes in Glioblastoma: An Application of
  Topological and Functional Data Analysis}.
\bjournal{Journal of the American Statistical Association}
\bpages{1--12}.
\bdoi{10.1080/01621459.2019.1671198}
\end{barticle}
\endbibitem

\bibitem{Devroye2017}
\begin{barticle}[author]
\bauthor{\bsnm{Devroye},~\bfnm{Luc}\binits{L.}},
  \bauthor{\bsnm{Gy\"{o}rfi},~\bfnm{L\'{a}szl\'{o}}\binits{L.}},
  \bauthor{\bsnm{Lugosi},~\bfnm{G\'{a}bor}\binits{G.}} \AND
  \bauthor{\bsnm{Walk},~\bfnm{Harro}\binits{H.}}
(\byear{2017}).
\btitle{On the Measure of {V}oronoi Cells}.
\bjournal{J. Appl. Probab.}
\bvolume{54}
\bpages{394--408}.
\bdoi{10.1017/jpr.2017.7}
\bmrnumber{3668473}
\end{barticle}
\endbibitem

\bibitem{Devroye1979}
\begin{barticle}[author]
\bauthor{\bsnm{Devroye},~\bfnm{L.~P.}\binits{L.~P.}} \AND
  \bauthor{\bsnm{Wagner},~\bfnm{T.~J.}\binits{T.~J.}}
(\byear{1979}).
\btitle{The {$L\sb{1}$} Convergence of Kernel Density Estimates}.
\bjournal{Ann. Statist.}
\bvolume{7}
\bpages{1136--1139}.
\bmrnumber{536515}
\end{barticle}
\endbibitem

\bibitem{DeWoskin2010}
\begin{barticle}[author]
\bauthor{\bsnm{DeWoskin},~\bfnm{D.}\binits{D.}},
  \bauthor{\bsnm{Climent},~\bfnm{J.}\binits{J.}},
  \bauthor{\bsnm{Cruz-White},~\bfnm{I.}\binits{I.}},
  \bauthor{\bsnm{Vazquez},~\bfnm{M.}\binits{M.}},
  \bauthor{\bsnm{Park},~\bfnm{C.}\binits{C.}} \AND
  \bauthor{\bsnm{Arsuaga},~\bfnm{J.}\binits{J.}}
(\byear{2010}).
\btitle{Applications of Computational Homology to the Analysis of Treatment
  Response in Breast Cancer Patients}.
\bjournal{Topology Appl.}
\bvolume{157}
\bpages{157--164}.
\bdoi{10.1016/j.topol.2009.04.036}
\bmrnumber{2556091}
\end{barticle}
\endbibitem

\bibitem{Edelsbrunner2002}
\begin{barticle}[author]
\bauthor{\bsnm{Edelsbrunner}}, \bauthor{\bsnm{Letscher}} \AND
  \bauthor{\bsnm{Zomorodian}}
(\byear{2002}).
\btitle{Topological Persistence and Simplification}.
\bjournal{Discrete {\&} Computational Geometry}
\bvolume{28}
\bpages{511--533}.
\bdoi{10.1007/s00454-002-2885-2}
\end{barticle}
\endbibitem

\bibitem{Fasy2014}
\begin{barticle}[author]
\bauthor{\bsnm{Fasy},~\bfnm{Brittany~Terese}\binits{B.~T.}},
  \bauthor{\bsnm{Lecci},~\bfnm{Fabrizio}\binits{F.}},
  \bauthor{\bsnm{Rinaldo},~\bfnm{Alessandro}\binits{A.}},
  \bauthor{\bsnm{Wasserman},~\bfnm{Larry}\binits{L.}},
  \bauthor{\bsnm{Balakrishnan},~\bfnm{Sivaraman}\binits{S.}} \AND
  \bauthor{\bsnm{Singh},~\bfnm{Aarti}\binits{A.}}
(\byear{2014}).
\btitle{Confidence Sets for Persistence Diagrams}.
\bjournal{Ann. Statist.}
\bvolume{42}
\bpages{2301--2339}.
\bdoi{10.1214/14-AOS1252}
\bmrnumber{3269981}
\end{barticle}
\endbibitem

\bibitem{Ferdosi2011}
\begin{barticle}[author]
\bauthor{\bsnm{Ferdosi},~\bfnm{B.~J.}\binits{B.~J.}},
  \bauthor{\bsnm{Buddelmeijer},~\bfnm{H.}\binits{H.}},
  \bauthor{\bsnm{Trager},~\bfnm{S.~C.}\binits{S.~C.}},
  \bauthor{\bsnm{Wilkinson},~\bfnm{M.~H.~F.}\binits{M.~H.~F.}} \AND
  \bauthor{\bsnm{Roerdink},~\bfnm{J.~B. T.~M.}\binits{J.~B. T.~M.}}
(\byear{2011}).
\btitle{Comparison of Density Estimation Methods for Astronomical Datasets}.
\bjournal{Astronomy {\&} Astrophysics}
\bvolume{531}
\bpages{A114}.
\bdoi{10.1051/0004-6361/201116878}
\end{barticle}
\endbibitem

\bibitem{folland_1999}
\begin{bbook}[author]
\bauthor{\bsnm{Folland},~\bfnm{Gerald~B.}\binits{G.~B.}}
(\byear{1999}).
\btitle{Real Analysis: Modern Techniques and Applications}.
\bpublisher{Wiley}.
\end{bbook}
\endbibitem

\bibitem{Hansen2008}
\begin{barticle}[author]
\bauthor{\bsnm{Hansen},~\bfnm{Bruce~E.}\binits{B.~E.}}
(\byear{2008}).
\btitle{Uniform Convergence Rates for Kernel Estimation with Dependent Data}.
\bjournal{Econometric Theory}
\bvolume{24}
\bpages{726--748}.
\bdoi{10.1017/S0266466608080304}
\bmrnumber{2409261}
\end{barticle}
\endbibitem

\bibitem{Hiraoka2018}
\begin{barticle}[author]
\bauthor{\bsnm{Hiraoka},~\bfnm{Yasuaki}\binits{Y.}},
  \bauthor{\bsnm{Shirai},~\bfnm{Tomoyuki}\binits{T.}} \AND
  \bauthor{\bsnm{Trinh},~\bfnm{Khanh~Duy}\binits{K.~D.}}
(\byear{2018}).
\btitle{Limit Theorems for Persistence Diagrams}.
\bjournal{Ann. Appl. Probab.}
\bvolume{28}
\bpages{2740--2780}.
\bdoi{10.1214/17-AAP1371}
\bmrnumber{3847972}
\end{barticle}
\endbibitem

\bibitem{Kim2018}
\begin{bmisc}[author]
\bauthor{\bsnm{Kim},~\bfnm{Jisu}\binits{J.}},
  \bauthor{\bsnm{Shin},~\bfnm{Jaehyeok}\binits{J.}},
  \bauthor{\bsnm{Rinaldo},~\bfnm{Alessandro}\binits{A.}} \AND
  \bauthor{\bsnm{Wasserman},~\bfnm{Larry}\binits{L.}}
(\byear{2018}).
\btitle{Uniform Convergence Rate of the Kernel Density Estimator Adaptive to
  Intrinsic Volume Dimension}.
\end{bmisc}
\endbibitem

\bibitem{Kramar2013}
\begin{barticle}[author]
\bauthor{\bsnm{Kramar},~\bfnm{M.}\binits{M.}},
  \bauthor{\bsnm{Goullet},~\bfnm{A.}\binits{A.}},
  \bauthor{\bsnm{Kondic},~\bfnm{L.}\binits{L.}} \AND
  \bauthor{\bsnm{Mischaikow},~\bfnm{K.}\binits{K.}}
(\byear{2013}).
\btitle{Persistence of Force Networks in Compressed Granular Media}.
\bjournal{Physical Review E}
\bvolume{87}.
\bdoi{10.1103/physreve.87.042207}
\end{barticle}
\endbibitem

\bibitem{Kramar2016}
\begin{barticle}[author]
\bauthor{\bsnm{Kram{\'{a}}r},~\bfnm{Miroslav}\binits{M.}},
  \bauthor{\bsnm{Levanger},~\bfnm{Rachel}\binits{R.}},
  \bauthor{\bsnm{Tithof},~\bfnm{Jeffrey}\binits{J.}},
  \bauthor{\bsnm{Suri},~\bfnm{Balachandra}\binits{B.}},
  \bauthor{\bsnm{Xu},~\bfnm{Mu}\binits{M.}},
  \bauthor{\bsnm{Paul},~\bfnm{Mark}\binits{M.}},
  \bauthor{\bsnm{Schatz},~\bfnm{Michael~F.}\binits{M.~F.}} \AND
  \bauthor{\bsnm{Mischaikow},~\bfnm{Konstantin}\binits{K.}}
(\byear{2016}).
\btitle{Analysis of Kolmogorov Flow and Rayleigh{\textendash}b{\'{e}}nard
  Convection Using Persistent Homology}.
\bjournal{Physica D: Nonlinear Phenomena}
\bvolume{334}
\bpages{82--98}.
\bdoi{10.1016/j.physd.2016.02.003}
\end{barticle}
\endbibitem

\bibitem{Krebs2019}
\begin{bmisc}[author]
\bauthor{\bsnm{Krebs},~\bfnm{Johannes T.~N.}\binits{J.~T.~N.}} \AND
  \bauthor{\bsnm{Polonik},~\bfnm{Wolfgang}\binits{W.}}
(\byear{2019}).
\btitle{On the Asymptotic Normality of Persistent Betti Numbers}.
\end{bmisc}
\endbibitem

\bibitem{LachiezeRey2019}
\begin{barticle}[author]
\bauthor{\bsnm{Lachi{\`{e}}ze-Rey},~\bfnm{Raphaël}\binits{R.}},
  \bauthor{\bsnm{Schulte},~\bfnm{Matthias}\binits{M.}} \AND
  \bauthor{\bsnm{Yukich},~\bfnm{J.~E.}\binits{J.~E.}}
(\byear{2019}).
\btitle{Normal Approximation for Stabilizing Functionals}.
\bjournal{The Annals of Applied Probability}
\bvolume{29}.
\bdoi{10.1214/18-aap1405}
\end{barticle}
\endbibitem

\bibitem{LachiezeRey2020}
\begin{barticle}[author]
\bauthor{\bsnm{Lachièze-Rey},~\bfnm{Raphaël}\binits{R.}},
  \bauthor{\bsnm{Peccati},~\bfnm{Giovanni}\binits{G.}} \AND
  \bauthor{\bsnm{Yang},~\bfnm{Xiaochuan}\binits{X.}}
(\byear{2020}).
\btitle{Quantitative Two-scale Stabilization on the Poisson Space}.
\end{barticle}
\endbibitem

\bibitem{Last2015}
\begin{barticle}[author]
\bauthor{\bsnm{Last},~\bfnm{Günter}\binits{G.}},
  \bauthor{\bsnm{Peccati},~\bfnm{Giovanni}\binits{G.}} \AND
  \bauthor{\bsnm{Schulte},~\bfnm{Matthias}\binits{M.}}
(\byear{2015}).
\btitle{Normal Approximation on Poisson Spaces: Mehler's Formula, Second Order
  Poincar{\'{e}} Inequalities and Stabilization}.
\bjournal{Probability Theory and Related Fields}
\bvolume{165}
\bpages{667--723}.
\bdoi{10.1007/s00440-015-0643-7}
\end{barticle}
\endbibitem

\bibitem{Latala1997}
\begin{barticle}[author]
\bauthor{\bsnm{Lata{\l}a},~\bfnm{Rafa{\l}}\binits{R.}}
(\byear{1997}).
\btitle{Estimation of Moments of Sums of Independent Real Random Variables}.
\bjournal{Ann. Probab.}
\bvolume{25}
\bpages{1502--1513}.
\bdoi{10.1214/aop/1024404522}
\bmrnumber{1457628}
\end{barticle}
\endbibitem

\bibitem{Owada2018}
\begin{barticle}[author]
\bauthor{\bsnm{Owada},~\bfnm{Takashi}\binits{T.}}
(\byear{2018}).
\btitle{Limit Theorems for {B}etti Numbers of Extreme Sample Clouds with
  Application to Persistence Barcodes}.
\bjournal{Ann. Appl. Probab.}
\bvolume{28}
\bpages{2814--2854}.
\bdoi{10.1214/17-AAP1375}
\bmrnumber{3847974}
\end{barticle}
\endbibitem

\bibitem{Owada2017}
\begin{barticle}[author]
\bauthor{\bsnm{Owada},~\bfnm{Takashi}\binits{T.}} \AND
  \bauthor{\bsnm{Adler},~\bfnm{Robert~J.}\binits{R.~J.}}
(\byear{2017}).
\btitle{Limit Theorems for Point Processes under Geometric Constraints (and
  Topological Crackle)}.
\bjournal{Ann. Probab.}
\bvolume{45}
\bpages{2004--2055}.
\bdoi{10.1214/16-AOP1106}
\bmrnumber{3650420}
\end{barticle}
\endbibitem

\bibitem{Penrose2001}
\begin{barticle}[author]
\bauthor{\bsnm{Penrose},~\bfnm{Mathew~D.}\binits{M.~D.}} \AND
  \bauthor{\bsnm{Yukich},~\bfnm{J.~E.}\binits{J.~E.}}
(\byear{2001}).
\btitle{Central Limit Theorems for Some Graphs in Computational Geometry}.
\bjournal{Ann. Appl. Probab.}
\bvolume{11}
\bpages{1005--1041}.
\bdoi{10.1214/aoap/1015345393}
\bmrnumber{1878288}
\end{barticle}
\endbibitem

\bibitem{Penrose2003}
\begin{barticle}[author]
\bauthor{\bsnm{Penrose},~\bfnm{Mathew~D.}\binits{M.~D.}} \AND
  \bauthor{\bsnm{Yukich},~\bfnm{J.~E.}\binits{J.~E.}}
(\byear{2003}).
\btitle{Weak Laws of Large Numbers in Geometric Probability}.
\bjournal{Ann. Appl. Probab.}
\bvolume{13}
\bpages{277--303}.
\bdoi{10.1214/aoap/1042765669}
\bmrnumber{1952000}
\end{barticle}
\endbibitem

\bibitem{Politis1999}
\begin{bbook}[author]
\bauthor{\bsnm{Politis},~\bfnm{Dimitris~N.}\binits{D.~N.}},
  \bauthor{\bsnm{Romano},~\bfnm{Joseph~P.}\binits{J.~P.}} \AND
  \bauthor{\bsnm{Wolf},~\bfnm{Michael}\binits{M.}}
(\byear{1999}).
\btitle{Subsampling}.
\bseries{Springer Series in Statistics}.
\bpublisher{Springer-Verlag, New York}.
\bdoi{10.1007/978-1-4612-1554-7}
\bmrnumber{1707286}
\end{bbook}
\endbibitem

\bibitem{Pranav2019}
\begin{barticle}[author]
\bauthor{\bsnm{Pranav},~\bfnm{Pratyush}\binits{P.}},
  \bauthor{\bsnm{Adler},~\bfnm{Robert~J.}\binits{R.~J.}},
  \bauthor{\bsnm{Buchert},~\bfnm{Thomas}\binits{T.}},
  \bauthor{\bsnm{Edelsbrunner},~\bfnm{Herbert}\binits{H.}},
  \bauthor{\bsnm{Jones},~\bfnm{Bernard J.~T.}\binits{B.~J.~T.}},
  \bauthor{\bsnm{Schwartzman},~\bfnm{Armin}\binits{A.}},
  \bauthor{\bsnm{Wagner},~\bfnm{Hubert}\binits{H.}} \AND
  \bauthor{\bparticle{van~de} \bsnm{Weygaert},~\bfnm{Rien}\binits{R.}}
(\byear{2019}).
\btitle{Unexpected Topology of the Temperature Fluctuations in the Cosmic
  Microwave Background}.
\bjournal{Astronomy {\&} Astrophysics}
\bvolume{627}
\bpages{A163}.
\bdoi{10.1051/0004-6361/201834916}
\end{barticle}
\endbibitem

\bibitem{Pranav2016}
\begin{barticle}[author]
\bauthor{\bsnm{Pranav},~\bfnm{Pratyush}\binits{P.}},
  \bauthor{\bsnm{Edelsbrunner},~\bfnm{Herbert}\binits{H.}},
  \bauthor{\bparticle{van~de} \bsnm{Weygaert},~\bfnm{Rien}\binits{R.}},
  \bauthor{\bsnm{Vegter},~\bfnm{Gert}\binits{G.}},
  \bauthor{\bsnm{Kerber},~\bfnm{Michael}\binits{M.}},
  \bauthor{\bsnm{Jones},~\bfnm{Bernard J.~T.}\binits{B.~J.~T.}} \AND
  \bauthor{\bsnm{Wintraecken},~\bfnm{Mathijs}\binits{M.}}
(\byear{2016}).
\btitle{The Topology of the Cosmic Web in Terms of Persistent Betti Numbers}.
\bjournal{Monthly Notices of the Royal Astronomical Society}
\bvolume{465}
\bpages{4281--4310}.
\bdoi{10.1093/mnras/stw2862}
\end{barticle}
\endbibitem

\bibitem{Pranav2019a}
\begin{barticle}[author]
\bauthor{\bsnm{Pranav},~\bfnm{Pratyush}\binits{P.}},
  \bauthor{\bparticle{van~de} \bsnm{Weygaert},~\bfnm{Rien}\binits{R.}},
  \bauthor{\bsnm{Vegter},~\bfnm{Gert}\binits{G.}},
  \bauthor{\bsnm{Jones},~\bfnm{Bernard J~T}\binits{B.~J.~T.}},
  \bauthor{\bsnm{Adler},~\bfnm{Robert~J}\binits{R.~J.}},
  \bauthor{\bsnm{Feldbrugge},~\bfnm{Job}\binits{J.}},
  \bauthor{\bsnm{Park},~\bfnm{Changbom}\binits{C.}},
  \bauthor{\bsnm{Buchert},~\bfnm{Thomas}\binits{T.}} \AND
  \bauthor{\bsnm{Kerber},~\bfnm{Michael}\binits{M.}}
(\byear{2019}).
\btitle{Topology and Geometry of {G}aussian Random Fields {I}: On {B}etti
  Numbers, {E}uler Characteristic, and {M}inkowski Functionals}.
\bjournal{Monthly Notices of the Royal Astronomical Society}
\bvolume{485}
\bpages{4167--4208}.
\bdoi{10.1093/mnras/stz541}
\end{barticle}
\endbibitem

\bibitem{Roycraft2021}
\begin{bmisc}[author]
\bauthor{\bsnm{Roycraft},~\bfnm{Benjamin}\binits{B.}}
(\byear{2021}).
\btitle{github.com/btroycraft/stabilizing{\_}statistics{\_}bootstrap}.
\bdoi{10.5281/ZENODO.4627098}
\end{bmisc}
\endbibitem

\bibitem{Roycraft2020}
\begin{barticle}[author]
\bauthor{\bsnm{Roycraft},~\bfnm{Benjamin}\binits{B.}},
  \bauthor{\bsnm{Krebs},~\bfnm{Johannes}\binits{J.}} \AND
  \bauthor{\bsnm{Polonik},~\bfnm{Wolfgang}\binits{W.}}
(\byear{2021}).
\btitle{Supplement to "Bootstrapping Persistent Betti Numbers and Other
  Stabilizing Statistics"}.
\end{barticle}
\endbibitem

\bibitem{Silverman1986}
\begin{bbook}[author]
\bauthor{\bsnm{Silverman},~\bfnm{B.~W.}\binits{B.~W.}}
(\byear{1986}).
\btitle{Density Estimation for Statistics and Data Analysis}.
\bseries{Monographs on Statistics and Applied Probability}.
\bpublisher{Chapman \& Hall, London}.
\bdoi{10.1007/978-1-4899-3324-9}
\bmrnumber{848134}
\end{bbook}
\endbibitem

\bibitem{Singh2016}
\begin{bmisc}[author]
\bauthor{\bsnm{Singh},~\bfnm{Shashank}\binits{S.}} \AND
  \bauthor{\bsnm{Póczos},~\bfnm{Barnabás}\binits{B.}}
(\byear{2016}).
\btitle{Analysis of k-Nearest Neighbor Distances with Application to Entropy
  Estimation}.
\end{bmisc}
\endbibitem

\bibitem{Trinh2019}
\begin{barticle}[author]
\bauthor{\bsnm{Trinh},~\bfnm{Khanh~Duy}\binits{K.~D.}}
(\byear{2019}).
\btitle{On Central Limit Theorems in Stochastic Geometry for Add-one Cost
  Stabilizing Functionals}.
\bjournal{Electron. Commun. Probab.}
\bvolume{24}
\bpages{Paper No. 76, 15}.
\bdoi{10.1214/19-ecp279}
\bmrnumber{4049088}
\end{barticle}
\endbibitem

\bibitem{Turner2014}
\begin{barticle}[author]
\bauthor{\bsnm{Turner},~\bfnm{Katharine}\binits{K.}},
  \bauthor{\bsnm{Mukherjee},~\bfnm{Sayan}\binits{S.}} \AND
  \bauthor{\bsnm{Boyer},~\bfnm{Doug~M.}\binits{D.~M.}}
(\byear{2014}).
\btitle{Persistent Homology Transform for Modeling Shapes and Surfaces}.
\bjournal{Inf. Inference}
\bvolume{3}
\bpages{310--344}.
\bdoi{10.1093/imaiai/iau011}
\bmrnumber{3311455}
\end{barticle}
\endbibitem

\bibitem{Ulmer2019}
\begin{barticle}[author]
\bauthor{\bsnm{Ulmer},~\bfnm{M.}\binits{M.}},
  \bauthor{\bsnm{Ziegelmeier},~\bfnm{Lori}\binits{L.}} \AND
  \bauthor{\bsnm{Topaz},~\bfnm{Chad~M.}\binits{C.~M.}}
(\byear{2019}).
\btitle{A Topological Approach to Selecting Models of Biological Experiments}.
\bjournal{PLOS ONE}
\bvolume{14}
\bpages{1-18}.
\bdoi{10.1371/journal.pone.0213679}
\end{barticle}
\endbibitem

\bibitem{Wasserman2018}
\begin{barticle}[author]
\bauthor{\bsnm{Wasserman},~\bfnm{Larry}\binits{L.}}
(\byear{2018}).
\btitle{Topological Data Analysis}.
\bjournal{Annu. Rev. Stat. Appl.}
\bvolume{5}
\bpages{501--535}.
\bdoi{10.1146/annurev-statistics-031017-100045}
\bmrnumber{3774757}
\end{barticle}
\endbibitem

\bibitem{Xia2014}
\begin{barticle}[author]
\bauthor{\bsnm{Xia},~\bfnm{Kelin}\binits{K.}},
  \bauthor{\bsnm{Feng},~\bfnm{Xin}\binits{X.}},
  \bauthor{\bsnm{Tong},~\bfnm{Yiying}\binits{Y.}} \AND
  \bauthor{\bsnm{Wei},~\bfnm{Guo~Wei}\binits{G.~W.}}
(\byear{2014}).
\btitle{Persistent Homology for the Quantitative Prediction of Fullerene
  Stability}.
\bjournal{Journal of Computational Chemistry}
\bvolume{36}
\bpages{408--422}.
\bdoi{10.1002/jcc.23816}
\end{barticle}
\endbibitem

\bibitem{Yogeshwaran2015}
\begin{barticle}[author]
\bauthor{\bsnm{Yogeshwaran},~\bfnm{D.}\binits{D.}} \AND
  \bauthor{\bsnm{Adler},~\bfnm{Robert~J.}\binits{R.~J.}}
(\byear{2015}).
\btitle{On the Topology of Random Complexes Built Over Stationary Point
  Processes}.
\bjournal{Ann. Appl. Probab.}
\bvolume{25}
\bpages{3338--3380}.
\bdoi{10.1214/14-AAP1075}
\bmrnumber{3404638}
\end{barticle}
\endbibitem

\bibitem{Yogeshwaran2017}
\begin{barticle}[author]
\bauthor{\bsnm{Yogeshwaran},~\bfnm{D.}\binits{D.}},
  \bauthor{\bsnm{Subag},~\bfnm{Eliran}\binits{E.}} \AND
  \bauthor{\bsnm{Adler},~\bfnm{Robert~J.}\binits{R.~J.}}
(\byear{2017}).
\btitle{Random Geometric Complexes in the Thermodynamic Regime}.
\bjournal{Probab. Theory Related Fields}
\bvolume{167}
\bpages{107--142}.
\bdoi{10.1007/s00440-015-0678-9}
\bmrnumber{3602843}
\end{barticle}
\endbibitem

\bibitem{Zomorodian2005}
\begin{barticle}[author]
\bauthor{\bsnm{Zomorodian},~\bfnm{Afra}\binits{A.}} \AND
  \bauthor{\bsnm{Carlsson},~\bfnm{Gunnar}\binits{G.}}
(\byear{2005}).
\btitle{Computing Persistent Homology}.
\bjournal{Discrete Comput. Geom.}
\bvolume{33}
\bpages{249--274}.
\bdoi{10.1007/s00454-004-1146-y}
\bmrnumber{2121296}
\end{barticle}
\endbibitem

\end{thebibliography}
\end{document}